%% file: CCpaper.tex
\documentclass[10pt]{article}

\usepackage[margin=3cm]{geometry}

\usepackage{graphicx} 
\usepackage{booktabs} 
\usepackage{array} 
\usepackage{verbatim} 
\usepackage{subfig} 
\usepackage{amsmath}
\usepackage{amsfonts}
\usepackage{amssymb}
\usepackage{amsthm}
\usepackage{enumerate}
\usepackage{hyperref}
\usepackage{graphicx}
\usepackage[version=0.96]{pgf}
\usepackage{tikz}
\usetikzlibrary{arrows,shapes,automata,backgrounds,petri,positioning}

\usepackage{todonotes}

\theoremstyle{plain}
\newtheorem{theorem}{Theorem}[section]
\newtheorem{lemma}[theorem]{Lemma}
\newtheorem{proposition}[theorem]{Proposition}
\newtheorem{corollary}[theorem]{Corollary}

\theoremstyle{definition}
\newtheorem{definition}[theorem]{Definition}
\newtheorem{example}[theorem]{Example}

\theoremstyle{remark}
\newtheorem{remark}[theorem]{Remark}
\allowdisplaybreaks
\numberwithin{equation}{section}

\newcommand{\R}{\mathbb{R}}
\newcommand{\C}{\mathbb{C}}
\newcommand{\diag}{\mathrm{diag}}
\renewcommand{\d}{\mathrm{d}}

\newcommand{\eps}{\epsilon}
\newcommand{\imag}{\mathrm{Im}\,}

\newcommand{\supp}{\mathrm{supp}}

\newcommand{\sddots}{\smash{\raisebox{-4pt}{\footnotesize{$\ddots$}}}}

\DeclareMathOperator*{\muesssup}{\mu-ess\,sup}

\numberwithin{equation}{section}

\begin{document}
  
  \title{Spectra of Jacobi operators via connection coefficient matrices}
  \author{Marcus Webb \thanks{Department of Mathematics, University of Manchester, Manchester, UK.
      (\texttt{marcus.webb@manchester.ac.uk}, \texttt{https://personalpages.manchester.ac.uk/staff/marcus.webb/})} \and Sheehan Olver\thanks{Department of Mathematics, Imperial College, London, UK.
      (\texttt{s.olver@imperial.ac.uk}, \texttt{http://wwwf.imperial.ac.uk/\~solver/})}}

  \date{\today}

 \maketitle 
\begin{abstract}
  We address the computational spectral theory of Jacobi operators that are compact perturbations of the free Jacobi operator via the asymptotic properties of a connection coefficient matrix. In particular, for Jacobi operators that are finite-rank perturbations  we show that the computation of the spectrum can be reduced to a polynomial root finding problem, from a polynomial that is derived explicitly from the entries of a connection coefficient matrix. A formula for the spectral measure of the operator is also derived explicitly from these entries. The analysis is extended to trace-class perturbations. We address issues of computability in the framework of the Solvability Complexity Index, proving that the spectrum of compact perturbations of the free Jacobi operator is computable in finite time with guaranteed error control in the Hausdorff metric on sets. 
\end{abstract}

  
  \noindent \textbf{AMS classification numbers} \, Primary 47B36, Secondary 47A10, 47A75, 65J10
    
 \input{intro}

 \section*{Acknowledgements}
The authors thank the following people for helpful discussions on this work as it developed: Anders Hansen, Arieh Iserles, Andrew Swan, Peter Clarkson, Alfredo Dea\~no, Alex Townsend, Walter Van Assche, Bernhard Beckermann, Matthew Colbrook, Giles Shaw and David Sanders, amongst others. We are also grateful to the anonymous referees for their comments which greatly improved the quality of the paper. The first author was supported by the London Mathematical Society Cecil King Travel Scholarship to visit the University of Sydney in 2016, the UK Engineering and Physical Sciences Research Council (EPSRC) grant EP/H023348/1 for the University of Cambridge Centre for Doctoral Training, the Cambridge Centre for Analysis, by the FWO research project G.A004.14 at KU Leuven, Belgium, and by an FWO Postdoctoral Fellowship. The second author was supported by a Leverhulme Trust Research Project Grant. Portions of this manuscript were submitted in an earlier form as part of the PhD thesis of first author at the University of Cambridge in 2017.
 
 \input{spectraltheory}

 \input{connectioncoeffs}
 \input{toeplitzplusfinite}
 \input{toeplitzplustraceclass}

 \input{computability}
 \input{conclusions}

\input{appendix}
\newpage
\bibliographystyle{amsplain}
\bibliography{CCpaper}

\end{document}

%% file: intro.tex

  \section{Introduction}

  A Jacobi operator is a selfadjoint operator on $\ell^2 = \ell^2(\{0,1,2,\ldots\})$, which with respect to the standard orthonormal basis $\{e_0, e_1,e_2,
  \ldots \}$ has a tridiagonal matrix representation,
  \begin{equation}
  J = \left(\begin{array}{cccc}
  \alpha_0 & \beta_0  &          &  \\
  \beta_0  & \alpha_1 & \beta_1  &  \\
  & \beta_1  & \alpha_2 & \sddots \\
  &          &  \sddots  & \sddots
  \end{array} \right),
  \end{equation}
  where $\alpha_k$ and $\beta_k$ are real numbers with $\beta_k > 0$. In the special case where $\{\beta_k\}_{k=0}^\infty$ is a constant sequence, this operator can be interpreted as a discrete Schr\"odinger operator on the half line, and its spectral theory is arguably its most important aspect.
  
  The spectral theorem for Jacobi operators guarantees the existence of a probability measure $\mu$ supported on the spectrum $\sigma(J) \subset \R$, called the spectral measure, and a unitary operator $U:\ell^2 \to L^2_\mu(\R)$ such that
  \begin{equation}
  U J U^* [f](s) = sf(s),
  \end{equation}
  for all $f \in L^2_\mu(\R)$ \cite{deift2000orthogonal}. The coefficients $\{\alpha_k\}_{k\geq 0}$ and $\{\beta_k\}_{k\geq 0}$ are the three--term recurrence coefficients of the orthonormal polynomials $\{P_k\}_{k\geq0}$ with respect to $\mu$, which are given by $P_k = Ue_k$.
  
  Suppose we have a second Jacobi operator, $D$, for which the spectral theory is known analytically. The point of this paper is to show that for certain classes of Jacobi operators $J$ and $D$, the computation and theoretical study of the spectrum and spectral measure of $J$ can be conducted effectively using the {\it connection coefficient matrix} between $J$ and $D$, combined with known properties of $D$.
  
   \begin{definition}\label{def:connection}
    The {\it connection coefficient matrix} $C=C_{J\to D} = (c_{ij})_{i,j =
      0}^\infty$ is the upper triangular matrix representing the change of basis
    between the orthonormal polynomials $(P_k)_{k=0}^\infty$ of $J$, and the orthonormal polynomials $(Q_k)_{k=0}^\infty$ of $D$, whose entries satisfy,
    \begin{equation}
    P_k(s) = c_{0k}Q_0(s) + c_{1k}Q_1(s) + \cdots + c_{kk} Q_k(s).
    \end{equation}
  \end{definition} 
 
We pay particular attention to the case where $D$ is the so-called \emph{free Jacobi operator},
 \begin{equation}
 \Delta = \left(\begin{array}{cccc}
 0    & \frac12 &         &  \\
 \frac12 &    0    & \frac12 &  \\
 & \frac12 &    0    & \sddots \\
 &         & \sddots  & \sddots
 \end{array} \right),
 \end{equation}
 and $J$ is a Jacobi operator of the form $J = \Delta + K$, where $K$ is compact. In the discrete Schr\"odinger operator setting, $K$ is a diagonal potential function which decays to zero at infinity \cite{simon1979trace}. Another reason this class of operators is well studied is because the Jacobi operators for the classical Jacobi polynomials are of this form \cite{olverDLMF}. Since scaling and shifting by the identity operator affects the spectrum in a trivial way, results about these Jacobi operators $J$ apply to all Jacobi operators which are compact perturbations of a Toeplitz operator  ({\it Toeplitz-plus-compact}).
 
 The spectral theory of Toeplitz operators such as $\Delta$ is known explicitly \cite{bottcher2013analysis}. The spectral measure of $\Delta$ is the semi-circle $\d\mu_\Delta(s) = \frac{2}{\pi}(1-s^2)^{\frac12}$ (restricted to $[-1,1]$), and the orthonormal polynomials are the Chebyshev polynomials of the second kind, $U_k(s)$. We prove the following new results.
 
 If $J$ is a Jacobi operator which is a  \emph{finite rank perturbation} of $\Delta$ ({\it Toeplitz-plus-finite-rank}), i.e. there exists $n$ such that
 \begin{equation}
 \alpha_k = 0, \quad \beta_{k-1} = \frac12 \text{ for all } k \geq n,
 \end{equation}
 \begin{itemize}
   \item Theorem \ref{thm:connectioncoeffs}: The connection coefficient matrix $C_{J\to\Delta}$ can be decomposed into $C_{\rm Toe} + C_{\rm fin}$ where $C_{\rm Toe}$ is Toeplitz, upper triangular and has bandwidth $2n-1$, and the entries of $C_{\rm fin}$ are zero outside the $n-1 \times 2n-1$ principal submatrix.
   \item Theorem \ref{thm:joukowskimeasure}: Let $c$ be the Toeplitz symbol of $C_{\rm Toe}$. It is a degree $2n-1$ polynomial with $r \leq n$ roots inside the complex open unit disc $\mathbb{D}$, all of which are real and simple. The spectrum of $J$ is
   \begin{equation}\label{eqn:introsigmaJ}
   \sigma(J) = [-1,1] \cup \left\{\lambda(z_k) : c(z_k) = 0, z_k \in \mathbb{D}\right\},
   \end{equation}
   where $\lambda(z) := \frac12(z+z^{-1}) : \mathbb{D} \to \C\setminus[-1,1]$ is the Joukowski transformation. The spectral measure of $J$ is given by the formula
   \begin{equation}\label{eqn:intromu}
   \d\mu(s) = \frac{\d\mu_\Delta(s)}{|c(e^{i\theta})|^2} + \sum_{k=1}^r 
   \frac{(z_k-z_k^{-1})^2}{z_kc'(z_k)c(z_k^{-1})}\d\delta_{\lambda(z_k)}(s),
   \end{equation}
    where $\cos(\theta) = s$. The denominator in the first term can be expressed using the polynomial, $|c(e^{i\theta})|^2= \sum_{k=0}^{2n-1} \langle C^Te_k ,C^T e_0 \rangle U_k(s)$.
 \end{itemize}
 
 \begin{figure}[h!]
   \begin{center}
     \includegraphics[width=.3\textwidth]{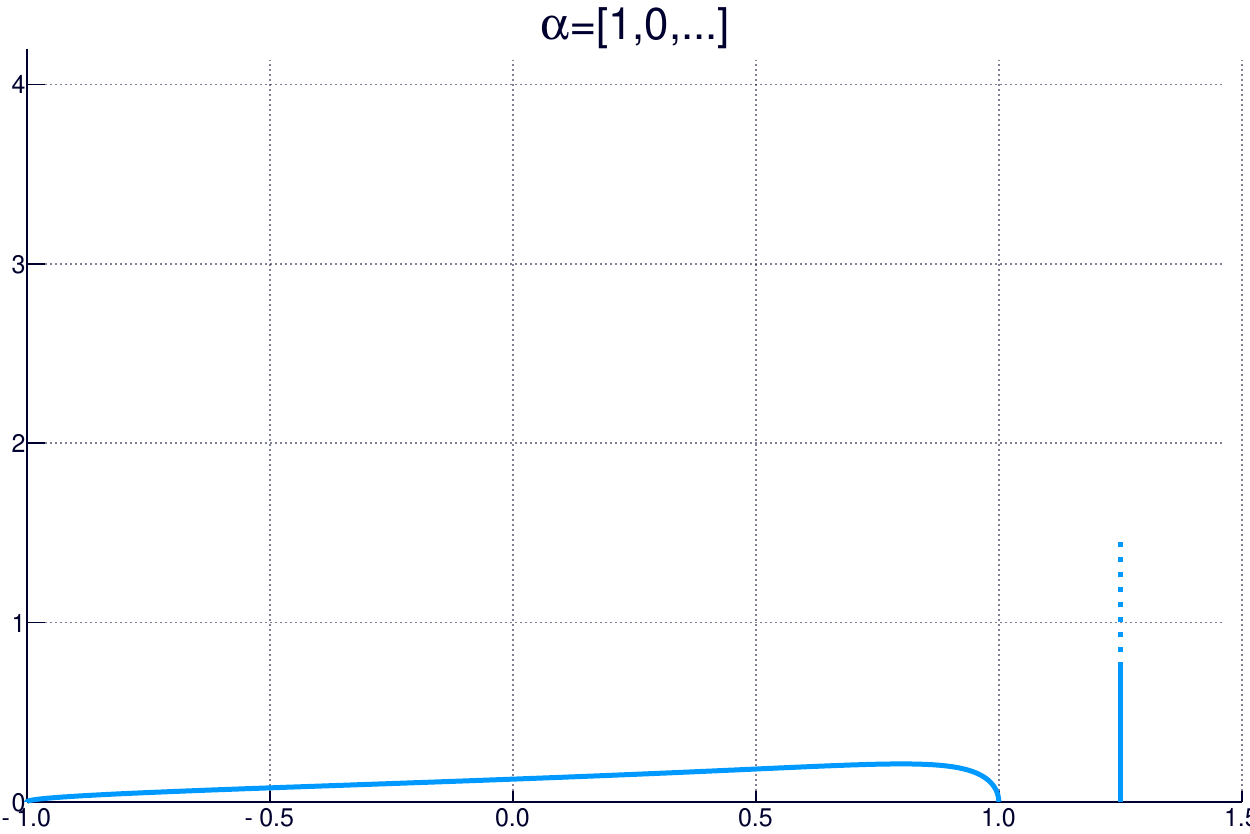}\hskip8pt\includegraphics[width=.3\textwidth]{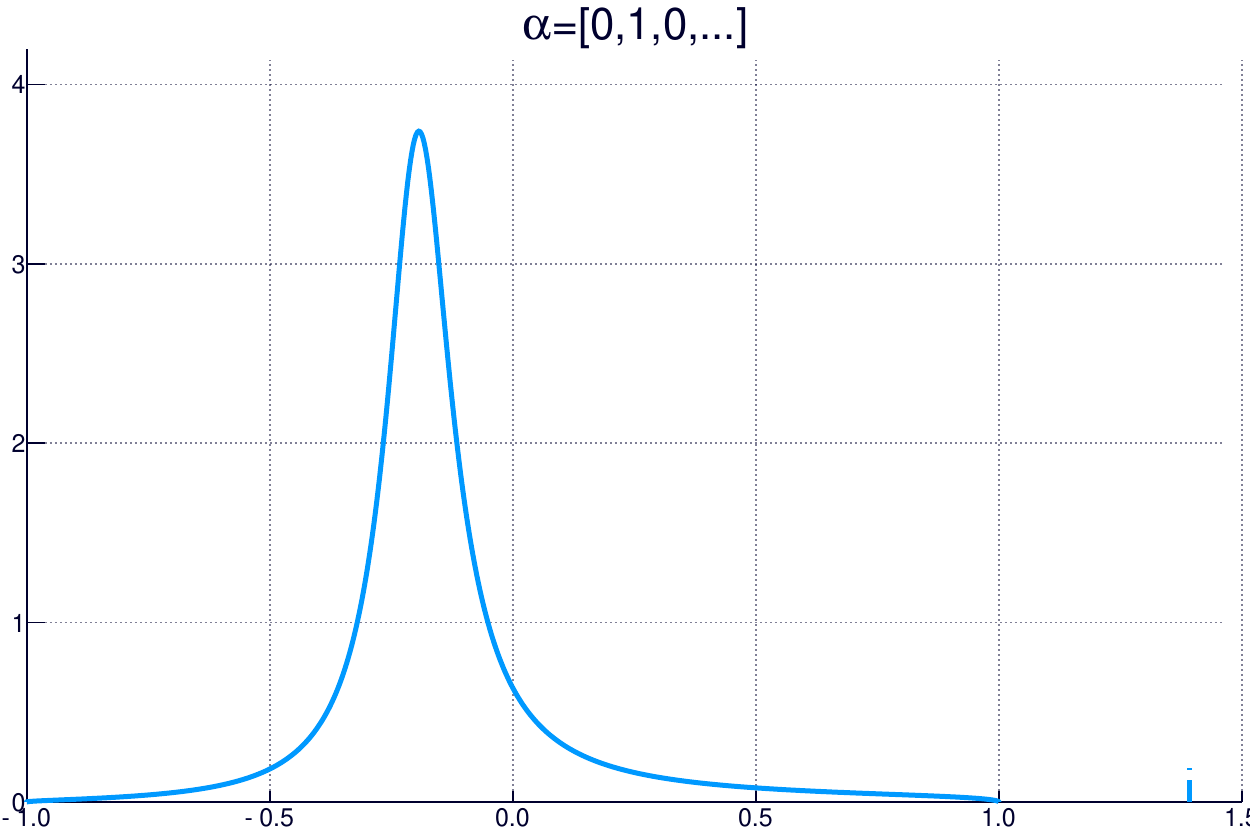}\hskip8pt\includegraphics[width=.3\textwidth]{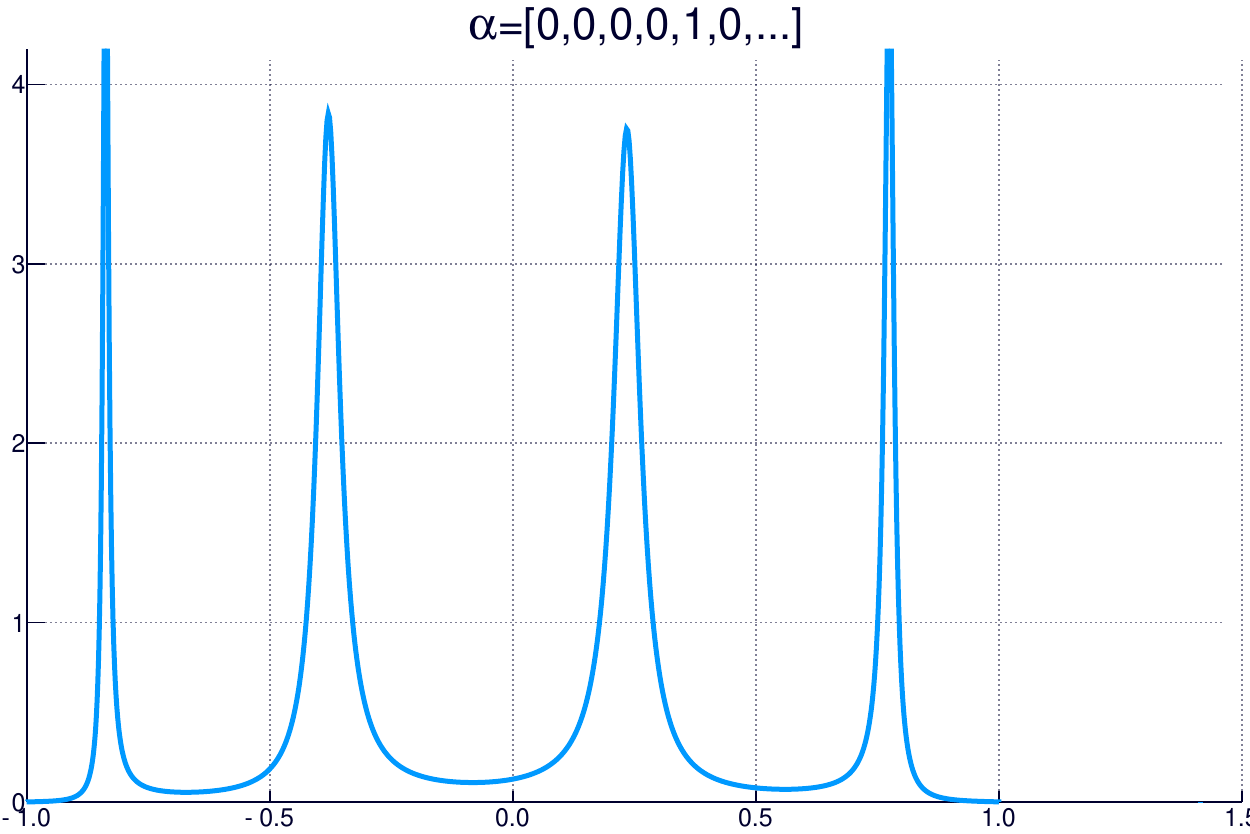}
     \caption{These are the spectral measures of three different Jacobi operators; each differs from $\Delta$ in only one entry. The left plot is of the spectral measure of the Jacobi operator which is $\Delta$ except the $(0,0)$ entry is 1, the middle plot is that except the $(1,1)$ entry is 1, and the right plot is that except the $(4,4)$ entry is 1. This can be interpreted as a discrete Schr\"odinger operator with a single Dirac potential at different points along $[0,\infty)$. The continuous parts of the measures are given exactly by the computable formula in equation \eqref{eqn:intromu}, and each has a single Dirac delta corresponding to discrete spectrum (the weight of the delta gets progressively smaller in each plot), the location of which can be computed with guaranteed error using interval arithmetic (see Appendix \ref{sec:numericalexperiments})}\label{fig:delta}.
   \end{center}
 \end{figure}

For $R  > 0 $, define the Banach space $\ell^1_R$ to be scalar sequences such that $\sum_{k=0}^\infty |v_k| R^k < \infty$. If $J$ is a \emph{trace class perturbation} of $\Delta$ ({\it Toeplitz-plus-trace-class}), i.e.,
 \begin{equation}
 \sum_{k=0}^\infty \left|\alpha_k\right| + \left|\beta_k-\frac12\right| < \infty,
 \end{equation}
 \begin{itemize}
   \item Theorem \ref{thm:traceclassC}: $C = C_{J\to\Delta}$ is bounded as an operator from $\ell^1_{R}$ into itself, for all $R > 1$. Further, we have the decomposition $C = C_{\rm Toe} + C_K$ where $C_{\rm Toe}$ is upper triangular Toeplitz and $C_K$ is compact as an operator from $\ell^1_{R}$ into itself for all $R > 1$.
   \item Theorem \ref{thm:symbolsconverge} and Theorem \ref{thm:traceclasszresolvent} : The Toeplitz symbol of $C_{\rm Toe}$, $c$,  is analytic in the complex unit disc. The discrete eigenvalues, as in the Toeplitz-plus-finite-rank case are of the form $\frac12(z_k+z_k^{-1})$ where $z_k$ are the roots of $c$ in the open unit disc.
 \end{itemize}
 
 The relevance of the space $\ell^1_R$ here is that for an upper triangular Toeplitz matrix which is bounded as an operator from $\ell^1_R$ to itself for all $R>1$, the symbol of that operator is analytic in the open unit disc.
 
  Following the pioneering work of Ben-Artzi--Colbrook--Hansen--Nevanlinna--Seidel on the Solvability Complexity Index \cite{ben2015can,ben2015new,hansen2011solvability,colbrook2019foundations,colbrook2019foundations2}, we prove two theorems about computability. We assume real number arithmetic, and the results do not necessarily apply to algorithms using floating point arithmetic.
 \begin{itemize}
   \item Theorem \ref{thm:computToeplusfinite}: If $J$ is a Toeplitz-plus-finite-rank Jacobi operator, then in a finite number of operations, the absolutely continuous part of the spectral measure is computable exactly, and the locations and weights of the discrete part of the spectral measure are computable to any desired accuracy. If the rank of the perturbation is known {\it a priori}  then the algorithm can be designed to terminate with guaranteed error control.
   \item Theorem \ref{thm:computeTpluscomp}: If $J=\Delta+K$ is a Toeplitz-plus-compact Jacobi operator, then in a finite number of operations, the spectrum of $J$ is computable to any desired accuracy in the Hausdorff metric on subsets of $\R$. If the quantity $\sup_{k \geq m} |\alpha_k| + \sup_{k\geq m} |\beta_k-\frac12|$ can be estimated for all $m$, then the algorithm can be designed to terminate with guaranteed error control.
 \end{itemize} 
 
 The present authors consider these results to be the beginning of a long term project on the computation of spectra of structured operators. Directions for future research are outlined in Section \ref{sec:conclusions}.
 
  \subsection{Relation to existing work on connection coefficients}
 
 Recall that $J$ and $D$ are Jacobi operators with orthonormal polynomials $\{P_k\}_{k=0}^\infty$ and $\{Q_k\}_{k=0}^\infty$ respectively, and spectral measures $\mu$ and $\nu$ respectively. Uvarov gave expressions for the relation between $\{P_k\}_{k=0}^\infty$ and $\{Q_k\}_{k=0}^\infty$ in the case where the Radon--Nikodym derivative $\frac{\d\nu}{\d\mu}$ is a rational function, utilising connection coefficients \cite{uvarov1959relation,uvarov1969connection}. Decades later, Kautsky and Golub related properties of $J$, $D$ and the connection coefficients matrix $C = C_{J\to D}$ with the Radon-Nikodym derivative $\frac{\d\nu}{\d\mu}$, with applications to practical computation of Gauss quadrature rules in mind \cite{kautsky1983calculation}. In the setting of Gauss quadrature rules, the connection coefficients are usually known as modified or mixed moments, (see \cite{gautschi1970construction,sack1971algorithm,wheeler1974modified}). The following results, which we prove in Section 3 for completeness, are straightforward generalisations of what can be found in the papers cited in this paragraph from the 1960's, 1970's and 1980's.
 
The Jacobi operators and the connection coefficients satisfy
 \begin{equation}\label{eqn:JCDC}
 CJ = DC,
 \end{equation} 
 which makes sense as operators acting on finitely supported sequences (this is made clear and proved in Theorem \ref{thm:Ccommute}). A finite-dimensional version of this result with a rank-one remainder term first appears in \cite[Lem.~1]{kautsky1983calculation}. The connection coefficients matrix also determines the existence and certain properties of the Radon--Nikodym derivative $\frac{\d\nu}{\d\mu}$:
 \begin{itemize}
   \item Proposition \ref{prop:Cmeasuremultplier}: $\frac{\d\nu}{\d\mu} \in L^2_\mu(\R)$ if and only if the first row of $C$ is an $\ell^2$ sequence, in which case
   \begin{equation}\label{ref:modifiedmoments}
   \frac{\mathrm{d}\nu}{\mathrm{d}\mu} = \sum_{k=0}^\infty c_{0,k} P_k.
   \end{equation}
     \item Proposition \ref{prop:Cbounded}: If $\frac{\mathrm{d}\nu}{\mathrm{d}\mu} \in L^\infty_\mu(\R)$ then $C$ is a bounded operator on $\ell^2$ and
     \begin{equation}
     \|C\|_2^2 = \muesssup_{s \in \sigma(J)} \left|\frac{\mathrm{d}\nu}{\mathrm{d}\mu}(s)\right|.
     \end{equation}
     \item Corollary \ref{cor:Cbddinvert}: If both $\frac{\mathrm{d}\nu}{\mathrm{d}\mu} \in L^\infty_\mu(\R)$ and $\frac{\mathrm{d}\mu}{\mathrm{d}\nu} \in L^\infty_\nu(\R)$ then $C$ is bounded and invertible on $\ell^2$.
  \end{itemize}
 
\subsection{Relation to existing work on spectra of Jacobi operators}

There has been extensive work in the last 20 years or so on the spectral theory of Jacobi operators which are compact perturbations of the free Jacobi operator, particularly with applications to quantum theory and random matrix theory in mind \cite{killip2003sum,damanik2006jost1,damanik2006jost2,simon2010szegHo,dubbs2015infinite,gamboa2016sum}. The literature focuses on so-called sum rules, which are certain remarkable relationships between the spectral measures and the entries of Jacobi operators, and builds upon some 20th century developments in the Szeg\H{o} theory of orthogonal polynomials \cite{szegHo1939orthogonal,geronimus1977orthogonal,geronimo1980scattering,dombrowski1986orthogonal,van1989asymptotics,van1990asymptotics,van1991orthogonal,nevai1992compact,van1994chebyshev}.

For a Jacobi operator $J$, which is a compact perturbation of the free Jacobi operator $\Delta$, with orthogonal polynomials $\{P_k\}_{k=0}^\infty$ and spectral measure $\d\mu$, the central analytical object of interest is the function
\begin{equation}\label{eqn:defineJost}
u_0(z) = \lim_{k\to\infty} (1-z^2) z^k P_k\left(\frac12\left(z+z^{-1}\right) \right), \qquad z \in \mathbb{D}.
\end{equation}
Conditions on the perturbation $J-\Delta$ yield a nicer $u_0$. For instance, when $J-\Delta$ is trace class, $u_0$ is analytic and its roots in the unit disc are in one-to-one correspondence with the eigenvalues of $J$, and when $J - \Delta$ is of finite rank, $u_0$ is a polynomial \cite{killip2003sum}. 

The function $u_0$ defined as in \eqref{eqn:defineJost} can be arrived at through several different definitions: the Jost functions, the perturbation determinant, the Szeg\H{o} function and the Geronimo--Case functions (see the introduction of \cite{killip2003sum}). One contribution of the present paper is that it provides a new interpretation of $u_0$: the Toeplitz symbol $c(z)$. Furthermore, this new interpretation also has a matrix associated to it, the connection coefficients matrix $C$, which is defined for \emph{any} two Jacobi operators, not only perturbations of the free Jacobi operator. This opens the door to generalisation, something the present authors intend to pursue in the future.

Very recently, Colbrook, Bogdan, and Hansen have introduced techniques for computing spectra with error control \cite{colbrook2019compute} which work on quite broad classes of operators.  Colbrook \cite{colbrook2019computing} extended these techniques to computing spectral measures of operators, including Jacobi operators as a special case. While this work applies to broader classes of operators than ours, it gives less accurate results for the class of Jacobi operators we consider, and in particular does not produce explicit formulae such as Theorem~\ref{thm:joukowskimeasure} or as precise results as Theorem~\ref{thm:computeTpluscomp}. In particular, our assumptions on the structure of the operator  are sufficient to produce better results in terms the SCI  hierarchy: we can compute the discrete  spectra and spectral measures of trace-class perturbations or compact perturbations with known decay  with error control ($\Delta_1$ in the notation of \cite{colbrook2019computing}),  the spectral measure  of Jacobi operators that are finite rank perturbations of $\Delta$ in finite operations ($\Delta_0$), and the absolutely continuous spectrum is always $[-1,1]$.   This  compares favourably to \cite{colbrook2019computing} which proves $\Delta_2$ classification results (one limit with no error control) for the spectral measures, projections, functional calculus and Radon-Nikodym derivatives of a larger class of operators. Furthermore, the generality of \cite{colbrook2019computing} means that the location of the absolutely continuous and pure point spectra are no longer known, as is the case for our class of operators. This causes the computation of spectral decompositions to become very difficult and higher up in the SCI hierarchy.

\subsection{Outline of the paper}

\begin{itemize}
\item In Section \ref{sec:spectraltheory} we outline basic, established results about spectral theory of Jacobi operators.
\item In Section \ref{sec:connectioncoeffs} we discuss the basic properties of the connection coefficients matrix $C_{J \to D}$ for general Jacobi operators $J$ and $D$, and how they relate to the spectra of $J$ and $D$.
\item In Section \ref{sec:Toeplitzplusfinite} we show how connection coefficient matrices apply to Toeplitz-plus-finite-rank Jacobi operators, and in Section \ref{sec:toeplitzplustraceclass} we extend these results to the Toeplitz-plus-trace-class case.
\item Section \ref{sec:computability} is devoted to issues of computability.
\item Appendix \ref{sec:numericalexperiments} gives an array of numerical examples produced using an open source Julia package SpectralMeasures.jl \cite{SpectralMeasuresjl} that  implements the ideas of this paper. It makes extensive use of the  open source Julia package  {ApproxFun.jl} \cite{olverapproxfun,olver2014practical}, in particular the features for defining and manipulating functions and infinite-dimensional operators.  
\end{itemize}

%% file: spectraltheory.tex

\section{Spectral theory of Jacobi operators}\label{sec:spectraltheory}
  
  In this section we present well known results about the spectra of Jacobi operators. This gives a self-contained account of what is required to prove the results later in the paper, and sets the notation.
%
%
  
  \begin{definition}\label{def:resolvent}
    Define the principal resolvent function for $\lambda \in \C \setminus \sigma(J)$,
    \begin{equation}
      G(\lambda) = \langle e_0, (J-\lambda)^{-1} e_0\rangle
    \end{equation}
    where
    $$
    \langle x,y \rangle := \sum_{k=0}^\infty \bar x_k  y_k.
    $$
  \end{definition}
  \begin{theorem}[\cite{deift2000orthogonal,teschl2000jacobi,simon1979trace}]\label{thm:G}
  Let $J$ be a bounded Jacobi operator.
    \begin{enumerate}[(i)]
      \item There exists a unique compactly supported probability measure $\mu$ 
on $\R$, called the spectral measure of $J$, such that
      \begin{equation}
	G(\lambda) = \int (s-\lambda)^{-1} \,\d\mu(s).
      \end{equation}
      \item For any $s_1 < s_2$ in $\R$,
      \begin{equation}
	\frac12 \mu(\{s_1\}) + \mu((s_1,s_2))+\frac12\mu(\{s_2\}) = \lim_{\eps \searrow 0} 
\frac{1}{\pi} \int_{s_1}^{s_2} \imag G(s+i\eps) \, \d s.
      \end{equation}
      \item The spectrum of $J$ is
      \begin{equation}
       \sigma(J) = \mathrm{supp}(\mu) = \overline{\{ s \in \R : \liminf_{\eps 
\searrow 0} \imag G(s+i\eps) > 0  \}}.
      \end{equation}
       
       The point spectrum $\sigma_p(J)$ of $J$ is the set of points $s \in \R$ such that the limit
        \begin{equation}\label{eqn:pointspec}
         \mu(\{s\}) = \lim_{\eps \searrow 0} \frac{\eps}{i} G(s+i\eps)
        \end{equation}
        exists and is positive.
       
       The continuous spectrum of $J$ is the set of points $s \in \R$ such that $\mu(\{s\}) = 0$ but
       \begin{equation}
        \liminf_{\eps \searrow 0} \imag G(s + i\eps) > 0.
       \end{equation}
    \end{enumerate}
  \end{theorem}
   
   
The measure $\mu$ is the spectral measure that appears in the spectral theorem for self-adjoint operators on Hilbert space \cite{dunford1971linear}, as demonstrated by the following theorem \cite{deift2000orthogonal,teschl2000jacobi,simon1979trace}.
 
  \begin{definition}\label{orthpolys} 
    The orthonormal polynomials for $J$ are $P_0, P_1, P_2, \ldots$ defined by 
the three term recurrence
    \begin{align}
     sP_k(s) &= \beta_{k-1} P_{k-1}(s) + \alpha_k P_k(s) + \beta_k P_{k+1}(s), 
\\
     P_{-1}(s) &= 0, \quad P_0(s) = 1.
    \end{align}
  \end{definition}

  \begin{theorem}[\cite{deift2000orthogonal}]\label{thm:maindeift}
  Let $J$ be a bounded Jacobi operator and let $P_0, P_1, P_2, \ldots$ be as defined in Definition \ref{orthpolys}. Then 
we have the following.
    \begin{enumerate}[(i)]
     \item The polynomials are such that $P_k(J)e_0 = e_k$.
     \item The polynomials are orthonormal with respect to the spectral measure 
of $J$,
      \begin{equation}
	\int P_j(s)P_k(s)\,\d\mu(s) = \delta_{jk}.
      \end{equation}
      
      \item Define the unitary operator $U : \ell^2 \to L^2_\mu(\R)$ such that 
$Ue_k = P_k$. Then for all $f \in L^2_\mu(\R)$,
      \begin{equation}
	UJU^*[f](s) = sf(s).
      \end{equation}
      \item For all polynomials $f$, the entries of $f(J)$ are equal to,
     \begin{equation}
      \langle e_i, f(J) e_j\rangle = \int f(s) P_i(s)P_j(s) \,\d\mu(s).
     \end{equation}
     For $f \in L^1_\mu(\R)$, this formula defines the matrix $f(J)$.
    \end{enumerate}
  \end{theorem}
  
  The following definition is standard in orthogonal polynomial theory.
  
  \begin{definition}[\cite{gautschi2004orthogonal,van1991orthogonal}]\label{associatedpolys}
   The first associated polynomials for $J$ are 
$P^{\mu}_0,P^{\mu}_1,P^{\mu}_2,\ldots$ defined by the three term recurrence
   \begin{align}
     \lambda P^{\mu}_k(\lambda) &= \beta_{k-1} P^{\mu}_{k-1}(\lambda) + 
\alpha_k P^{\mu}_k(\lambda) + \beta_k P^{\mu}_{k+1}(\lambda), \\
     P^{\mu}_0(\lambda) &= 0, \quad P^{\mu}_1(\lambda) = \beta_0^{-1}.
   \end{align}
  \end{definition}
 
 The relevance of the first associated polynomials for this work is the following integral formula.
 
 \begin{lemma}(\cite[pp.~17,18]{gautschi2004orthogonal})\label{lem:associatedpolys}
   The first associated polynomials are given by the integral formula
   \begin{equation}
   P^{\mu}_k(\lambda) = \int \frac{P_k(s) - P_k(\lambda)}{s-\lambda} 
   \,\d\mu(s), \quad \lambda \in \C\setminus\sigma(J).
   \end{equation}
 \end{lemma}
 
 For notational convenience we also define the $\mu$-derivative of a general 
polynomial.
 
 \begin{definition}\label{def:muderiv}
 Let $\mu$ be a probability measure compactly supported on the real line and let $f$ 
be a polynomial. The $\mu$-derivative of $f$ is the polynomial defined by
  \begin{equation}
   f^{\mu}(\lambda) = \int \frac{f(s)-f(\lambda)}{s-\lambda} \,\d\mu(s).
  \end{equation}
 \end{definition}

%% file: connectioncoeffs.tex
\section{Connection coefficient matrices}\label{sec:connectioncoeffs}
  
  In this section we give preliminary results to indicate the relevance of connection coefficient matrices to spectral theory of Jacobi operators.
  
  \subsection{Basic properties}
  
  As in the introduction, consider a second bounded Jacobi operator,
  \begin{equation*}
    D = \left(\begin{array}{cccc}
      \gamma_0 & \delta_0 & & \\
      \delta_0 & \gamma_1 & \delta_1 & \\
      & \delta_1 & \gamma_2 & \sddots \\
      & & \sddots & \sddots
    \end{array} \right),
  \end{equation*}
  with principal resolvent function $H(z)$, spectral measure $\nu$ and 
orthogonal polynomials denoted $Q_0, Q_1, Q_2, \ldots$. In the introduction (Definition \ref{def:connection}) we defined the connection coefficient matrix between $J$ and $D$, $C = C_{J\to D}$ to have entries satisfying
  \begin{equation}
  P_k(s) = c_{0k}Q_0(s) + c_{1k}Q_1(s) + \cdots + c_{kk} Q_k(s).
  \end{equation}
  
  \begin{definition}
    Denote the space of complex-valued sequences with finitely many nonzero elements by $\ell_{\mathcal{F}}$, and its algebraic dual, the space of all complex-valued sequences, by $\ell_{\mathcal{F}}^\star$.
  \end{definition}
  
  Note that $C : \ell_{\mathcal{F}} \rightarrow \ell_{\mathcal{F}}$, because it is upper triangular, and $C^T : \ell_{\mathcal{F}}^\star \rightarrow \ell_{\mathcal{F}}^\star$, because it is lower triangular, and thus we may write
  \begin{equation*}
  \left(\begin{array}{c} P_0(s) \\ P_1(s) \\ P_2(s) \\ \vdots \end{array}
  \right) = C^T \left(\begin{array}{c} Q_0(s) \\ Q_1(s) \\ Q_2(s) \\ \vdots
  \end{array} \right) \text{ for all } s \in \C.
  \end{equation*}
  
   By orthonormality of the polynomial sequences the entries can also be interpreted as
  \begin{equation}\label{eqn:Centries}
    C_{J\to D} = \left(\begin{array}{cccc}
      \langle P_0, Q_0 \rangle_\nu &  \langle P_1,Q_0 \rangle_\nu & \langle P_2, Q_0 \rangle_\nu & \cdots \\
      0 & \langle P_1,Q_1 \rangle_\nu & \langle P_2,Q_1 \rangle_\nu & \cdots \\
      0 & 0  & \langle P_2,Q_2 \rangle_\nu & \cdots \\
      \vdots & \vdots & \vdots & \sddots
    \end{array}
    \right),
  \end{equation}
  where $\langle \cdot, \cdot \rangle_\nu$ is the standard inner product on $L^2_\nu(\R)$.
   


A recurrence relationship for the connection coefficients matrix was discovered by Sack and Donovan \cite{sack1971algorithm} and independently by Wheeler \cite{wheeler1974modified}, in the context of Gauss quadrature formulae.

 \begin{lemma}{\cite{sack1971algorithm,wheeler1974modified}}\label{lem:5ptsystem}
  The entries of the connection coefficients matrix $C_{J \to D}$ satisfy the following 5-point discrete system:
  \begin{equation*}
    -\delta_{i-1}c_{i-1,j} + \beta_{j-1}c_{i,j-1} + (\alpha_j-\gamma_i)c_{ij} + \beta_{j}c_{i,j+1} -\delta_{i}c_{i+1,j} = 0, \text{ for all } 0 \leq i < j,
    \end{equation*}
 with boundary conditions
   \begin{equation*}
   c_{ij} = \begin{cases}
             1 &\text{ if } i=j=0, \\
             0 &\text{ if } j = 0 \text{ and } i\neq 0, \\
             0 &\text{ if } j = -1 \text{ or } i = -1.
            \end{cases}
  \end{equation*}
 \end{lemma}
 \begin{proof}
  Assume by convention that $c_{ij} = 0$ if $i = -1$ or $j=-1$. Now using this boundary condition and the three term recurrences for the polynomial sequences, we see that
  \begin{align*}
    \langle Q_i(s),sP_j(s) \rangle_\nu &= \beta_{j-1} \langle 
Q_i,P_{j-1}\rangle_\nu + \alpha_j \langle Q_i,P_j \rangle_\nu + \beta_j \langle 
Q_i,P_{j+1} \rangle_\nu \\
    &= \beta_{j-1} c_{i,j-1} + \alpha_j c_{ij} + \beta_j c_{i,j+1},
  \end{align*}
  and
  \begin{align*}
    \langle sQ_i(s),P_j(s) \rangle_\nu &= \delta_{i-1} \langle 
Q_{i-1},P_j\rangle_\nu + \gamma_i \langle Q_i,P_j \rangle_\nu + \delta_i 
\langle Q_{i+1},P_j \rangle_\nu \\
    &= \delta_{i-1} c_{i-1,j} + \gamma_i c_{ij} + \delta_i c_{i+1,j}.
  \end{align*}
  Since $\langle sQ_i(s),P_j(s) \rangle_\nu = \langle Q_i(s),sP_j(s) 
\rangle_\nu$, we have the result for the interior points $0 \leq i < j$. 

The remaining boundary conditions come from $c_{i0} = \langle Q_i, P_0 \rangle_\nu$ which equals 1 if $i=0$ and 0 otherwise.
\end{proof}

 The 5-point recurrence formula can be restated as infinite-vector-valued three-term recurrence relations for rows and columns of $C$.
 
 \begin{corollary}\label{cor:rowsandcolsofC}
 The columns of $C$ satisfy
 \begin{align*}
  c_{*,0} &= e_0 \\
  Dc_{*,0} &= \alpha_0 c_{*,0} + \beta_0 c_{*,1} \\
  Dc_{*,j} &= \beta_{j-1}c_{*,j-1} + \alpha_j c_{*,j} + \beta_j c_{*,j+1}.
 \end{align*}
Consequently the $j$th column can be written $c_{*,j} = P_j(D)e_0$.

  The rows of $C$ satisfy
  \begin{align*}
   c_{0,*}J &= \gamma_0 c_{0,*} + \delta_0 c_{1,*}, \\
   c_{i,*}J &= \delta_{i-1} c_{i-1,*} + \gamma_i c_{i,*} + \delta_i 
c_{i+1,*}. 
  \end{align*}
Consequently, the $i$th row can be written $c_{i,*} = c_{0,*}Q_i(J)$. 
 \end{corollary}
 \begin{proof}
 The 5-point discrete system described in Lemma~\ref{lem:5ptsystem} can be used to find an explicit linear recurrence to compute the entries of $C$,
 \begin{align}
 c_{0,0} &= 1 \label{connectionrecurrence1}\\
 c_{0,1} &= (\gamma_0-\alpha_0)/\beta_0 \\
 c_{1,1} &= \delta_0/\beta_0 \\
 c_{0,j} &= \left((\gamma_0-\alpha_{j-1})c_{0,j-1} + \delta_0c_{1,j-1} - 
 \beta_{j-2} c_{0,j-2}\right)/\beta_{j-1} \\
 c_{i,j} &= \left(\delta_{i-1}c_{i-1,j-1} + (\gamma_i-\alpha_{j-1})c_{i,j-1} 
 + \delta_ic_{i+1,j-1} - \beta_{j-2}c_{i,j-2}\right)/\beta_{j-1}. 
 \label{connectionrecurrence2}
 \end{align}
 The recurrences of the rows and columns of $C$ are these written in vectorial form.
 
 The consequences follow from the uniqueness of solution to second order difference 
equations with two initial data (adding $c_{-1,*} = 0$ and $c_{*,-1} = 0$).
\end{proof}

 \subsection{Connection coefficients and spectral theory}
 
 The following theorems give precise results about how the connection coefficients matrix $C$ can be useful for studying and computing the spectra of Jacobi operators.
 
 \begin{theorem}[\cite{kautsky1983calculation}]\label{thm:Ccommute}
  Let $J$ and $D$ be bounded Jacobi operators and $C = C_{J \to D}$ the connection coefficients matrix. For all polynomials $p$, we have the following as operators from $\ell_{\mathcal{F}}$ to $\ell_{\mathcal{F}}$,
$$
 Cp(J) = p(D)C.
$$ 
\begin{remark}
  This is a generalisation of the result of Kausky and Golub \cite[Lem~1]{kautsky1983calculation}, that
  \begin{equation*}
  C_{N\times N}J_{N\times N} = D_{N\times N}C_{N\times N} + e_Nc_N^T,
  \end{equation*}
  where $C_{N\times N},J_{N\times N},D_{N\times N}$ are the principal $N\times N$ submatrices of $C,J,D$, and $c_N$ is a certain vector in $\R^N$.
\end{remark}
    \begin{proof}
     First we begin with the case $p(z) = z$. By definition,
\begin{align*}
 CJe_0 &= C(\alpha_0 e_0 + \beta_0 e_1) \\ 
       &= \alpha_0 Ce_0 + \beta_0 C e_1 \\
       &= \alpha_0 c_{*,0}+ \beta_0 c_{*,1}.
\end{align*}
Then by Corollary~\ref{cor:rowsandcolsofC}, this is equal to $Dc_{*,0}$, which is equal to $DCe_0$. Now, for any $j > 0$,
\begin{align*}
 CJe_j &= C(\beta_{j-1} e_{j-1} + \alpha_j e_j + \beta_j e_{j+1}) \\
       &= \beta_{j-1} c_{*,j-1} + \alpha_k c_{*,j} + \beta_j c_{*,j+1}.
\end{align*}
Then by Corollary~\ref{cor:rowsandcolsofC}, this is equal to $Dc_{*,j}$, which is equal to $DCe_j$. Hence $CJ = DC$.

Now, when $f(z) = z^k$ for any $k >0$, $D^kC = D^{k-1}CJ = \cdots = CJ^k$. By 
linearity $Cf(J) = f(D)C$ for all polynomials $f$. 
\end{proof}
\end{theorem}



We believe that the basic properties relating the Radon-Nikodym derivative $\frac{\mathrm{d}\nu}{\mathrm{d}\mu}$ and the connection coefficients matrix $C$ have not been given in the literature before, but follow naturally from discussions in, for example, \cite[Ch.~5]{golub2009matrices}, on mixed moments and modifications of weight functions for orthogonal polynomials.

 \begin{proposition}\label{prop:Cmeasuremultplier}
  Let $J$ and $D$ be bounded Jacobi operators with spectral measures $\mu$ and $\nu$ respectively, and connection coefficient matrix $C = C_{J\to D}$. Then
  \begin{align*}
  \frac{\d\nu}{\d\mu} \in L^2_\mu(\R) \text{ if and only if } c_{0,*} \in \ell^2,
  \end{align*}
  in which case
  \begin{equation}
\frac{\d\nu}{\d\mu} = \sum_{k=0}^\infty c_{0,k} P_k.
  \end{equation}
   \begin{proof}
     Suppose first that $\frac{\d\nu}{\d\mu} \in L^2_\mu(\R)$. Then $\frac{\d\nu}{\d\mu} = \sum_{k=0}^\infty a_k P_k$, for some $a\in\ell^2$, because $P_0,P_1,P_2,\ldots$ is an orthonormal basis of $L^2_\mu(\R)$. The coefficients are given by,
     \begin{align*}
     a_k &= \int P_k(s) \frac{\d\nu}{\d\mu}(s) \,\d\mu(s) \\
         &= \int P_k(s) \,\d\nu(s) \qquad \text{(definition of R--N derivative)}\\
         &= c_{0,k} \qquad \text{(equation \eqref{eqn:Centries}).}
     \end{align*}
     Hence $c_{0,*} \in \ell^2$ and gives the $P_k$ coefficients of $\frac{\d\nu}{\d\mu}$.
     
     Conversely, suppose that $c_{0,*} \in \ell^2$. Then the function $\sum_{k=0}^\infty c_{0,k}P_k$ is in $L^2_\mu(\R)$, and by the same manipulations as above its projections onto polynomial subspaces are equal to that of $\frac{\d\nu}{\d\mu}$.
  \end{proof}
  \end{proposition}
 \begin{remark}
 If we have a situation in which $c_{0,*} \in \ell^2$, we can by Proposition \ref{prop:Cmeasuremultplier} and the existence of the Radon--Nikodym derivative on $\supp(\mu)$ deduce that $\sigma(D) \subset \sigma(J)$ and the function defined by $\sum_{k=0}^\infty c_{0,k}P_k$ is zero on $\sigma(J) \setminus \sigma(D)$. This observation translates into a rootfinding problem in Section \ref{sec:Toeplitzplusfinite}.
 \end{remark}
 
 \begin{lemma}\label{lem:CTC}
   Let $J$ and $D$ be bounded Jacobi operators with spectral measures $\mu$ and $\nu$ respectively, and connection coefficient matrix $C = C_{J\to D}$. If $\nu$ is absolutely continuous with respect to $\mu$, then as operators mapping $\ell_{\mathcal{F}} \to \ell_{\mathcal{F}}^\star$,
   \begin{equation*}
   C^TC = \frac{\d\nu}{\d\mu}(J).
   \end{equation*}
   Here the matrix $\frac{\d\nu}{\d\mu}(J)$ is interpreted as in Theorem \ref{thm:maindeift} part (iii).
   \begin{proof}
     Note first that since $C:\ell_{\mathcal{F}} \to \ell_{\mathcal{F}}$ and $C^T : \ell_{\mathcal{F}}^\star \to \ell_{\mathcal{F}}^\star$, $C^TC$ is well-defined $\ell_{\mathcal{F}} \to \ell_{\mathcal{F}}^\star$. Then we have,
     \begin{align*}
     \langle e_i, C^T C e_j \rangle &= \langle e_0, P_i(D)P_j(D) e_0\rangle \qquad \text{(Corollary \ref{cor:rowsandcolsofC})}\\
     &= \int P_i(s) P_j(s) \,\d\nu(s) \qquad \text{(Theorem \ref{thm:maindeift} part (iii))}\\
     &= \int P_i(s) P_j(s) \frac{\d\nu}{\d\mu}(s) \,\d\mu(s) \qquad \text{(definition of R--N derivative)} \\
     &= \left\langle e_i, \frac{\d\nu}{\d\mu}(J) e_j \right\rangle \qquad \text{(Theorem \ref{thm:maindeift} part (iii))}.
     \end{align*}
     This completes the proof.
   \end{proof}
 \end{lemma}
 
 \begin{proposition}\label{prop:Cbounded}
   Let $J$ and $D$ be bounded Jacobi operators with spectral measures $\mu$ and $\nu$ respectively, and connection coefficient matrix $C = C_{J\to D}$. If $\frac{\mathrm{d}\nu}{\mathrm{d}\mu} \in L^\infty_\mu(\R)$ then $C$ is a bounded operator on $\ell^2$ and
 \begin{equation*}
 \|C\|_2^2 = \muesssup_{s \in \sigma(J)} \left|\frac{\mathrm{d}\nu}{\mathrm{d}\mu}(s)\right|.
 \end{equation*}
 Here $\|\cdot \|_2$ is the operator norm from $\ell^2 \to \ell^2$ and $\muesssup$ is the supremum up to $\mu$-almost everywhere equivalence of functions.
 \begin{proof}
Since $\mu$ is a probability measure (and hence $\sigma$-finite), we have the standard characterisation,
   $$
   \muesssup_{s \in \sigma(J)} \left|\frac{\mathrm{d}\nu}{\mathrm{d}\mu}(s)\right| = \sup_{\substack{g \in L_\mu^1(\R) \\ \|g\|_1 \leq 1}} \int \frac{\d\nu}{\d\mu}(s) |g(s)| \, \d \mu(s),
   $$
   which can be modified to
    $$
   \muesssup_{s \in \sigma(J)} \left|\frac{\mathrm{d}\nu}{\mathrm{d}\mu}(s)\right| = \sup_{\substack{f \in L_\mu^2(\R) \\ \|f\|_2 \leq 1}} \int \frac{\d\nu}{\d\mu}(s) (f(s))^2 \, \d \mu(s),
   $$
   by associating positive functions $g \in L^1_\mu(\R)$ with their square-roots $f \in L_\mu^2(\R)$. Now, since $\frac{\mathrm{d}\nu}{\mathrm{d}\mu} \in L_\mu^\infty(\R)$, this supremum can actually be taken over all polynomials by the following argument. If $f \in L_\mu^2(\R)$ with $\|f\|_{2,\mu} \leq 1$ then for any $\eps > 0$ there exists polynomial $p$ such that $\|p\|_{2,\mu} \leq 1$ and $\|f-p\|_2 \leq \eps$, since polynomials are dense in $L_\mu^2(\R)$ (this follows, for example, from the compact support of $\mu$ \cite{akhiezer1965classical}). It is readily shown using H\"older and triangle inequalities that
   \begin{equation*}
   \int \frac{\d\nu}{\d\mu}(s) \left((f(s))^2 - (p(s))^2 \right) \, \d \mu(s) \leq \muesssup_{s \in \sigma(J)} \left|\frac{\mathrm{d}\nu}{\mathrm{d}\mu}(s)\right| 2\eps,
   \end{equation*}
   so that any supremum can be arbitrarily approximated using a polynomial.

    Since $\{P_k \}_{k=0}^\infty$ is a complete orthonormal basis of $L_\mu^2(\R)$ (completeness holds as a result of $\mu$ having compact support), then $f \in L_\mu^2(\R)$ if and only if there is a unique sequence $v \in \ell^2$ such that $f = \sum_{k=0}^\infty v_k P_k$ with the series converging in the $L_\mu^2(\R)$ norm. Furthermore, $\|f\|_2 = \|v\|_2$. Hence,
    \begin{eqnarray*}
   \muesssup_{s \in \sigma(J)} \left|\frac{\mathrm{d}\nu}{\mathrm{d}\mu}(s)\right| &=& \sup_{\substack{v \in \ell_{\mathcal{F}} \\ \|v\|_2 = 1}} \int \frac{\d\nu}{\d\mu}(s) \left( \sum_{k} v_k P_k(s) \right)^2 \, \d \mu(s) \\
   &=&  \sup_{\substack{v \in \ell_{\mathcal{F}} \\ \|v\|_2 = 1}} \int \sum_{j,k}  v_j v_k \frac{\d\nu}{\d\mu}(s) P_j(s) P_k(s) \, \d \mu(s). \\
   &=& \sup_{\substack{v \in \ell_{\mathcal{F}} \\ \|v\|_2 = 1}} \sum_{j,k} v_j v_k \int \frac{\d\nu}{\d\mu}(s) P_j(s) P_k(s) \, \d \mu(s). 
   \end{eqnarray*}

 Since $\ell_{\mathcal{F}}$ is a dense subspace of $\ell^2$, we have $\|C\|_2 = \sup_{\substack{v \in \ell_{\mathcal{F}}},\\ \|v\|_2 = 1} \|Cv\|_2$. Now, $\|Cv\|_2^2 = \langle v, C^T C v\rangle$, and by Lemma \ref{lem:CTC}, $C^TC = \frac{\d\nu}{\d\mu}(J)$. Therefore,
   \begin{eqnarray*}
     \|C\|_2^2 &=&  \sup_{\substack{v \in \ell_{\mathcal{F}} \\ \|v\|_2 = 1}} \left\langle v, \frac{\d\nu}{\d\mu}(J) v \right\rangle = \sup_{\substack{v \in \ell_{\mathcal{F}} \\ \|v\|_2 = 1}} \sum_{j,k} v_j v_k \left[\frac{\d\nu}{\d\mu}(J)\right]_{j,k}.
     \end{eqnarray*}
   By Theorem \ref{thm:maindeift} part (iv), we conclude that
   \begin{equation}\label{eqn:normCint}
     \|C\|_2^2 = \sup_{\substack{v \in \ell_{\mathcal{F}} \\ \|v\|_2 = 1}}  \sum_{j,k=0}^\infty  v_j v_k \int \frac{\d\nu}{\d\mu}(s) P_j(s) P_k(s) \, \d \mu(s).
     \end{equation}
   Therefore, $\|C\|_2^2 =  \muesssup_{s \in \sigma(J)} \left|\frac{\mathrm{d}\nu}{\mathrm{d}\mu}(s)\right|$ and $C$ is bounded.  
 \end{proof}
 \end{proposition}

\begin{corollary}\label{cor:Cbddinvert}
Let $J$ and $D$ be bounded Jacobi operators with spectral measures $\mu$ and $\nu$ respectively, and connection coefficient matrix $C = C_{J\to D}$. If $\frac{\mathrm{d}\nu}{\mathrm{d}\mu} \in L^\infty_\mu(\R)$ and $\frac{\mathrm{d}\mu}{\mathrm{d}\nu} \in L^\infty_\nu(\R)$ then $C$ is bounded and invertible on $\ell^2$.
\begin{proof}
  By Proposition \ref{prop:Cbounded}, $C_{J\to D}$ is bounded if $\frac{\mathrm{d}\nu}{\mathrm{d}\mu} \in L^\infty_\mu(\R)$, and $C_{D\to J}$ is bounded if $\frac{\mathrm{d}\mu}{\mathrm{d}\nu} \in L^\infty_\nu(\R)$. Combining this the fact that $C_{J\to D}^{-1} = C_{D \to J}$, as operators from $\ell_{\mathcal{F}}$ to itself, we complete the proof.
\end{proof}
\end{corollary}

 The following definition and lemma are useful later.
 
 \begin{definition}\label{def:muconnection}
   Given polynomial sequences $P_0,P_1,P_2,\ldots$ and $Q_0,Q_1,Q_2,\ldots$ for 
Jacobi operators $J$ and $D$ respectively, we define the matrix 
$C^\mu$ to be the connection coefficients matrix between 
$P_0^\mu,P_1^\mu,P_2^\mu,\ldots$ and $Q_0,Q_1,Q_2,\ldots$ as in Definition 
\ref{def:connection}, where $P_0^\mu,P_1^\mu,P_2^\mu,\ldots$ are the first 
associated polynomials for $J$ as in Definition \ref{associatedpolys}. 
Noting that the lower triangular matrix $(C^\mu)^T$ is a well defined operator from $\ell_{\mathcal{F}}^\star$ into itself, we have
   \begin{equation*}
   \left(\begin{array}{c} P^\mu_0(s) \\ P^\mu_1(s) \\ P^\mu_2(s) \\ \vdots 
\end{array} \right) = (C^\mu)^T \left(\begin{array}{c} Q_0(s) \\ Q_1(s) \\ Q_2(s) 
\\ \vdots \end{array} \right) \text{ for all } s.
  \end{equation*} 
  \begin{remark}
   Note that $C^\mu$ is \emph{strictly} upper triangular, because the first 
associated polynomials have their degrees one less than their indices.
  \end{remark}
  \end{definition} 
 
 \begin{lemma}\label{lem:Jmu}
  The operator $C^\mu$ as defined above for $C_{J\to D}$ is in fact 
$\beta_0^{-1}\left(0 , C_{J^\mu \to D} \right)$, where
  \begin{equation*}
   J^\mu = \left(\begin{array}{cccc}
      \alpha_1 & \beta_1 & & \\
      \beta_1 & \alpha_2 & \beta_2 & \\
      & \beta_2 & \alpha_3 & \sddots \\
      & & \sddots & \sddots
    \end{array} \right).
  \end{equation*}
\begin{proof}
 The (unique) orthonormal polynomials for $J^\mu$ are 
$\beta_0 P_1^\mu,\beta_0P_2^\mu,\beta_0P_3^\mu,\ldots$, and $P^\mu_0 = 0$.
\end{proof}
 \end{lemma}

%% file: toeplitzplusfinite.tex
\section{Toeplitz-plus-finite-rank Jacobi operators}\label{sec:Toeplitzplusfinite}
  
  In this section we present several novel results which show how the connection coefficient matrices can be used for computing the spectral measure of a Toeplitz-plus-finite-rank Jacobi operator.
  
  \subsection{Jacobi operators for Chebyshev polynomials}
  
 There are two particular Jacobi operators with Toeplitz-plus-finite-rank structure that 
are of great interest,
\begin{equation}\label{eqn:DeltaGamma}
    \Delta = \left(\begin{array}{cccc}
      0 & \frac12 & & \\
      \frac12 & 0 & \frac12 & \\
      & \frac12 & 0 & \sddots \\
      & & \sddots & \sddots
    \end{array} \right), \text{ and } \Gamma = \left(\begin{array}{ccccc}
      0 & \frac1{\sqrt{2}} & & & \\
      \frac1{\sqrt{2}} & 0 & \frac12 & & \\
      & \frac12 & 0 & \frac12 & \\
      & & \frac12 & 0 & \sddots \\
      & & & \sddots & \sddots
    \end{array} \right).
  \end{equation}
The spectral measures of $\Delta$ and $\Gamma$ are
\begin{equation*}
 \d\mu_\Delta(s) = \frac{2}{\pi} \sqrt{1-s^2}\d s, \quad \d\mu_\Gamma(s) = 
\frac{1}{\pi}\frac{1}{\sqrt{1-s^2}} \d s,
\end{equation*}
supported on $[-1,1]$. 

Using results of Stieltjes in his seminal paper \cite{stieltjes1894recherches}, \cite[App.]{akhiezer1965classical}, the principal resolvent can be written elegantly as a continued  fraction,
  \begin{equation}\label{continuedfraction}
   G(\lambda) = \frac{-1}{\lambda - \alpha_0 - \frac{\beta_0^2}{\lambda - 
 \alpha_1 - \frac{\beta_1^2}{\lambda - \alpha_2 -...}}}.
  \end{equation}
  
  Using this gives explicit expressions for the principal resolvents,
\begin{equation*}
 G_\Delta(\lambda) = 2\sqrt{\lambda+1}\sqrt{\lambda-1}-2\lambda, \quad 
G_\Gamma(\lambda) = \frac{-1}{\sqrt{\lambda+1}\sqrt{\lambda-1}}.
\end{equation*}
\begin{remark}
 We must be careful about which branch we refer to when we write the 
resolvents in this explicit form. Wherever $\sqrt{}$ is written above we mean 
the standard branch that is positive on $(0,\infty)$ with branch cut 
$(-\infty,0]$. This gives a branch cut along $[-1,1]$ in both cases, the 
discontinuity of $G$ across which makes the Perron--Stieltjes inversion formula in Theorem 
\ref{thm:G} work. It also ensures the $\mathcal{O}(\lambda^{-1})$ decay resolvents enjoy as $\lambda \to \infty$.
\end{remark}

The orthonormal polynomials for $\Delta$ are the Chebyshev polynomials of the second kind, which we denote $U_k(s)$,
\begin{equation*}
U_k(s) = \frac{\sin((k+1)\cos^{-1}(s))}{\sin(\cos^{-1}(s))}.
\end{equation*}

The orthonormal polynomials for $\Gamma$ are the \emph{normalised} Chebyshev 
polynomials of the first kind, which we denote $\tilde T_k(s)$. Note that these are not the usual Chebyshev polynomials of the first kind (denoted $T_k(s)$) \cite{gautschi2004orthogonal,deift2000orthogonal}. We in fact have,
\begin{equation*}
 \tilde{T}_0(s) = 1, \quad \tilde{T}_k(s) = \sqrt{2}\cos(k\cos^{-1}(s)).
\end{equation*}

The first associated polynomials have simple relationships with the orthonormal polynomials,
\begin{equation}\label{eqn:muderivs}
 U^{\mu_\Delta}_k = 2U_{k-1}, \quad \tilde{T}^{\mu_\Gamma}_k = \sqrt{2}U_{k-1}.
\end{equation}

\subsection{Basic perturbations}

In this section we demonstrate for two simple, rank-one perturbations of $\Delta$ how the connection coefficient matrix relates properties of the spectrum of the operators. This will give some intuition as to what to expect in more general cases.

  \begin{example}[Basic perturbation 1]\label{ex:basic1}
   Let $\alpha \in \R$, and define
   \begin{equation*}
    J_\alpha = \left(\begin{array}{ccccc}
      \frac{\alpha}{2} & \frac12 & & & \\
      \frac12 & 0 & \frac12 & & \\
      & \frac12 & 0 & \frac12 & \\
      & & \frac12 & 0 & \sddots \\
      & & & \sddots & \sddots
    \end{array} \right).
   \end{equation*}
Then the connection coefficient matrix $C_{J_\alpha \to \Delta}$ is the bidiagonal Toeplitz matrix
\begin{equation}\label{eqn:basic1C}
 C_{J_\alpha \to \Delta} = \left(\begin{array}{ccccc}
      1 &  -\alpha &  &  &  \\
       & 1 & -\alpha &  &  \\
       &   & 1 & -\alpha &  \\
       &  &  & \sddots & \sddots
    \end{array}
    \right).
\end{equation}
This can be computed using the explicit recurrences \eqref{connectionrecurrence1}--\eqref{connectionrecurrence2}. The connection coefficient matrix $C_{\Delta \to J_\alpha}$ (which is the inverse of $C_{J_\alpha \to \Delta}$ on $\ell_{\mathcal{F}}$) is the full Toeplitz matrix
\begin{equation*}
 C_{\Delta \to J_\alpha} = \left(\begin{array}{ccccc}
      1 &  \alpha & \alpha^2 & \alpha^3 & \cdots \\
       & 1 & \alpha & \alpha^2 & \cdots \\
       &   & 1 & \alpha & \cdots \\
       &  &  & \sddots & \sddots
    \end{array}
    \right).
\end{equation*}
From this we see that $C=C_{J_\alpha \to \Delta}$ has a bounded inverse in $\ell^2$ if and only if $|\alpha|<1$. Hence by Theorem \ref{thm:Ccommute}, if $|\alpha|<1$ then $CJ_\alpha C^{-1} = \Delta$ with each operator bounded on $\ell^2$, so that $\sigma(J_\alpha) = \sigma(\Delta) = [-1,1]$. We will discuss what happens when $|\alpha| \geq 1$ later in the section.
\end{example}

\begin{example}[Basic perturbation 2]\label{ex:basic2}
   Let $\beta > 0$, and define
   \begin{equation*}
    J_\beta = \left(\begin{array}{ccccc}
      0 & \frac{\beta}2 & & & \\
      \frac{\beta}2 & 0 & \frac12 & & \\
      & \frac12 & 0 & \frac12 & \\
      & & \frac12 & 0 & \sddots \\
      & & & \sddots & \sddots
    \end{array} \right).
   \end{equation*}
Then the connection coefficient matrix $C_{J_\beta \to \Delta}$ is the banded Toeplitz-plus-rank-1 matrix
\begin{equation}\label{eqn:basic2C}
 C_{J_\beta \to \Delta} = \left(\begin{array}{cccccc}
      1 &  0 & \beta^{-1}-\beta &  &  &  \\
       & \beta^{-1} & 0 & \beta^{-1}-\beta &  &  \\
       &   & \beta^{-1} & 0 &  \beta^{-1}-\beta & \\
       &   &  & \beta^{-1} & 0 & \sddots \\
      &  & & & \sddots & \sddots
    \end{array}
    \right).
\end{equation}
Just as in Example \ref{ex:basic1}, this can be computed using the explicit recurrences \eqref{connectionrecurrence1}--\eqref{connectionrecurrence2}. The connection coefficient matrix $C_{\Delta \to J_\beta}$ (which is the inverse of $C_{J_\beta \to \Delta}$ on $\ell_{\mathcal{F}}$) is the Toeplitz-plus-rank-1 matrix
\begin{equation*}
 C_{\Delta \to J_\beta} = \left(\begin{array}{cccccccc}
      1 & 0  & \beta^2-1 & 0 & (\beta^2-1)^2 & 0 & (\beta^2-1)^3 & \cdots \\
       & \beta & 0 & \beta(\beta^2-1) & 0 & \beta(\beta^2-1)^2 & 0 & \cdots \\
       &   & \beta & 0 & \beta(\beta^2-1) & 0 & \beta(\beta^2-1)^2 & \cdots \\
       &  &  & \beta & 0 & \beta(\beta^2-1) & 0 & \cdots \\
       &  &  &   & \sddots & \sddots & \sddots & \sddots
    \end{array}
    \right).
\end{equation*}
From this we see that $C = C_{J_\beta \to \Delta}$ has a bounded inverse on $\ell^2$ if and only if $\beta<\sqrt{2}$. Hence by Theorem \ref{thm:Ccommute}, if $\beta<\sqrt{2}$ then $CJ_\beta C^{-1} = \Delta$ with each operator bounded on $\ell^2$, so that $\sigma(J_\beta) = \sigma(\Delta) = [-1,1]$. We will discuss what happens when $\beta \geq \sqrt{2}$ later in the section. Note that the case $\beta = \sqrt{2}$ gives the Jacobi operator $\Gamma$ in equation \eqref{eqn:DeltaGamma}.
\end{example}

\subsection{Fine properties of the connection coefficients}\label{subsec:fineconnectioncoeffs}

The two basic perturbations of $\Delta$ discussed above give connection coefficient matrices that are highly structured. The following lemmata and theorems prove that this is no coincidence; in fact, if Jacobi operator $J$ is a finite-rank perturbation of $\Delta$ then $C_{J \to \Delta}$ is also a finite-rank perturbation of Toeplitz.

\begin{remark}
 Note for the following results that all vectors and matrices are indexed starting from 0.
\end{remark}

\begin{lemma}\label{lem:diagonalsofC}
  If $\delta_j = \beta_j$ for $j \geq n$ then $c_{jj} = c_{nn}$ for all $j \geq n$.
  \begin{proof}
    By the recurrence in Lemma \ref{lem:5ptsystem}, $c_{jj} = (\delta_{j-1}/\beta_{j-1})c_{j-1,j-1}$. The result follows by induction.
  \end{proof}
\end{lemma}

\begin{lemma}\label{lem:ToeplitzC}
Let $J$ and $D$ be Jacobi operators with coefficients $\{\alpha_k,\beta_k\}$ and $\{\gamma_k,\delta_k\}$ respectively, such that there exists an $n$ such that\footnote{More intuitively, the entries of $J$ and $D$ are both equal and Toeplitz, except in the principal $n\times n$ submatrix, where neither statement necessarily holds.}
\begin{equation*}
\alpha_k = \gamma_k = \alpha_n, \quad \beta_{k-1} = \delta_{k-1} = \beta_{n-1} \text{ for all } k\geq n.
\end{equation*}
 Then the entries of the connection coefficient matrix $C = C_{J \to D}$ satisfy
 \begin{equation*}
 c_{i,j} = c_{i-1,j-1} \text{ for all } i,j > 0 \text{ such that } i \geq n.
 \end{equation*} 
\begin{remark}
 This means that $C$ is of the form $C = C_{\rm Toe} + C_{\rm fin}$ where $C_{\rm Toe}$ is Toeplitz and $C_{\rm fin}$ is zero except in the first $n-1$ rows. For example, when $n = 4$, we have the following structure
\begin{equation*}
C = \left( \begin{array}{ccccccc}
t_0 & t_1 & t_2 & t_3 & t_4 & t_5 & \cdots \\
& t_0 & t_1 & t_2 & t_3 & t_4 & \sddots \\
&  & t_0 & t_1 & t_2 & t_3 & \sddots \\
&  &  & t_0 & t_1 & t_2 & \sddots \\
&  &  &     & \sddots & \sddots & \sddots
\end{array}\right) +
\left( \begin{array}{cccccc}
f_{00} & f_{01} & f_{02} & f_{03} & f_{04}  & \cdots \\
& f_{11} & f_{12} & f_{13} & f_{14}  & \cdots \\
&  & f_{22} & f_{23} & f_{24}  & \cdots \\
&  &  &        &       &    \\
&  &  &        &       &    \\
\end{array} \right).
\end{equation*}
\end{remark}
\begin{proof}
 We prove by induction on $k = 0,1,2,\ldots$ that
 \begin{equation}\label{eqn:inductivehyp}
  c_{i,i+k} = c_{i-1,i+k-1} \text{ for all } i \geq n.
 \end{equation}
We use the recurrences in Lemma \ref{lem:5ptsystem} and equations \eqref{connectionrecurrence1}--\eqref{connectionrecurrence2}. The base case $k=0$ is proved in Lemma \ref{lem:diagonalsofC}. Now we deal with the second 
base case, $k=1$. For any $i \geq n$, we have $\beta_i = \delta_i = \beta_{i-1} = \delta_{i-1}$, and $\alpha_i = \gamma_i$, so
\begin{align*}
c_{i,i+1} &=\left(\delta_{i-1}c_{i-1,i} + (\gamma_i-\alpha_{i})c_{i,i} 
+ \delta_ic_{i+1,i} - \beta_{i-1}c_{i,i-1}\right)/\beta_{i} \\
          &= 1 \cdot c_{i-1,i} + 0 \cdot c_{i,i} + 1 \cdot 0 - 1 \cdot 0 \\
          &= c_{i-1,i}.
\end{align*}
Now we deal with the case $k > 1$. For any $i \geq n$, we have $\delta_i = \delta_{i-1} = \beta_{i+k-2} = \beta_{i+k-1}$, and $\alpha_{i+k-1} = \gamma_i$, so
\begin{align*}
 c_{i,i+k} &= \left(\delta_{i-1}c_{i-1,i+k-1} + (\gamma_i-\alpha_{i+k-1})c_{i,i+k-1} 
+ \delta_ic_{i+1,i+k-1} - \beta_{i+k-2}c_{i,i+k-2}\right)/\beta_{i+k-1} \\
&= 1 \cdot c_{i-1,i+k-1} + 0 \cdot c_{i,i+k-1} + 1 \cdot c_{i+1,i+k-1} - 1\cdot c_{i,i+k-2} \\
           &= c_{i-1,i+k-1} + c_{i+1,i+k-1} - c_{i,i+k-2} \\
           &= c_{i-1,i+k-1}.
\end{align*}
The last line follows from the induction hypothesis for the case $k-2$ (hence 
why we needed two base cases).
\end{proof}
\end{lemma} 
  
  The special case in which $D$ is Toeplitz gives even more structure to $C$, as demonstrated by the following theorem. We state the results for a finite-rank perturbation of the free Jacobi operator $\Delta$, but they apply to general Toeplitz-plus-finite rank Jacobi operators because the connection coefficients matrix $C$ is unaffected by a scaling and shift by the identity applied to both $J$ and $D$.
  
 \begin{theorem}\label{thm:connectioncoeffs}
   Let $J$ be a Jacobi operator such that there exists an $n$ such that
   \begin{equation*}
   \alpha_k = 0, \quad \beta_{k-1} = \frac12 \text{ for all } k \geq n,
   \end{equation*}
   i.e. it is equal to the free Jacobi operator $\Delta$ outside the $n\times n$ principal submatrix.
   Then the entries of the connection coefficient matrix $C = C_{J \to \Delta}$ satisfy
   \begin{align}
     c_{i,j} &= c_{i-1,j-1} \text{ for all } i,j > 0 \text{ such that } i+j \geq 2n  \label{iplusjcase} \\
     c_{0,j} &= 0 \text{ for all } j \geq 2n. \label{jandzerocase}
    \end{align}
\begin{remark}
This means that $C$ is of the form $C = C_{\rm Toe} + C_{\rm fin}$ where $C_{\rm Toe}$ is Toeplitz with bandwidth $2n-1$ and $C_{\rm fin}$ zero except for entries in the $(n-1) \times (2n-2)$ principal submatrix. For example when $n = 4$, we have the following structure,
\begin{equation*}
C = \left( \begin{array}{cccccccccc}
t_0 & t_1 & t_2 & t_3 & t_4 & t_5 & t_6 & t_7 & & \\
& t_0 & t_1 & t_2 & t_3 & t_4 & t_5 & t_6 & t_7 &  \\
&  & t_0 & t_1 & t_2 & t_3 & t_4 & t_5 & t_6 & \sddots \\
&  &  & t_0 & t_1 & t_2 & t_3 & t_4 & t_5 &  \sddots \\
&  &  &     & \sddots &  \sddots &  \sddots& \sddots &  \sddots&  \sddots
\end{array}\right) +
\left( \begin{array}{cccccccc}
f_{0,0} & f_{0,1} & f_{0,2} & f_{0,3} & f_{0,4}  & f_{0,5}& &  \\
& f_{1,1} & f_{1,2} & f_{1,3} & f_{1,4}  & & &  \\
&  & f_{2,2} & f_{2,3} &   & & &  \\
&  &  &        &       &  & &  \\
&  &  &        &       &  & &  \\
\end{array} \right).
\end{equation*}
\end{remark}
   \begin{proof}
First we prove \eqref{iplusjcase}. Fix $i,j$ such that $i+j \geq 2n$. Note that 
the case $i \geq n$ is proven in Lemma \ref{lem:ToeplitzC}. Hence we assume $i < n$, and therefore $j > n$. Using Lemma \ref{lem:5ptsystem} and equations \eqref{connectionrecurrence1}--\eqref{connectionrecurrence2} we find the following recurrence. Substituting $\delta_i = \frac12$, $\gamma_i = 0$ for all $i$, and $\alpha_{k} = 0$, $\beta_{k-1} =\frac12$ for $k \geq n$ into the recurrence, we have
\begin{align*}
     c_{i,j} &= \left(\delta_{i-1}c_{i-1,j-1} + (\gamma_i-\alpha_{j-1})c_{i,j-1} 
+ \delta_ic_{i+1,j-1} - \beta_{j-2}c_{i,j-2}\right)/\beta_{j-1}.  \\
&= \left( \frac12 c_{i-1,j-1} - \alpha_{j-1}c_{i,j-1} + 
\frac12c_{i+1,j-1} - \beta_{j-2}c_{i,j-2} \right)/\beta_{j-1} \\
     &= c_{i-1,j-1} + c_{i+1,j-1} - c_{i,j-2}.
    \end{align*}
Repeating this process on $c_{i+1,j-1}$ in the above expression gives
\begin{equation*}
 c_{i,j} = c_{i-1,j-1} + c_{i+2,j-2} - c_{i+1,j-3}.
\end{equation*}
Repeating the process on $c_{i+2,j-2}$ and so on eventually gives
\begin{equation*}
 c_{i,j} = c_{i-1,j-1} + c_{n,i+j-n} - c_{n-1,i+j-n-1}.
\end{equation*}
By Lemma \ref{lem:ToeplitzC}, $c_{n,i+j-n} = c_{n-1,i+j-n-1}$, so we are left with $c_{i,j} = c_{i-1,j-1}$. This completes the proof of \eqref{iplusjcase}.

Now we prove \eqref{jandzerocase}. Let $j \geq 2n$. Then
    \begin{align*}
     c_{0,j} &= \left((\gamma_0-\alpha_{j-1})c_{0,j-1} + \delta_0c_{1,j-1} - 
\beta_{j-2} c_{0,j-2}\right)/\beta_{j-1} \\
&= \left(-\alpha_{j-1}c_{0,j-1} + \frac12c_{1,j-1} - \beta_{j-2} 
c_{0,j-2}\right)/\beta_{j-1} \\
     &= c_{1,j-1} - c_{0,j-2}.
    \end{align*}
This is equal to zero by \eqref{iplusjcase}, because $1 + (j-1) \geq 2n$.
   \end{proof}
   \end{theorem}  
  \begin{corollary}\label{cor:cmu}
   Let $C^\mu$ be as defined in Definition \ref{def:muconnection} for $C$ as 
in Theorem \ref{thm:connectioncoeffs}. Then $C^\mu = C^\mu_{\rm Toe} + C^\mu_{\rm fin}$, where $C^\mu_{\rm Toe}$ is Toeplitz with 
bandwidth $2n-2$ and $C^\mu_{\rm fin}$ is zero outside the $(n-2) \times (2n-1)$ principal submatrix.
   \begin{proof}
    This follows from Theorem \ref{thm:connectioncoeffs} applied to $J^\mu$ as 
defined in Lemma \ref{lem:Jmu}.
   \end{proof}
  \end{corollary}

  \begin{remark}\label{rem:CCT} 
    A technical point worth noting for use in proofs later is that for Toeplitz-plus-finite-rank Jacobi operators like $J$ and $D$ occurring in Theorem \ref{thm:connectioncoeffs} and Corollary \ref{cor:cmu}, the operators $C$, $C^T$, $C^\mu$ and $(C^\mu)^T$ all map $\ell_{\mathcal{F}}$ to $\ell_{\mathcal{F}}$. Consequently, combinations such as $CC^T$, $C^\mu C^T$ are all well defined operators from $\ell_{\mathcal{F}}$ to $\ell_{\mathcal{F}}$.
  \end{remark}
  
  \subsection{Properties of the resolvent}
  
  When the Jacobi operator $J$ is Toeplitz-plus-finite rank, as a consequence of the structure of the connection coefficients matrix proved in subsection \ref{subsec:fineconnectioncoeffs}, the principal resolvent $G$ (see Definition \ref{def:resolvent}) and spectral measure (see Theorem \ref{thm:G}) are also highly structured. As usual these proofs are stated for a finite-rank perturbation of the free Jacobi operator $\Delta$, but apply to general Toeplitz-plus-finite rank Jacobi operators by applying appropriate scaling and shifting.
  
\begin{theorem}\label{thm:mainresolvent}
 Let $J$ be a Jacobi operator such that there exists an $n$ such that
 \begin{equation*}
 \alpha_k = 0, \quad \beta_{k-1} = \frac12 \text{ for all } k \geq n,
 \end{equation*}
 i.e. it is equal to the free Jacobi operator $\Delta$ outside the $n\times n$ principal submatrix. Then the principal resolvent for $J$ is
 \begin{equation}\label{principalresidualformula}
   G(\lambda) = \frac{G_\Delta(\lambda) - p_C^{\mu}(\lambda)}{p_C(\lambda)},
  \end{equation}
  where
  \begin{align}
   p_C(\lambda) &= \sum_{k = 0}^{2n-1} c_{0,k} P_k(\lambda) = \sum_{k=0}^{2n-1} 
\langle C^Te_k ,C^T e_0 \rangle U_k(\lambda),\label{eqn:PandUequality} \\  
  p^\mu_C(\lambda) &= \sum_{k = 1}^{2n-1} c_{0,k} P^\mu_k(\lambda) = \sum_{k=0}^{2n-1} 
  \langle (C^\mu)^T e_k, C^T e_0 \rangle U_k(\lambda),\label{eqn:PmuandUequality}
  \end{align}
  $P_k$ are the orthonormal polynomials for $J$, $P_k^\mu$ are the first associated polynomials for $J$ as in Definition \ref{associatedpolys}, and $U_k$ are the Chebyshev polynomials of the second kind.
\begin{remark}\label{rem:PCmu}
  $p_C^{\mu}$ is the $\mu$-derivative of $p_C$ as in Definition \ref{def:muderiv}.
\end{remark}
 \begin{proof}
       Using Theorem \ref{thm:G} and Proposition \ref{prop:Cmeasuremultplier},
       \begin{align*}
       G_\Delta(\lambda) &= \int (s-\lambda)^{-1} \d\mu_\Delta(s) \\
       &= \int (s-\lambda)^{-1} p_C(s) \d\mu(s).
       \end{align*}
       Now, since $p_C$ is a polynomial we can split this into
       \begin{equation*}
       G_\Delta(\lambda) = \int (s-\lambda)^{-1} p_C(\lambda) \d\mu(s) + \int (s-\lambda)^{-1} (p_C(s)-p_C(\lambda)) \d\mu(s).
       \end{equation*}
   The first term is equal to $p_C(\lambda)G(\lambda)$, and the second term is equal to $p_C^\mu(\lambda)$ by Lemma \ref{lem:associatedpolys} and Remark \ref{rem:PCmu}. The equation can now be immediately rearranged to obtain \eqref{principalresidualformula}. 
   
   To see the equality in equation \eqref{eqn:PandUequality}, note that by the definition of the 
connection coefficient matrix $C$,
   \begin{align*}
    \sum_{k=0}^{2n-1}c_{0,k}P_k(\lambda) &= \sum_{k=0}^{2n-1} 
c_{0,k}\sum_{j=0}^{2n-1} c_{j,k}U_j(\lambda) \\
    &= \sum_{j=0}^{2n-1} \left(\sum_{k=0}^{2n-1} c_{0,k}c_{j,k} 
\right)U_j(\lambda) \\
    &= \sum_{j=0}^{2n-1} \langle C^Te_j,C^T e_0 \rangle U_j(\lambda).
   \end{align*}
Equation \eqref{eqn:PmuandUequality} follows by the same algebra.
 \end{proof}
\end{theorem}

\begin{theorem}\label{thm:mainmeasure}
 Let $J$ be a Jacobi operator such that there exists an $n$ such that
\begin{equation*}
\alpha_k = 0, \quad \beta_{k-1} = \frac12 \text{ for all } k \geq n,
\end{equation*}
i.e. it is equal to the free Jacobi operator $\Delta$ outside the $n\times n$ principal submatrix.  Then the spectral measure for $J$ is
 \begin{equation}\label{eqn:mulambdaformula}
  \mu(s) = \frac{1}{p_C(s)}\mu_\Delta(s) + \sum_{k=1}^{r} w_k \delta_{\lambda_k}(s),
 \end{equation}
where $\lambda_1, \ldots,\lambda_{r}$ are the roots of $p_C$ in $\R \setminus \{1,-1\}$ such that
\begin{equation*}
 w_k = \lim_{\eps \searrow 0} \frac{\eps}{i} G(\lambda_k+i\eps) \neq 0.
\end{equation*}
There are no roots of $p_C$ inside $(-1,1)$, but there may be simple roots at $\pm 1$.
\begin{remark}\label{rem:rleqn}
  We will see in Theorem \ref{thm:joukowskimeasure} that the number of roots of $p_C$ for which $w_k \neq 0$ is at most $n$ (i.e. $r\leq n$). Hence, while the degree of $p_C$ is at most $2n-1$, at least $n-1$ are cancelled out by factors in the numerator.
  \end{remark}
\begin{proof}
  Let $G$ and $\mu$ be the principal resolvent and spectral measure of $J$ respectively. By Theorem \ref{thm:mainresolvent},
 \begin{equation*}
  G(\lambda) = \frac{G_\Delta(\lambda) - p_C^{\mu}(\lambda)}{p_C(\lambda)}.
 \end{equation*}
Letting $\lambda_1,\ldots,\lambda_{2n-1}$ be the roots of $p_C$ in the complex plane, define the set 
\begin{equation*}
S = [-1,1] \cup (\{\lambda_1,\ldots,\lambda_{2n-1}\} \cap \R).
\end{equation*}
By inspection of the above formula for $G$, and because resolvents of selfadjoint operators are analytic off the real line, we have that $G$ is continuous outside of $S$. Therefore, for any $s\in\R$ such that $\mathrm{dist}(s,S) > 0$, we have
 \begin{equation*}
  \lim_{\eps \searrow 0} \imag G(s+i\eps) = \imag G(s) = 0.
 \end{equation*}
Hence by Theorem \ref{thm:G} part (ii), for any interval $(s_1,s_2)$ such that $\mathrm{dist}(S,(s_1,s_2)) > 0$, we have $\mu((s_1,s_2)) + \frac12 \mu(\{s_1\}) + \frac12 \mu(\{s_2\}) = 0$. Therefore the essential support of $\mu$ is contained within $S$.

We are interested in the real roots of $p_C$. Let us consider the potential for roots of $p_C$ in the interval $[-1,1]$. By Proposition \ref{prop:Cmeasuremultplier}, $\d\mu_\Delta(s) = p_C(s) \d\mu(s)$ for all $s \in \R$. For any $s\in[-1,1]$ such that $p_C(s) \neq 0$, it follows that $\d\mu(s) = \frac{2}{\pi}\frac{\sqrt{1-s^2}}{p_C(s)}\d s$. From this we have
\begin{equation*}
1 \geq \mu((-1,1)) = \int_{-1}^1 \frac{2}{\pi}\frac{\sqrt{1-s^2}}{p_C(s)} \,\d s.
\end{equation*}
This integral is only finite, so $p_C$ has no roots in $(-1,1)$, but may have simple roots at $\pm 1$. One example where we have simple roots at $\pm 1$ is seen in Example \ref{ex:basic2rev} with $\beta = \sqrt{2}$.

Since $S$ is a disjoint union of $[-1,1]$ and a finite set $S'$ we can write
\begin{equation*}
\mu(s) \frac{1}{p_C(s)} \mu_\Delta(s) + \sum_{\lambda_k \in S'} \mu(\{\lambda_k\}) \delta_{\lambda_k}(s).
\end{equation*}

By Theorem \ref{thm:G} part (iii), $ \mu(\{s\}) = \lim_{\eps \searrow 0} \frac{\eps}{i} G(s+i\eps) \text{ for all } s\in\R$. This gives the desired formula for $w_k$. 
\end{proof}
\end{theorem}

\begin{remark}
 Theorem \ref{thm:mainmeasure} gives an explicit formula for the spectral measure of $J$, when $J$ is Toeplitz-plus-finite-rank Jacobi operator. The entries of $C$ can be computed in $\mathcal{O}(n^2)$ operations (for an $n \times n$ perturbation of Toeplitz). Hence, the absolutely continuous part of the measure can be computed \emph{exactly} in finite time. It would appear at first that we may compute the locations of the point spectrum by computing the roots of $p_C$, but as stated in Remark \ref{rem:rleqn} we find that not all real roots of $p_C$ have $w_k \neq 0$. Hence we rely on cancellation between the numerator and denominator in the formula for $G(\lambda)$, which  is a dangerous game, because if roots of polynomials are only known approximately then it is impossible to distinguish between cancellation and the case where a pole and a root are merely extremely close. Subsection \ref{subsec:joukowski} remedies this situation.
\end{remark}


\begin{example}[Basic perturbation 1 revisited]\label{ex:basic1rev}

The polynomial $p_C$ in Theorem \ref{thm:mainresolvent} is
\begin{equation*}
 p_C(\lambda) = c_{0,0} P_0(\lambda) +c_{0,1} P_1(\lambda) = 1 - \alpha(2\lambda - \alpha) = 2\alpha\left(\frac12(\alpha + \alpha^{-1}) - \lambda\right),
\end{equation*}
and the $\mu$-derivative is $p_C^\mu(\lambda) = -2\alpha$. Theorem \ref{thm:mainresolvent} gives
\begin{equation*}
 G(\lambda) = \frac{G_\Delta(\lambda) + 2\alpha}{2\alpha\left(\frac12(\alpha + \alpha^{-1})-\lambda\right)}.
\end{equation*}

Consider the case $|\alpha| \leq 1$. Then a brief calculation reveals $G_\Delta(\frac12(\alpha+\alpha^{-1})) = -2\alpha$. Hence the root $\lambda = \frac12(\alpha + \alpha^{-1})$ of the denominator is always cancelled out. Hence $G$ has no poles, and so $J$ has no eigenvalues.

In the case where $|\alpha| > 1$, we have a different situation. Here $G_\Delta(\frac12(\alpha+\alpha^{-1})) = -2\alpha^{-1}$. Therefore the root $\lambda = \frac12(\alpha + \alpha^{-1})$ of the denominator is \emph{never} cancelled out. Hence there is always a pole of $G$ at $\lambda = \frac12(\alpha+\alpha^{-1})$, and therefore also an eigenvalue of $J$ there.

Notice a heavy reliance on cancellations in the numerator and denominator for the existence of eigenvalues. The approach in subsection \ref{subsec:joukowski} avoids this.
\end{example}

\begin{example}[Basic perturbation 2 revisited]\label{ex:basic2rev}

The polynomial $p_C$ in Theorem \ref{thm:mainresolvent} is
\begin{equation*}
 p_C(\lambda) = c_{0,0}P_0(\lambda) + c_{0,2}P_2(\lambda) = 1 + (\beta^{-1}-\beta)(4\beta^{-1}\lambda^2 - \beta).
\end{equation*}
This simplifies to $p_C(\lambda) = 4(1-\beta^{-2})\left(\frac{\beta^4}{4(\beta^2-1)} - \lambda^2\right)$. Using Definition \ref{associatedpolys}, the $\mu$-derivative is $p_C^\mu(\lambda) = c_{0,2}P_2^\mu(\lambda) = 4\beta^{-1}\lambda$. Theorem \ref{thm:mainresolvent} gives
\begin{equation*}
 G(\lambda) = \frac{G_\Delta(\lambda) + 4\beta^{-1}\lambda}{4(1-\beta^{-2})\left(\frac{\beta^4}{4(\beta^2-1)}-\lambda^2\right)}.
\end{equation*}
Clearly the only points at which $G$ may have a pole is $\lambda = \pm\frac{\beta^2}{2\sqrt{\beta^2-1}}$. However, it is difficult to see whether there would be cancellation on the numerator. In the previous discussion on this example we noted that there would not be any poles when $|\beta|<\sqrt{2}$, which means that the numerator must be zero at these points, but it is far from clear here. The techniques we develop in the sequel illuminate this issue, especially for examples which are much more complicated than the two trivial ones given so far.
\end{example}

\subsection{The Joukowski transformation}\label{subsec:joukowski}

The following two lemmata and two theorems prove that the issue of cancellation and the number of discrete spectra in Theorem \ref{thm:mainresolvent} and Theorem \ref{thm:mainmeasure} can be solved by making the 
change of variables 
\begin{equation*}
 \lambda(z) = \frac12(z+z^{-1})
\end{equation*}
This map is known as the Joukowski map. It is an analytic bijection from $\mathbb{D} = \{z \in \C : |z| < 1 \}$ to $\C \setminus [-1,1]$, sending the unit circle to two copies of the interval $[-1,1]$.

The Joukowski map has special relevance for the principal resolvent of $\Delta$. A brief calculation reveals that for $z \in \mathbb{D}$,
  \begin{equation}\label{eqn:GDeltaz}
   G_\Delta(\lambda(z)) = -2z.
  \end{equation}
   Further, we will see that the polynomials $p_C(\lambda)$ and $p_C^\mu(\lambda)$ occurring in our formula for $G$ can be expressed neatly as polynomials in $z$ and $z^{-1}$. This is a consequence of a special property of the Chebyshev polynomials of the second kind, that for any $k \in \mathbb{Z}$ and $z \in \mathbb{D}$
  \begin{equation}\label{eqn:chebUmtoinfty}
     \frac{U_{m-k}(\lambda(z))}{U_m(\lambda(z))} \to z^k \text{ as } m \to \infty.
  \end{equation}
These convenient facts allow us to remove any square roots involved in the formulae in Theorem \ref{thm:mainresolvent}.
  
\begin{lemma}\label{lem:cformula}
Let $p_C(\lambda) = \sum_{k=0}^{2n-1} \langle e_0, CC^T e_k \rangle U_k(\lambda)$ as in Theorem 
\ref{thm:mainresolvent} and let $c$ be the symbol of $C_{\rm Toe}$, the Toeplitz part of $C$ as guaranteed by Theorem \ref{thm:connectioncoeffs}. Then
\begin{equation}\label{eqn:pCformula}
 p_C(\lambda(z)) = c(z)c(z^{-1}),
\end{equation}
where $\lambda(z) = \frac12(z + z^{-1})$.
   \begin{proof}
   The key quantity to observe for this proof is 
   \begin{equation}\label{eqn:czcziinvquantity}
   \frac{\sum_{k=0}^{2n-1}c_{m,m+k}P_{m+k}(\lambda(z))}{U_m(\lambda(z))},
   \end{equation}
   for $z \in \mathbb{D}$ as $m\to \infty$. We will show it is equal to both sides of equation \eqref{eqn:pCformula}. Consider the polynomial $U_m \cdot p_C$. The $j$th coefficient in an expansion in the basis $P_0,P_1,P_2,\ldots$ is given by
   \begin{equation*}
   \int \left(U_m(s)p_C(s)\right)P_j \d\mu(s) = \int U_m(s)P_j(s) \d\mu_{\Delta}(s) = c_{m,j},
   \end{equation*}
   because $p_C = \frac{\d\mu_{\Delta}}{\d\mu}$ by Proposition \ref{prop:Cmeasuremultplier}. Hence
   \begin{equation*}
   p_C(\lambda(z)) = \frac{\sum_{k=0}^{2n-1}c_{m,m+k}P_{m+k}(\lambda(z))}{U_m(\lambda(z))},
   \end{equation*}
   for all $m \in \mathbb{N}$ and all $z \in \mathbb{D}$.

   Now we show that \eqref{eqn:czcziinvquantity} converges to $c(z)c(z^{-1})$ as $m\to \infty$. By the definition of the connection coefficients, $P_{m+k} = \sum_{j=0}^{2n-1} c_{m+k-j,m+k} U_{m+k-j}$. Therefore,
   \begin{equation*}
   \frac{\sum_{k=0}^{2n-1}c_{m,m+k}P_{m+k}(\lambda(z))}{U_m(\lambda(z))} = \sum_{j,k=0}^{2n-1} c_{m,m+k} c_{m+k-j,m+k}\frac{U_{m+k-j}(\lambda(z))}{U_{m}(\lambda(z))}.
   \end{equation*}
   
   Now, by Theorem \ref{thm:connectioncoeffs}, $C = C_{\rm Toe} + C_{\rm fin}$, where $C_{\rm fin}$ is zero outside the $(n-1) \times (2n-2)$ principal submatrix. Hence for $m$ sufficiently large we have $c_{m,m+k} = t_k$ for a sequence $(t_k)_{k\in\mathbb{Z}}$ such that $t_k = 0$ for $k \notin \{0,1,\ldots,2n-1\}$. Hence we have for $m$ sufficiently large,
\begin{align*}
\frac{\sum_{k=0}^{2n-1}c_{m,m+k}P_{m+k}(\lambda(z))}{U_m(\lambda(z))} &= \sum_{k=0}^{2n-1} \sum_{j=0}^{2n-1} t_k t_j 
\frac{U_{m+k-j}(\lambda(z))}{U_m(\lambda(z))}.
\end{align*}
By equation \eqref{eqn:chebUmtoinfty}, this tends to $\sum_{k=0}^{2n-1} \sum_{j=0}^{2n-1} t_k t_j z^{j-k}$ as $m \to \infty$. This is equal to $c(z)c(z^{-1})$, as required to complete the proof.
   \end{proof}
  \end{lemma}

  \begin{lemma}\label{lem:cmuformula}
Let $p^\mu_C(\lambda) = \sum_{k=0}^{2n-1} 
\langle e_k ,C^\mu C^T e_0 \rangle U_k(\lambda)$ as in Theorem 
\ref{thm:mainresolvent} and let $c_\mu$ be the symbol of $C_{\rm Toe}^\mu$, the Toeplitz part of $C^\mu$ as guaranteed by Corollary \ref{cor:cmu}. Then
   \begin{equation}\label{eqn:pCmuformula}
p_C^\mu(\lambda(z)) = c(z^{-1})c_\mu(z) - 2z,
   \end{equation}
    where $\lambda(z) = \frac12(z+z^{-1})$ and $z \in \mathbb{D}$.
   \begin{proof}
    The key quantity to observe for this proof is 
    \begin{equation}\label{eqn:weirdquantity}
\frac{\sum_{k=0}^{2n-1}c_{m,m+k}P^\mu_{m+k}(\lambda(z))}{U_m(\lambda(z))},
    \end{equation}
    for $z\in\mathbb{D}$, as $m \to \infty$. We will compute two equivalent expressions for this quantity to derive equation \eqref{eqn:pCmuformula}. In the proof of Lemma \ref{lem:cformula}, it was shown that $U_m(\lambda) p_C(\lambda) = \sum_{k=0}^{2n-1} c_{m,m+k}P_k(\lambda)$. We take the $\mu$-derivative (see Definition \ref{def:muderiv}) of both sides as follows.
    \begin{align*}
    \int \frac{U_m(\lambda)p_C(\lambda) - U_m(s)p_C(s)}{\lambda - s} \d\mu(s) &= U_m(\lambda)\int \frac{p_C(\lambda) - p_C(s)}{\lambda - s} \d\mu(s) + \int \frac{U_m(\lambda) - U_m(s)}{\lambda - s} p_C(s)\d\mu(s) \\
    &= U_m(\lambda) p^\mu_C(\lambda) + U_m^{\mu_{\Delta}}(\lambda),
    \end{align*}
    because by Proposition \ref{prop:Cmeasuremultplier}, $p_C = \frac{\d\mu_{\Delta}}{\d\mu}$. Using the formula from equation \eqref{eqn:muderivs}, $U_m^{\mu_{\Delta}} = 2U_{m-1}$, we find that the $\mu$-derivative of $U_m(s) p_C(s)$ is equal to $ U_m(\lambda) p^\mu_C(\lambda) + 2U_{m-1}(\lambda)$. Taking the $\mu$-derivative from the other side gives $\sum_{k=0}^{2n-1} c_{m,m+k}P^\mu_k(\lambda)$. Taking the limit as $m\to\infty$ and using equation \eqref{eqn:chebUmtoinfty}, we have our first limit for the quantity in equation \eqref{eqn:weirdquantity}:
    \begin{equation*}
    \frac{\sum_{k=0}^{2n-1}c_{m,m+k}P^\mu_{m+k}(\lambda(z))}{U_m(\lambda(z))} \to p_C^\mu(\lambda(z)) + 2z \text{ as } m \to \infty.
    \end{equation*}
    
    Now we show that \eqref{eqn:weirdquantity} converges to $c_\mu(z)c(z^{-1})$ as $m\to \infty$. By the definition of the connection coefficients matrix $C^\mu$, $P^\mu_{m+k} = \sum_{j=1}^{2n-2} c^\mu_{m+k-j,m+k} U_{m+k-j}$. Therefore,
    \begin{equation*}
    \frac{\sum_{k=0}^{2n-1}c_{m,m+k}P^\mu_{m+k}(\lambda(z))}{U_m(\lambda(z))} = \sum_{k=0}^{2n-1}\sum_{j=1}^{2n-2} c_{m,m+k} c^\mu_{m+k-j,m+k}\frac{U_{m+k-j}(\lambda(z))}{U_{m}(\lambda(z))}.
    \end{equation*}
    
    By Corollary \ref{cor:cmu}, $C^\mu = C^\mu_{\rm Toe} + C^\mu_{\rm fin}$, where $C^\mu_{\rm fin}$ is zero outside the principal $(n-2) \times (2n-1)$ submatrix. Hence for $m$ sufficiently large we have $c^\mu_{m,m+k} = t^\mu_k$ for a sequence $(t^\mu_k)_{k\in\mathbb{Z}}$ such that $t^\mu_k = 0$ for $k \notin \{1,\ldots,2n-2\}$. Hence we have for  sufficiently large $m$,
    \begin{align*}
    \frac{\sum_{k=0}^{2n-1}c_{m,m+k}P_{m+k}(\lambda(z))}{U_m(\lambda(z))} &= \sum_{k=0}^{2n-1} \sum_{j=1}^{2n-2} t_k t^\mu_j 
    \frac{U_{m+k-j}(\lambda(z))}{U_m(\lambda(z))}.
    \end{align*}
    By equation \eqref{eqn:chebUmtoinfty}, this tends to $\sum_{k=0}^{2n-1} \sum_{j=1}^{2n-2} t_k t^\mu_j z^{j-k}$ as $m \to \infty$. This is equal to $c_\mu(z)c(z^{-1})$. Equating this with the other equality for equation \eqref{eqn:weirdquantity} gives $p^\mu(\lambda(z)) = c(z^{-1})c_\mu(z) - 2z$ as required.
   \end{proof}
  \end{lemma}
  
  The following theorem describes Theorem \ref{thm:mainresolvent} under the change of variables induced by the Joukowski map. The remarkable thing is that the resolvent is expressible as a rational function inside the unit disc.
    
  \begin{theorem}\label{thm:joukowskiresolvent}
Let $J$ be a Jacobi operator such that there exists an $n$ such that
\begin{equation*}
\alpha_k = 0, \quad \beta_{k-1} = \frac12 \text{ for all } k \geq n,
\end{equation*}
i.e. it is equal to the free Jacobi operator $\Delta$ outside the $n\times n$ principal submatrix. By Theorem \ref{thm:connectioncoeffs} the connection coefficient matrix can be decomposed into $C = C_{\rm Toe} + C_{\rm fin}$. By Corollary \ref{cor:cmu}, we similarly have $C^\mu = C^\mu_{\rm Toe} + C^\mu_{\rm fin}$. If $c$ and $c_\mu$ are the Toeplitz symbols of $C_{\rm Toe}$ and $C^\mu_{\rm Toe}$ respectively, then for $\lambda(z) = \frac12(z+z^{-1})$ with $z\in\mathbb{D}$, the principal resolvent $G$ is given by the rational function
   \begin{equation}\label{eqn:resolventequalscmuoverc}
    G(\lambda(z)) = -\frac{c_\mu(z)}{c(z)}.
   \end{equation}
    \begin{proof}
 Combining Theorem \ref{thm:mainresolvent}, equation \eqref{eqn:GDeltaz} and Lemmata \ref{lem:cformula} and \ref{lem:cmuformula}, we have
  \begin{align*}
   G(\lambda(z)) &= \frac{G_\Delta(\lambda(z)) - 
p^\mu(\lambda(z))}{p(\lambda(z))} \\
                 &= \frac{-2z - (c(z^{-1})c_\mu(z)-2z)}{c(z)c(z^{-1})} \\
                 &= -\frac{c_\mu(z)}{c(z)}.
  \end{align*}
  This completes the proof.
 \end{proof}
 \end{theorem}
 
 The following theorem gives a better description of the weights $w_k$ in Theorem \ref{thm:mainmeasure}, utilising the Joukowski map and the Toeplitz symbol $c$.
 
 \begin{theorem}\label{thm:joukowskimeasure}
 Let $J$ be a Jacobi operator such that there exists an $n$ such that
 \begin{equation*}
 \alpha_k = 0, \quad \beta_{k-1} = \frac12 \text{ for all } k \geq n,
 \end{equation*}
 i.e. it is equal to the free Jacobi operator $\Delta$ outside the $n\times n$ principal submatrix. By Theorem \ref{thm:connectioncoeffs} the connection coefficient matrix can be written $C = C_{\rm Toe} + C_{\rm fin}$. If $c$ is the Toeplitz symbol of $C_{\rm Toe}$, then the spectral measure of $J$ is
   \begin{equation*}
    \mu(s) = \frac{1}{p_C(s)}\mu_\Delta(s) + \sum_{k=1}^r 
\frac{(z_k-z_k^{-1})^2}{z_kc'(z_k)c(z_k^{-1})}\delta_{\lambda(z_k)}(s).
   \end{equation*}
Here $z_k$ are the roots of $c$ that lie in the open unit disk, which are all real and simple. The only roots of $c$ on the unit circle are $\pm 1$, which can also only be 
simple. Further, $r \leq n$.
 \begin{proof}
 By Theorem \ref{thm:mainmeasure},
 \begin{equation*}
  \mu(s) = \frac{1}{p_C(s)}\mu_\Delta(s) + \sum_{k=1}^{r} w_k \delta_{\lambda_k}(s),
 \end{equation*}
 where $r \leq n$. Hence we just need to prove something more specific about the roots of $c$, $\lambda_1,\ldots,\lambda_r$, and $w_1,\ldots,w_r$.
 
By Theorem \ref{thm:joukowskiresolvent}, $G(\lambda(z)) = -c_\mu(z)/c(z)$ for $z \in \mathbb{D}$. By Lemma \ref{lem:cmuformula}, $c(z^{-1})c_\mu(z)-2z = p^\mu(\lambda(z)) = p^\mu(\lambda(z^{-1})) = c(z)c_\mu(z^{-1})-2z^{-1}$, so
\begin{equation}\label{eqn:rootsofc}
 c(z^{-1})c_\mu(z) - c(z)c_\mu(z^{-1}) = 2(z - z^{-1}).
\end{equation}
Therefore $c(z)$ and $c_\mu(z)$ cannot simultaneously be zero unless $z = z^{-1}$, which only happens at $z = \pm 1$. By the same reasoning, $c(z)$ and $c(z^{-1})$ also cannot be simultaneously zero unless $z = \pm 1$. Since the Joukowski map $\lambda$ is a bijection from $\mathbb{D}$ to $\C \setminus [-1,1]$, this shows that the (simple and real) poles of $G$ in $\C \setminus [-1,1]$ are precisely $\lambda(z_1), \ldots,\lambda(z_r)$, where $z_1,\ldots,z_r$ are the (necessarily simple and real) roots of $c$ in $\mathbb{D}$.

What are the values of the weights of the Dirac deltas, $w_1,\ldots,w_r$? By Theorem \ref{thm:mainmeasure},
 \begin{align*}
  w_k &= \lim_{\eps \searrow 0}\frac{\eps}{i} G(\lambda(z_k)+i\eps) \\
                     &= \lim_{\lambda \to \lambda(z_k)} (\lambda(z_k)-\lambda) 
G(\lambda) \\
                     &= \lim_{z \to z_k} \frac12(z_k+z_k^{-1} - z-z^{-1}) (-1) 
\frac{c_\mu(z)}{c(z)} \\
                     &= \lim_{z \to z_k} \frac12 z^{-1}(z-z_k)(z-z_k^{-1}) 
\frac{c_\mu(z)}{c(z)} \\
                     &= \frac12 z_k^{-1}(z_k-z_k^{-1})c_\mu(z_k) \lim_{z \to 
z_k}\frac{(z-z_k)}{c(z)} \\
                     &= \frac12 z_k^{-1}(z_k-z_k^{-1})\frac{c_\mu(z_k)}{c'(z_k)}.
\end{align*}

By equation \eqref{eqn:rootsofc}, since $c(z_k) = 0$, we have $c_\mu(z_k) = 2(z_k - z_k^{-1})/c(z_k^{-1})$. This gives
\begin{equation*}
 w_k = \frac{(z_k - z_k^{-1})^2}{z_k c(z_k^{-1}) c'(z_k)}.
\end{equation*}

Note that if $c(z) = 0$ then $c(\overline{z}) = 0$ because $c$ has real coefficients. If $c$ has a root $z_0$ on the unit circle, then $c(z_0) = c(z_0^{-1}) = 0$ because $\overline{z_0} = z_0^{-1}$, which earlier in the proof we showed only occurs if $z_0 = \pm 1$. Hence $c$ does not have roots on the unit circle except possibly $\pm 1$.
 \end{proof}
  \end{theorem}

 \begin{example}[Basic perturbation 1 re-revisited]\label{ex:basic1rerev}

Considering the connection coefficient matrix in equation \eqref{eqn:basic1C}, we see that the Toeplitz symbol $c$ is $c(z) = 1 - \alpha z$. By Theorem \ref{thm:joukowskimeasure} the roots of $c$ in the unit disc correspond to eigenvalues of $J_\alpha$. As is consistent with our previous considerations, $c$ has a root in the unit disc if and only if $|\alpha| > 1$, and those eigenvalues are $\lambda(\alpha^{-1}) = \frac12(\alpha + \alpha^{-1})$. See Appendix~\ref{sec:numericalexperiments} for figures depicting the spectral measure and the resolvent.
\end{example}

\begin{example}[Basic perturbation 2 re-revisited]\label{ex:basic2rerev}

Considering the connection coefficient matrix in equation \eqref{eqn:basic2C}, we see that the Toeplitz symbol $c$ is $c(z) = \beta^{-1} + (\beta^{-1}-\beta) z^2$. By Theorem \ref{thm:joukowskimeasure} the roots of $c$ in the unit disc correspond to eigenvalues of $J_\beta$. The roots of $c$ are $\pm\frac{1}{\sqrt{\beta^2-1}}$. If $\beta \in \left(0,\sqrt{2}\right]\setminus\{1\}$ then $\left|\pm\frac{1}{\sqrt{\beta^2-1}} \right| \geq 1$ so there are no roots of $c$ in the unit disc, as is consistent with the previous observations. What was difficult to see before is, if $\beta > \sqrt{2}$ then $\left|\pm\frac{1}{\sqrt{\beta^2-1}} \right| < 1$, so there is a root of $c$ inside $\mathbb{D}$, and it corresponds to an eigenvalue,
\begin{equation*}
 \lambda\left(\pm\frac{1}{\sqrt{\beta^2-1}}\right) = \pm\frac12\left( \frac{1}{\sqrt{\beta^2-1}} + \sqrt{\beta^2-1}\right)= \pm\frac{\beta^2}{2\sqrt{\beta^2-1}}.
\end{equation*}
See Appendix \ref{sec:numericalexperiments} for figures depicting the spectral measure and the resolvent.
\end{example}

%% file: toeplitzplustraceclass.tex
\section{Toeplitz-plus-trace-class Jacobi operators}\label{sec:toeplitzplustraceclass}
    
    In this section we extend the results of the previous section to the case where the Jacobi operator is Toeplitz-plus-trace-class. This cannot be done as a direct extension of the work in the previous section as the formulae obtained depended on the fact that some of the functions involved were merely polynomials in order to have a function defined for all $\lambda$ in an a priori known region of the complex plane. We admit that it may be possible to perform the analysis directly, but state that it is not straightforward. We are interested in feasible (finite) computation so are content to deal directly with the Toeplitz-plus-finite-rank case and perform a limiting process. The crucial question for computation is, can we approximate the spectral measure of a Toeplitz-plus-trace-class Jacobi operator whilst reading only finitely many entries of the matrix?
    
    Here we make clear the definition of a Toeplitz-plus-trace-class Jacobi operator.
    
    \begin{definition}
     An operator $K : \ell^2 \to \ell^2$ is said to be trace class if $\sum_{k=0}^{\infty} e_k^T (K^T K)^{1/2} e_k < \infty$. Hence we say that a Jacobi operator $J$ such that $\alpha_k \to 0$, $\beta_k \to \frac12$ as $k \to \infty$ is Toeplitz-plus-trace-class if
    \begin{equation*}
     \sum_{k=0}^\infty \left|\beta_k - \frac12\right| + |\alpha_k| < \infty.
    \end{equation*}
    \end{definition}
    
\subsection{Jacobi operators for Jacobi polynomials}\label{subsec:Jacobipolys}
    
    The most well known class of orthogonal polynomials is the Jacobi polynomials, whose measure of orthogonality is 
\begin{equation*}
 \d\mu(s) = 
\left(2^{\alpha+\beta+1}B(\alpha+1,\beta+1)\right)^{-1}(1-s)^\alpha(1+s)^\beta \bigg|_{s \in [-1,1]} \d s,
\end{equation*}
where $\alpha$,$\beta > -1$ and $B$ is Euler's Beta function. The Jacobi 
operator for the normalised Jacobi polynomials with respect to this probability measure, and hence the three-term recurrence coefficients, are given by \cite{olverDLMF},
 \begin{align*}
  \alpha_k &= \frac{\beta^2-\alpha^2}{(2k+\alpha+\beta)(2k+\alpha+\beta+2)} \\
  \beta_{k-1} &= 
2\sqrt{\frac{k(k+\alpha)(k+\beta)(k+\alpha+\beta)}{
(2k+\alpha+\beta-1)(2k+\alpha+\beta)^2(2k+\alpha+\beta+1)}}
 \end{align*}
  Note that $|\alpha_k| = \mathcal{O}(k^{-2})$ and 
  \begin{equation*}
   \beta_{k-1} = \frac12\sqrt{1+\frac{(4-8\alpha^2-8\beta^2)k^2+\mathcal{O}(k)}{(2k+\alpha+\beta-1)(2k+\alpha+\beta)^2(2k+\alpha+\beta+1)}} = \frac12 + \mathcal{O}(k^{-2}).
  \end{equation*}
Hence the Jacobi operators for the Jacobi polynomials are Toeplitz-plus-trace-class for all $\alpha,\beta > -1$.

The Chebyshev polynomials $T_k$ and $U_k$ discussed in the previous section are specific cases of Jacobi polynomials, with $\alpha,\beta = -\frac12,-\frac12$ for $T_k$ and $\alpha,\beta = \frac12,\frac12$ for $U_k$.

In Appendix \ref{sec:numericalexperiments} numerical computations of the spectral measures and resolvents of these Jacobi operators are presented.

\subsection{Toeplitz-plus-finite-rank approximations}

We propose to use the techniques from Section \ref{sec:Toeplitzplusfinite}. Therefore for a Jacobi operator $J$, we can define the Toeplitz-plus-finite-rank approximations $J^{[m]}$, where
\begin{equation}\label{eqn:Jm}
J^{[m]}_{i,j} = \begin{cases}
J_{i,j} & \text{ if } 0 \leq i,j < m \\
\Delta_{i,j} & \text{ otherwise.}
\end{cases}
\end{equation}
Each Jacobi operator $J^{[m]}$ has a spectral measure $\mu^{[m]}$ which can be computed using Theorem \ref{thm:joukowskimeasure}. The main question for this section is: how do the computable measures $\mu^{[m]}$ approximate the spectral measure $\mu$ of $J$?

\begin{proposition}\label{prop:weakconvergence}
  Let $J$ a Jacobi operator (bounded, but with no assumed structure imposed) and let $\mu$ be its spectral measure. Then the measures $\mu^{[1]},\mu^{[2]},\ldots$ which are the spectral measures of $J^{[1]},J^{[2]},\ldots$ converge to $\mu$ in a weak sense. Precisely,
  \begin{equation*}
  \lim_{m \to \infty} \int f(s) \,\d\mu^{[m]}(s) = \int f(s) \,\d\mu(s),
  \end{equation*}
  for all $f \in C_b(\mathbb{R})$.
  \begin{proof}
    Each spectral measure $\mu^{[m]}$ and $\mu$ are supported on the spectra of $J^{[m]}$ and $J$, each of which are contained within $[-\|J^{[m]}\|_2,\|J^{[m]}\|_2]$ and $[-\|J\|_2,\|J\|_2]$. Since $\|J^{[m]}\|_2$ and $\|J\|_2$ are less than
    \begin{equation*}
    M = 3\left(\sup_{k \geq 0} |\alpha_k| + \sup_{k \geq 0} |\beta_k|\right),
    \end{equation*}
    we have that all the spectral measures involved are supported within the interval $[-M,M]$. Hence we can consider integrating functions $f \in C([-M,M])$ without ambiguity.
    
    By Weierstrass' Theorem, polynomials are dense in $C([-M,M])$, so we only need to consider polynomials as test functions, and by linearity we only need to consider the orthogonal polynomials for $J$. The first polynomial $P_0$ has immediate convergence, since the measures are all probability measures. Now consider $P_k$ for some $k>0$, which satisfies $\int P_k(s) \,\d\mu(s) = 0$. For $m > k$, $P_k$ is also the $k$th orthogonal polynomial for $J^{[m]}$, hence $\int P_k(s) \,\d\mu^{[m]}(s) = 0$. This completes the proof.
  \end{proof}
\end{proposition}

\subsection{Asymptotics of the connection coefficients}

Here we formulate a lower triangular block operator equation $\mathcal{L}\underline{c} = e_0^0$ , where $e_0^0 = (e_0, 0, 0, \ldots)^\top$, satisfied by the entries of the connection coefficient matrices encoded into a vector $\underline{c}$. For Toeplitz-plus-trace-class Jacobi operators we give appropriate Banach spaces upon which the operator $\mathcal{L}$ is bounded and invertible, enabling precise results about the asymptotics of the connection coefficients to be derived.

\begin{lemma}\label{lem:systemforti}
 Let $J$ and $D$ be Jacobi operators with entries $\{\alpha_k,\beta_k\}_{k=0}^\infty$ and $\{\gamma_k,\delta_k\}_{k=0}^\infty$ respectively. If we decompose the upper triangular part of $C_{J\to D}$ into a sequence of sequences, stacking each diagonal on top of each other, we get the following block linear system,
\begin{equation}\label{eqn:infblockinfsystem}
 \left(\begin{array}{ccccc} B_{-1} &  & & & \\ 
                           A_0 & B_0 &  & & \\ 
                           B_0^T & A_1 & B_1 & &\\ 
                            & B_1^T & A_2 & B_2 & \\ 
                           &  & \sddots & \sddots  & \sddots
                           \end{array}\right) 
                           \left(\begin{array}{c} c_{*,*} \\ c_{*,*+1} \\ c_{*,*+2} \\ c_{*,*+3} \\ \vdots \end{array}\right) = \left(\begin{array}{c} e_0 \\ 0 \\ 0 \\ 0 \\ \vdots \end{array}\right),
\end{equation}
where for each $i$,
\begin{equation*}
B_i = 2\left(\begin{array}{cccc} \beta_i &  & & \\ 
        -\delta_0 & \beta_{i+1} &  & \\
         & -\delta_1 & \beta_{i+2} &  \\
         &  & \sddots & \sddots
       \end{array}\right), \quad A_i = 2\left(\begin{array}{cccc} \alpha_i- \gamma_0 &  & & \\ 
         &\alpha_{i+1}- \gamma_{1} &  & \\
         &  & \alpha_{i+2}- \gamma_{2} & \\
         &  &  & \sddots
       \end{array}\right).
\end{equation*}
For $B_{-1}$ to make sense we define $\beta_{-1} = 1/2$.
\begin{proof}
 This is simply the 5-point discrete system in Lemma \ref{lem:5ptsystem} rewritten.
\end{proof}
\end{lemma}

We write the infinite-dimensional-block-infinite-dimensional system \eqref{eqn:infblockinfsystem} in the form,
\begin{equation}\label{eqn:LC}
  \mathcal{L} \underline{c} = e_0^0.
\end{equation}
For general Jacobi operators $J$ and $D$, the operators $A_i$ and $B_i$ are well defined linear operators from $\ell_{\mathcal{F}}^\star$ to $\ell_{\mathcal{F}}^\star$. The block operator $\mathcal{L}$ is whence considered as a linear operator from the space of sequences of real sequences, $\ell_{\mathcal{F}}^\star(\ell_{\mathcal{F}}^\star)$ to itself. We will use this kind of notation for other spaces as follows.

\begin{definition}[Vector-valued sequences]\label{def:vecvalsec}
If $\ell_X$ is a vector space of scalar-valued sequences, and $Y$ is another vector space then we let $\ell_X(Y)$ denote the vector space of sequences of elements of $Y$. In many cases in which $\ell_X$ and $Y$ are both normed spaces, then $\ell_X(Y)$ naturally defines a normed space in which the norm is derived from that of $\ell_X$ by replacing all instances of absolute value with the norm on $Y$. For example, $\ell^p(\ell^\infty)$ is a normed space with norm $\|(a_k)_{k=0}^\infty\|_{\ell^p(\ell^\infty)} = \left(\sum_{k=0}^\infty \|a_k\|_\infty^p\right)^{\frac1{p}}$.
\end{definition}

The following two spaces are relevant for the Toeplitz-plus-trace-class Jacobi operators.

\begin{definition}[Sequences of bounded variation]
 Following \cite[Ch.~IV.2.3]{dunford1971linear}, denote by $bv$ the Banach space of all sequences with bounded variation, that is sequences such that the norm
\begin{equation*}
 \|a\|_{bv} = |a_0| + \sum_{k=0}^\infty |a_{k+1}-a_k|,
\end{equation*}
is finite.
\end{definition}

The following result is immediate from the definition of the norm on $bv$.
\begin{lemma}\label{lem:bvintolinfty}
 There is a continuous embedding of $bv$ into the Banach space of convergent sequences (endowed with the supremum norm) i.e. for all $(a_k)_{k=0}^\infty \in bv$, $\lim_{k\to\infty} a_k$ exists, and $\sup_{k}|a_k| \leq \|(a_k)_{k=0}^\infty\|_{bv}$. Furthermore, $\lim_{k\to\infty} |a_k| \leq \|a\|_{bv}$.
\end{lemma}

\begin{definition}[Geometrically weighted $\ell^1$]
For any $R > 0$, e define the Banach space $\ell^1_R$ to be the space of sequences such that the norm
\begin{equation*}
 \|v\|_{\ell_R^1} = \sum_{k=0}^\infty R^k |v_k|,
\end{equation*}
is finite.
\end{definition}

\begin{proposition}\label{prop:lR1norm}
The operator norm on $\ell^1_R$ is equal to
\begin{equation*}
 \|A\|_{\ell_R^1 \to \ell_R^1} = \sup_{j} \sum_{i} R^{i-j} |a_{ij}|.
\end{equation*}
\end{proposition}

The following Lemma and its Corollary show that it is natural to think of $\underline{c}$ as lying in the space $\ell^1_R(bv)$.
\begin{lemma}\label{lem:bigL}
 Let $J = \Delta + K$ be a Jacobi operator where $K$ is trace class and let $D = \Delta$. Then for any $R \in (0,1)$ the operator $\mathcal{L}$ in equation \eqref{eqn:LC} is bounded and invertible as an operator from $\ell_R^1(bv)$ to $\ell^1_R(\ell^1)$. Furthermore, if $\mathcal{L}^{[m]}$ is the operator in equation \eqref{eqn:LC} generated by the Toeplitz-plus-finite-rank truncation $J^{[m]}$, then
 \begin{equation*}
 \|\mathcal{L}-\mathcal{L}^{[m]}\|_{\ell^1_R(bv)\to\ell^1_R(\ell^1)} \to 0 \text{ as } m \to \infty.
 \end{equation*}
 \begin{proof}
 We can write $\mathcal{L}$ in equation \eqref{eqn:LC} in the form $\mathcal{L} = \mathcal{T} + \mathcal{K}$ where
  \begin{equation*}
  \mathcal{T} = \left(\begin{array}{ccccc} T &  & & &\\ 
  0 & T &  && \\
  T^T & 0 & T & &\\
  & T^T & 0 & T & \\
  &  & \sddots & \sddots & \sddots
  \end{array}\right), \quad T = \left(\begin{array}{cccc}
  1 & & & \\
  -1 & 1 & & \\
  & -1 & 1 & \\
  & & \sddots & \sddots
  \end{array} \right),
  \end{equation*}
  and
  \begin{equation*}
  \mathcal{K} = \left( \begin{array}{ccccc} K_{-1} &  & & & \\ 
  A_0 & K_0 & & & \\
  K_0 & A_1 & K_1 & & \\
  & K_1 & A_2 & K_2 & \\
  &  & \sddots & \sddots & \sddots
  \end{array}\right), \quad \begin{array}{c} A_i = 2\diag(\alpha_i,\alpha_{i+1},\ldots), \\ K_i = \diag(2\beta_i-1, 2\beta_{i+1}-1,\ldots). \end{array} 
  \end{equation*}
  
  This decomposition will allow us to prove that $\mathcal{L}$ is bounded and invertible as follows. We will show that as operators from $\ell^1_R(bv)$ to $\ell^1_R(\ell^1)$, $\mathcal{T}$ is bounded and invertible, and $\mathcal{K}$ is compact. This implies that $\mathcal{L}$ is a Fredholm operator with index $0$. Therefore, by the Fredholm Alternative Theorem, $\mathcal{L}$ is invertible if and only if it is injective. It is  indeed injective, because it is block lower triangular with invertible diagonal blocks, so forward substitution on the system $\mathcal{L} \underline{v} = \underline{0}$ implies that each entry of $\underline v$ must be zero.
  
First let us prove that $\mathcal{T}$ is bounded and invertible. It is elementary that $T$ is an isometric isomorphism from $bv$ to $\ell^1$ and $T^T$ is bounded with norm at most 1. Hence using Proposition \ref{prop:lR1norm} we have
\begin{equation*}
 \|\mathcal{T}\|_{\ell^1_R(bv) \to \ell^1_R(\ell^1)} = R^0\|T\|_{bv \to \ell^1} + R^2 \|T^T\|_{bv \to \ell^1} \leq 1 + R^2.
\end{equation*}
Because each operator is lower triangular, the left and right inverse of $\mathcal{T} : \ell_{\mathcal{F}}(\ell_{\mathcal{F}}) \rightarrow \ell_{\mathcal{F}}(\ell_{\mathcal{F}})$ is
\begin{equation*}
\mathcal{T}^{-1} = \left(\begin{array}{cccccc}
        T^{-1} & & & & & \\
        0 & T^{-1} & & & & \\
        -T^{-1}T^T T^{-1} & 0 & T^{-1} & & & \\
        0 & -T^{-1}T^T T^{-1} & 0 & T^{-1} & & \\
        T^{-1}(T^T T^{-1})^2 & 0 & -T^{-1}T^T T^{-1} & \sddots & \sddots & \\
        \vdots & \sddots & \sddots & \sddots & \sddots & \sddots
       \end{array}\right).
\end{equation*}
This matrix is block-lower triangular and block-Toeplitz with first column having $2i$th block of the form $T^{-1}(-T^T T^{-1})^i$ and $(2i+1)$th block zero. We must check that $\mathcal{T}^{-1}$ is bounded in the norms on $\ell_R^1(\ell^1)$ to $\ell_R^1(bv)$ so that it may be extended to $\ell_R^1(\ell^1)$ from the dense subspace $\ell_{\mathcal{F}}$. Again using Proposition \ref{prop:lR1norm} we have
\begin{align*}
 \|\mathcal{T}^{-1}\|_{\ell_R^1(\ell^1)\to\ell_R^1(bv)} &= \sup_j \sum_{i=j}^\infty R^{2(i-j)}\|T^{-1}(-T^T T^{-1})^{i-j}\|_{\ell^1\to bv} \\
 &= \sum_{k=0}^\infty R^{2k} \|T^{-1}(-T^T T^{-1})^k\|_{\ell^1\to bv} \\
 &\leq \sum_{k=0}^\infty R^{2k} \|T^{-1}\|_{\ell^1\to bv}\left(\|T^T\|_{bv \to \ell^1}\|T^{-1}\|_{\ell^1\to bv}\right)^k \\
 &\leq \sum_{k=0}^\infty R^{2k}  = (1-R^2)^{-1} < \infty.
\end{align*}

Now let us prove that $\mathcal{K} : \ell_R^1(bv) \to \ell_R^1(\ell^1)$ is compact. Consider the finite rank operator $\mathcal{K}^{[m]}$, where all elements are the same as in $\mathcal{K}$, except that all occurrences of $\alpha_i$ and $2\beta_i - 1$ are replaced by $0$ for $i \geq m$. Using Proposition \ref{prop:lR1norm} we have 
\begin{equation*}
 \|\mathcal{K}-\mathcal{K}^{[m]}\|_{\ell_R^1(bv) \to \ell_R^1(\ell^1)} = \sup_j R^0\|K_{j-1} - K^{[m]}_{j-1}\|_{bv \to \ell^1} + R^1\|A_j - A^{[m]}_{j}\|_{bv \to \ell^1} + R^2\|K_j - K^{[m]}_j\|_{bv \to \ell^1}.
\end{equation*}
By the continuous embedding in Lemma \ref{lem:bvintolinfty}, $\|\cdot\|_{bv \to \ell^1} \leq \|\cdot\|_{\ell^\infty \to \ell^1}$. Hence
\begin{align*}
 \|\mathcal{K}-\mathcal{K}^{[m]}\|_{\ell_R^1(bv) \to \ell_R^1(\ell^1)} &\leq \sum_{k=m}^\infty R^0|2\beta_{k-1} -1| + R^1|\alpha_k| + R^2|2\beta_k - 1| \\
                                                                       &\to 0 \text{ as } m \to \infty.
\end{align*}
This convergence is due to the fact that $J-\Delta$ is trace class. Since $\mathcal{K}$ is a norm limit of finite rank operators it is compact. This completes the proof that $\mathcal{L}$ is bounded and invertible.

Now consider the operator $\mathcal{L}^{[m]}$ defined in the statement of the Lemma, which is equal to $\mathcal{T} + \mathcal{K}^{[m]}$ (where $\mathcal{K}^{[m]}$ is precisely that which was considered whilst proving $\mathcal{K}$ is compact). Hence,
\begin{equation*}
\|\mathcal{L}-\mathcal{L}^{[m]}\|_{\ell^1_R(bv)\to\ell^1_R(\ell^1)} = \|\mathcal{K}-\mathcal{K}^{[m]}\|_{\ell^1_R(bv)\to\ell^1_R(\ell^1)} \to 0 \text{ as } m \to \infty.
\end{equation*}
This completes the proof.
  \end{proof}
 \end{lemma}

\begin{corollary}\label{cor:cvecconverge}
  Let $J = \Delta + K$ be a Jacobi operator where $K$ is trace class and let $\underline{c} \in \ell_{\mathcal{F}}^\star(\ell_{\mathcal{F}}^\star)$ be the vector of diagonals of $C_{J\to\Delta}$ as in equation \eqref{eqn:LC}. Then $\underline{c} \in \ell^1_R(bv)$. If $J$ has Toeplitz-plus-finite-rank approximations $J^{[m]}$ and $\underline{c}^{[m]}$ denotes the vector of diagonals of $C^{[m]}$, then
  \begin{equation*}
  \|\underline{c} - \underline{c}^{[m]}\|_{\ell^1_R(bv)} \to 0 \text{ as } m \to \infty.
  \end{equation*}
  \begin{proof}
    
    By equation \eqref{eqn:LC} 
    \begin{equation*}
    \underline{c} - \underline{c}^{[m]}  = (\mathcal{L}^{-1} - (\mathcal{L}^{[m]})^{-1})e_0^0.
    \end{equation*}
    Since $\|e_0^0\|_{\ell_R^1(\ell^1)} = 1$, the proof is completed if we show $\|\mathcal{L}^{-1} - (\mathcal{L}^{[m]})^{-1}\|_{\ell_R^1(\ell^1) \to \ell_R^1(bv)} \to 0$ as $m \to \infty$.
    
    Suppose that $m$ is sufficiently large so that $\|\mathcal{L} - \mathcal{L}^{[m]}\| < \|\mathcal{L}^{-1}\|^{-1}$ (guaranteed by Lemma \ref{lem:bigL}). Note that $\mathcal{L}^{-1}$ is bounded by the Inverse Mapping Theorem and Lemma \ref{lem:bigL}. Then by a well-known result (see for example, \cite{atkinson2005theoretical}),
    \begin{equation*}
    \|\mathcal{L}^{-1} - (\mathcal{L}^{[m]})^{-1}\| \leq \frac{\|\mathcal{L}^{-1}\|^2 \|\mathcal{L} - \mathcal{L}^{[m]}\|}{1-\|\mathcal{L}^{-1}\|\|\mathcal{L}-\mathcal{L}^{[m]}\|}.
    \end{equation*}
    This tends to zero as $m \to \infty$, by Lemma \ref{lem:bigL}.
  \end{proof}
\end{corollary}

 \begin{theorem}\label{thm:traceclassC}
   Let $J = \Delta + K$ be a Jacobi operator where $K$ is trace class. Then $C = C_{J\to\Delta}$ can be decomposed into 
   \begin{equation*}
   C = C_{\rm Toe} + C_{\rm com},
   \end{equation*}
   where $C_{\rm Toe}$ is upper triangular, Toeplitz and bounded as an operator from $\ell^1_{R}$ to $\ell^1_{R}$, and $C_{\rm com}$ is compact as an operator from $\ell^1_{R}$ to $\ell^1_{R}$, for all $R > 1$. Also, if $J$ has Toeplitz-plus-finite-rank approximations $J^{[m]}$ with connection coefficient matrices $C^{[m]}=C_{\rm Toe}^{[m]}+C_{\rm com}^{[m]}$, then
   \begin{equation*}
    C^{[m]} \to C, \quad C_{\rm Toe}^{[m]} \to C_{\rm Toe}, \quad C_{\rm com}^{[m]} \to C_{\rm com} \text{ as } m\to \infty,
   \end{equation*}
   in the operator norm topology over $\ell^1_{R}$ for all $R>1$.
 \begin{proof}
  By Lemma \ref{lem:bigL}, for each $k$ the sequence $(c_{0,0+k},c_{1,1+k},c_{2,2+k},\ldots)$ is an element of $bv$. By Lemma \ref{lem:bvintolinfty} each is therefore a convergent sequence, whose limits we call $t_k$. Hence we can define an upper triangular Toeplitz matrix $C_{\rm Toe}$ whose $(i,j)$th element is $t_{j-i}$, and define $C_{\rm com} = C-C_{\rm Toe}$.
  
  The Toeplitz matrix $C_{\rm Toe}$ is bounded from $\ell^1_{R}$ to $\ell^1_{R}$ for all $R>1$ by the following calculation.
  \begin{align*}
  \|C_{\rm Toe}\|_{\ell^1_{R}\to\ell^1_{R}} &= \sup_j \sum_{i=0}^j R^{i-j}|t_{j-i}| \\
                                                  &= \sum_{k=0}^\infty R^{-k}|t_k| \\
                                                  &\leq \sum_{k=0}^\infty R^{-k} \|c_{*,*+k}\|_{bv}  \\
                                                  &= \|\underline{c}\|_{\ell^1_{R^{-1}}(bv)}.
  \end{align*}
  By Lemma \ref{lem:bigL} this quantity is finite (since $R^{-1} \in (0,1)$).
  
  Now we show convergence results. The compactness of $C_{\rm com}$ will follow at the end. For all $R>1$,
  \begin{align*} 
  \|C-C^{[m]}\|_{\ell^1_{R}\to\ell^1_{R}} &= \sup_j \sum_{i=0}^j R^{i-j} |c_{i,j}-c_{i,j}^{[m]}| \\
                &= \sup_j \sum_{k=0}^j R^{-k} |c_{j-k,j} - c_{j-k,j}^{[m]}| \\
                &\leq \sup_j \sum_{k=0}^j R^{-k} \|c_{*,*+k} - c^{[m]}_{*,*+k}\|_{bv} \\
                &= \sum_{k=0}^\infty R^{-k} \|c_{*,*+k} - c^{[m]}_{*,*+k}\|_{bv} \\
                &= \|\underline{c}-\underline{c}^{[m]}\|_{\ell^1_{R^{-1}}(bv)}.
  \end{align*}
  For the third line of the above sequence of equations, note that for fixed $k$, $c_{0,k}-c^{[m]}_{0,k},c_{1,1+k}-c^{[m]}_{1,1+k},c_{2,2+k}-c^{[m]}_{2,2+k},\ldots$ is a $bv$ sequence, and refer to Lemma \ref{lem:bvintolinfty}.
  \begin{align*}
  \|C_{\rm Toe}-C_{\rm Toe}^{[m]}\|_{\ell^1_{R}\to\ell^1_{R}} &= \sup_j \sum_{i=0}^j R^{i-j} |t_{j-i}-t_{j-i}^{[m]}| \\
  &=\sum_{k=0}^j R^{-k} |t_k -t_k^{[m]}| \\
  &\leq \sum_{k=0}^\infty R^{-k} \|c_{*,*+k} - c^{[m]}_{*,*+k}\|_{bv} \\
  &= \|\underline{c}-\underline{c}^{[m]}\|_{\ell^1_{R^{-1}}(bv)}.
  \end{align*}
  For the third line of the above sequence, note that $t_k -t_k^{[m]}$ is the limit of the $bv$ sequence $c_{*,*+k} - c^{[m]}_{*,*+k}$, and refer to Lemma \ref{lem:bvintolinfty}.
  \begin{align*}
  \|C_{\rm com}-C_{\rm com}^{[m]}\|_{\ell^1_{R}\to\ell^1_{R}} \leq \|C-C^{[m]}\| + \|C_{\rm Toe}-C_{\rm Toe}^{[m]}\| \leq 2\|\underline{c}-\underline{c}^{[m]}\|_{\ell^1_{R^{-1}}(bv)}.
  \end{align*}
  Using Corollary \ref{cor:cvecconverge}, that $\|\underline{c}-\underline{c}^{[m]}\|_{\ell^1_R(bv)} \to 0$ as $m\to\infty$, we have the convergence results.
  
  By Theorem \ref{thm:connectioncoeffs}, $C_{\rm com}^{[m]}$ has finite rank. Therefore, since $C_{\rm com} = \lim_{m\to\infty} C_{\rm com}^{[m]}$ in the operator norm topology over $\ell^1_{R^{-1}}$, we have that $C_{\rm com}$ is compact in that topology.
 \end{proof}
 \end{theorem}
 \begin{corollary}\label{cor:traceclassCmu}
   Let $C^\mu$ be as defined in Definition \ref{def:muconnection} for $C$ as in Theorem \ref{thm:traceclassC}. Then $C^\mu$ can be decomposed into $C^\mu = C^\mu_{\rm Toe} + C^\mu_{\rm com}$ where $C^\mu_{\rm Toe}$ is upper triangular, Toeplitz and bounded as an operator from $\ell^1_{R}$ to $\ell^1_{R}$, and $C^\mu_{\rm com}$ is compact as an operator from $\ell^1_{R}$ to $\ell^1_{R}$, for all $R > 1$. Furthermore, if $J$ has Toeplitz-plus-finite-rank approximations $J^{[m]}$ with connection coefficient matrices $(C^\mu)^{[m]}=(C^\mu_{\rm Toe})^{[m]}+(C^\mu_{\rm com})^{[m]}$, then
   \begin{equation*}
   (C^\mu)^{[m]} \to C^\mu, \quad (C_{\rm Toe}^\mu)^{[m]} \to C^\mu_{\rm Toe}, \quad (C^\mu_{\rm com})^{[m]} \to C^\mu_{\rm com} \text{ as } m\to \infty,
   \end{equation*}
   in the operator norm topology over $\ell^1_{R}$.
   \begin{proof}
    This follows from Theorem \ref{thm:traceclassC} applied to $J^\mu$ as defined in Lemma \ref{lem:Jmu}.
   \end{proof}
  \end{corollary}
 
 \begin{theorem}\label{thm:symbolsconverge}
Let $J$ be a Jacobi operator such that $J = \Delta + K$ where $K$ is trace class. The Toeplitz symbols $c$ and $c_\mu$ of the Toeplitz parts of $C_{J\to \Delta}$ and $C^\mu_{J\to\Delta}$ are both analytic in the unit disc. Furthermore, if $J$ has Toeplitz-plus-finite-rank approximations $J^{[m]}$ with Toeplitz symbols $c^{[m]}$ and $c_\mu^{[m]}$, then $c^{[m]} \to c$ and $c^{[m]}_\mu \to c_\mu$ as $m \to \infty$ uniformly on compact subsets of $\mathbb{D}$.
\begin{proof}
  Let $R > 1$, and let $0 \leq |z| \leq R^{-1} < 1$. Then by Lemma \ref{lem:bvintolinfty} we have
  \begin{align*}
  \left|\sum_{k=0}^\infty t_k z^k \right| &\leq \sum_{k=0}^\infty |t_k| R^{-k} \leq \sum_{k=0}^\infty \|c_{*,*+k}\|_{bv} R^{-k} = \|\underline{c}\|_{\ell^1_{R^{-1}}(bv)},
  \end{align*}
  where $\underline{c}$ is as defined in equation \eqref{eqn:LC}. By Lemma \ref{lem:bigL} this quantity is finite. Since $R$ is arbitrary, the radius of convergence of the series is 1. The same is true for $c_\mu$ by Lemma \ref{lem:Jmu}.
  
  Now we prove that the Toeplitz symbols corresponding to the Toeplitz-plus-finite-rank approximations converge.
  \begin{align*}
  \sup_{|z|\leq R^{-1}} |c(z)-c^{[m]}(z)| &= \sup_{|z|\leq R^{-1}} \left|\sum_{k=0}^\infty (t_k-t_k^{[m]})z^k \right| \\
  &\leq \sum_{k=0}^\infty |t_k - t^{[m]}_{k}| R^{-k} \\
  &\leq \sum_{k=0}^\infty \|c_{*,*+k} - c^{[m]}_{*,*+k}\|_{bv}R^{-k} 
  = \|\underline{c}-\underline{c}^{[m]}\|_{\ell_{R^{-1}}^1(bv)},
  \end{align*}
  To go between the first and second lines, note that for each $k$, $c_{*,*+k} - c^{[m]}_{*,*+k}$ is a $bv$ sequence whose limit is $t_k - t^{[m]}_{k}$ and refer to Lemma \ref{lem:bvintolinfty}. Now, $\|\underline{c}-\underline{c}^{[m]}\|_{\ell_{R^{-1}}^1(bv)} \to 0$ as $m \to \infty$ by Corollary \ref{cor:cvecconverge}. The same is true for $\sup_{|z|\leq R^{-1}} |c_\mu(z)-c_\mu^{[m]}(z)|$ by Lemma \ref{lem:Jmu}.
\end{proof}
 \end{theorem}

\begin{theorem}[See \cite{kato1995perturbation}]\label{thm:specpert}
  Let $A$ and $B$ be bounded self-adjoint operators on $\ell^2$. Then
  \begin{equation*}
  \mathrm{dist}(\sigma(A),\sigma(B)) \leq \|A-B\|_2.
  \end{equation*}
\end{theorem}

\begin{theorem}\label{thm:traceclasszresolvent}
 Let $J = \Delta + K$ be a Toeplitz-plus-trace-class Jacobi operator, and let $c$ and $c_\mu$ be the analytic functions as defined in Theorem \ref{thm:traceclassC} and Corollary \ref{cor:traceclassCmu}. Then for $\lambda(z) = \frac12(z+z^{-1})$ with $z\in\mathbb{D}$ such that $\lambda(z) \notin \sigma(J)$, the principal resolvent $G$ is given by the meromorphic function
 \begin{equation}\label{eqn:Gequalscmuoverctraceclass}
  G(\lambda(z)) = -\frac{c_\mu(z)}{c(z)}.
 \end{equation}
 Therefore, all eigenvalues of $J$ are of the form $\lambda(z_k)$, where $z_k$ is a root of $c$ in $\mathbb{D}$.
\begin{proof}
  Let $z \in \mathbb{D}$ such that $\lambda(z) \notin \sigma(J)$, and let $J^{[m]}$ denote the Toeplitz-plus-finite-rank approximations of $J$ with principal resolvents $G^{[m]}$. Then $J^{[m]} \to J$ as $m\to\infty$, so by Theorem \ref{thm:specpert} there exists $M$ such that for all $m\geq M$, $\lambda(z) \notin \sigma(J^{[m]})$. For such $m$, both $G(\lambda(z))$ and $G^{[m]}(\lambda(z))$ are well defined, and using a well-known result on the difference of inverses (see for example, \cite{atkinson2005theoretical}), we have
  \begin{align*}
     G^{[m]}(\lambda) - G(\lambda) &= \left\langle e_0,\left((J^{[m]} - \lambda)^{-1} - (J-\lambda)^{-1}\right)e_0\right\rangle \\
                               &\leq \|(J^{[m]} - \lambda)^{-1} - (J-\lambda)^{-1}\|_2\\
                              &\leq \frac{\|(J-\lambda)^{-1}\|_2^2\|J-J^{[m]}\|_2}{1-\|(J-\lambda)^{-1}\|_2\|J-J^{[m]}\|_2}\\
                              &\to 0 \text{ as } m \to \infty.
    \end{align*}
  
  Theorem \ref{thm:symbolsconverge} shows that $\lim_{m\to\infty} c^{[m]}_\mu(z)/c^{[m]}(z) = c_\mu(z)/c(z)$. Therefore by Theorem \ref{thm:joukowskiresolvent} these limits are the same and we have equation \eqref{eqn:Gequalscmuoverctraceclass}.
\end{proof}
\end{theorem}

%% file: computability.tex
\section{Computability aspects}\label{sec:computability}

In this section we discuss computability questions \`a la Ben-Artzi--Colbrook--Hansen--Nevanlinna--Seidel \cite{ben2015can,ben2015new,hansen2011solvability}. This involves an informal definition of the Solvability Complexity Index (SCI), a recent development that rigorously describes the extent to which various scientific computing problems can be solved. It is in contrast to classical computability theory \`a la Turing, in which problems are solvable \emph{exactly} in finite time. In scientific computing we are often interested in problems which we can only \emph{approximate} the solution in finite time, such that in an ideal situation this approximation can be made as accurate as desired. For example, the solution to a differential equation, the roots of a polynomial, or the spectrum of a linear operator.

Throughout this section we will consider only real number arithmetic, and the results do not necessarily apply to algorithms using floating point arithmetic.

The Solvability Complexity Index (SCI) has a rather lengthy definition, but in the end is quite intuitive \cite{ben2015can}.

\begin{definition}[Computational problem]
  A computational problem is a 4-tuple, $\{\Xi,\Omega,\Lambda,\mathcal{M}\}$, where $\Omega$ is a set, called the \emph{input set}, $\Lambda$ is a set of functions from $\Omega$ into the complex numbers, called the \emph{evaluation set}, $\mathcal{M}$ is a metric space, and $\Xi : \Omega \to \mathcal{M}$ is the \emph{problem function}.
  \end{definition}

\begin{definition}[General Algorithm]
  Given a computational problem  $\{\Xi,\Omega,\Lambda,\mathcal{M}\}$, a general algorithm is a function $\Gamma : \Omega \to \mathcal{M}$ such that for each $A \in \Omega$, 
  \begin{enumerate}[(i)]
    \item The action of $\Gamma$ on $A$ only depends on the set $\{f(A) : f \in \Lambda_\Gamma(A) \}$ where $\Lambda_\Gamma(A)$ is a finite subset of $\Lambda$
    \item For every $B \in \Omega$, $f(B) = f(A)$ for all $f \in \Lambda_\Gamma(A)$ implies $\Lambda_\Gamma(B) = \Lambda_\Gamma(A)$.
    \end{enumerate}
  \end{definition}
This definition of an algorithm is very general indeed. There are no assumptions on \emph{how} $\Gamma$ computes its output; requirement $(i)$ ensures that it can only use a finite amount of information about its input, and requirement $(ii)$ ensures that the algorithm will only be affected by changes in the inputs which are actually measured. In short, $\Gamma$ depends only on, and is determined by, finitely many evaluable elements of each input.

\begin{definition}\label{def:SCI}[Solvability Complexity Index]
  A computational problem function $\{\Xi,\Omega,\Lambda,\mathcal{M}\}$ has Solvability Complexity Index $k$ if $k$ is the smallest integer such that, for each $(n_1,\ldots,n_k) \in \mathbb{N}^k$ there exists a general algorithm $\Gamma_{n_1,\ldots,n_k} : \Omega \to \mathcal{M}$, such that for all $A \in \Omega$,
 \begin{equation*}
  \Gamma(A) = \lim_{n_k \to \infty} \lim_{n_{k-1}\to\infty}\ldots\lim_{n_1 \to \infty} \Gamma_{n_1,\ldots,n_k}(A),
 \end{equation*}
 where the limit is taken in the metric on $\mathcal{M}$. In other words, the output of $\Gamma$ can be computed using a sequence of $k$ limits.
\end{definition}

We require a metric space for the SCI.
\begin{definition}\label{def:Hausdorff}
  The Hausdorff metric for two compact subsets of the complex plane $A$ and $B$ is defined to be
  \begin{equation*}
  d_H(A,B) = \mathrm{max} \left\{  \sup_{a \in A} \mathrm{dist}(a,B), \sup_{b \in B} \mathrm{dist}(b,A) \right\}.
  \end{equation*}
  If a sequence of sets $A_1,A_2,A_3,\ldots$ converges to $A$ in the Hausdorff metric, we write $A_n \xrightarrow{H} A$ as $n \to \infty$.
\end{definition}

The computational problems considered for the remainder of this paper have $\Omega$ as a set of bounded self-adjoint operators on $\ell^2$, $\Lambda$ is the set of functions which simply return each individual element of the matrix representation of the operator, $\mathcal{M}$ is the set of subsets of $\R$ equipped with the Hausdorff metric, and $\Xi$ returns the spectrum of the operator.

\begin{theorem}[\cite{ben2015can}]\label{thm:SCIgeneral}
 The Solvability Complexity Index of the problem of computing the spectrum of a self-adjoint operator $A \in \mathcal{B}(\ell^2)$ is equal to 2 with respect to the Hausdorff metric on $\mathbb{R}$. For compact operators and banded self-adjoint operators the SCI reduces to 1.
\end{theorem}

Theorem \ref{thm:SCIgeneral} implies that the SCI of computing the spectrum of bounded Jacobi operators in the Hausdorff metric is 1. In loose terms, the problem is solvable using only one limit of computable outputs. What more can we prove about the computability?

The results of Section \ref{sec:Toeplitzplusfinite} reduce the computation of the spectrum of a Toeplitz-plus-finite-rank Jacobi operator to finding the roots of a polynomial. From an uninformed position, one is lead to believe that polynomial rootfinding is a solved problem, with many standard approaches used every day. One common method is to use the QR algorithm to find the eigenvalues of the companion matrix for the polynomial. This can be done stably and efficiently in practice \cite{aurentz2015fast}. However, the QR algorithm is not necessarily convergent for non-normal matrices (companion matrices are normal if and only if they are unitary, which is exceptional). Fortunately, the SCI of polynomial rootfinding with respect to the Hausdorff metric in for subsets of $\C$ is 1, but if one requires the multiplicities of these roots then the SCI is not yet known \cite{ben2015can}.

A globally convergent polynomial rootfinding algorithm is given in \cite{hubbard2001find}. For any degree $d$ polynomial the authors describe a procedure guaranteed to compute fewer than $1.11 d(\log d)^2$ points in the complex plane, such that for each root of the polynomial, a Newton iteration starting from at least one of these points will converge to this root.

Let $\eps > 0$. If a polynomial $p$ of degree $d$ has $r$ roots, how do we know when to stop so that we have $r$ points in the complex plane each within $\eps$ of a distinct root of $p$? This leads us to the concept of error control.
  
  \begin{definition}\label{def:errorcontrol}[Error control]
 A function $\Gamma$ which takes inputs to elements in a metric space $\mathcal{M}$ is computable with error control if it has solvability complexity index 1, and for each $\eps$ we can compute $n$ to guarantee that
 \begin{equation*}
  d_{\mathcal{M}}(\Gamma_{n}(A),\Gamma(A)) < \eps. 
 \end{equation*}
In other words, the output of $\Gamma$ can be computed using a single limit, and an upper bound for the error committed by each $\Gamma_n$ is also computable.
\end{definition}

 Besides providing $\mathcal{O}(d(\log d)^2)$ initial data for the Newton iteration (to find the complex roots of a degree $d$ polynomial), the authors of \cite{hubbard2001find} discuss stopping criteria. In Section 9 of \cite{hubbard2001find}, it is noted therein that for Newton iterates $z_1, z_2, \ldots$, if $|z_{k}-z_{k-1}| < \eps/d$, then there exists a root $\xi$ of the polynomial in question such that $|z_k-\xi| < \eps$. It is then noted, however, that if there are multiple roots then it is in general impossible to compute their multiplicities with complete certainty. This is because the Newton iterates can pass arbitrarily close to a root to which this iterate does not, in the end, converge. Another consequence of this possibility is that roots could be missed out altogether because all of the iterates can be found to be close to a strict subset of the roots.

To salvage the situation, we give the following lemma, which adds some assumptions to the polynomial in question.

\begin{lemma}\label{lem:computingroots}
 Let $p$ be a polynomial and $\Omega \subset \C$ an open set such that, a priori, the degree $d$ is known and it is known that there are $r$ distinct roots of $p$ in $\Omega$ and no roots on the boundary of $\Omega$. Then the roots of $p$ in $\Omega$ is computable with error control in the Hausdorff metric (see Definition \ref{def:Hausdorff} and Definition \ref{def:errorcontrol}).
 \begin{proof}
Use Newton's method with the $\mathcal{O}(d(\log d)^2)$ complex initial data given in \cite{hubbard2001find}. Using the stopping criteria in the discussion preceding this lemma, the algorithm at each iteration produces $\mathcal{O}(d(\log d)^2)$ discs in the complex plane, within which all roots of $p$ must lie. To be clear, these discs have centres $z_k$ and radii $d\cdot|z_k - z_{k-1}|$. Let $R_k \subset \Omega$ denote the union of the discs which lie entirely inside $\Omega$ and have radius less than $\epsilon$ (the desired error). Note that this set may be empty if none of the discs are sufficiently small.

Because the Newton iterations are guaranteed to converge from these initial data, we  must have \emph{eventually}, for some sufficiently large $k$, that $R_k$ has $r$ connected components each with diameter less than $\epsilon$. Terminate when this verifiable condition has been fulfilled.
 \end{proof}
\end{lemma}

\begin{theorem}\label{thm:computToeplusfinite}
 Let $J = \Delta + F$ be a Toeplitz-plus-finite-rank Jacobi operator such that the rank of $F$ is known a priori. Then its point spectrum $\sigma_p(J)$ is computable with error control in the Hausdorff metric (see Definition \ref{def:Hausdorff} and Definition \ref{def:errorcontrol}).
 \begin{remark}
  Note that the full spectrum is simply $[-1,1] \cup \sigma_p(J)$.
 \end{remark}
 \begin{proof}
  Suppose $F$ is zero outside the $n \times n$ principal submatrix. The value of $n$ can be computed given that we know the rank of $F$. Compute the principal $2n \times 2n$ submatrix of the connection coefficients matrix $C_{J \to \Delta}$ using formulae \eqref{connectionrecurrence1}--\eqref{connectionrecurrence2}. The entries in the final column of this $2n \times 2n$ matrix give the coefficients of the Toeplitz symbol $c$, which is a degree $2n-1$ polynomial. 
  
  Decide if $\pm 1$ are roots by evaluating $p(\pm 1)$. Divide by the linear factors if necessary to obtain a polynomial $\tilde{p}$ such that $\tilde{p}(\pm 1) \neq 0$. Use Sturm's Theorem to determine the number of roots of $\tilde{p}$ in $(-1,1)$, which we denote $r$ \cite{rahman2002analytic}. Since all roots in $\overline{\mathbb{D}}$ are real, there are $r$ roots of $\tilde{p}$ in the open unit disc $\mathbb{D}$ and none on the boundary.
  
  By Lemma \ref{lem:computingroots}, the roots $z_1,\ldots,z_r$ of this polynomial $c$ which lie in $(-1,1)$ can be computed with error control. By Theorem \ref{thm:joukowskimeasure}, for the point spectrum of $J$ we actually require $\lambda_k = \frac12(z_k + z_k^{-1})$ to be computed with error control. Note that since $|\lambda_k| \leq \|J\|_2$ for each $k$, we have that $|z_k| \geq (1+2\|J\|_2)^{-1}$. We should ensure that this holds for the computed roots  $\hat{z}_k \in \mathbb{D}$ too. By the mean value theorem,
  \begin{align*}
   |\lambda(z_k) - \lambda(\hat{z}_k)| &\leq \sup_{|z|\geq (1+2\|J\|_2)^{-1}}|\lambda'(z)||z_k - \hat{z}_k| \\
                                       &= \frac12\left((1+2\|J\|_2)^2-1\right)|z_k - \hat{z}_k| \\
                                       &= 2\|J\|_2(1+\|J\|_2)|z_k - \hat{z}_k| \\
                                       &\leq 2(1+\|F\|_2)(2+\|F\|_2)|z_k - \hat{z}_k|.
  \end{align*}
Therefore it suffices to compute $\hat{z}_k$ such that $|z_k - \hat{z}_k| \leq \frac{\eps}{2}(1+\|F\|_2)^{-1}(2+\|F\|_2)^{-1}$, where $\eps$ is the desired error in the eigenvalues.
 \end{proof}
\end{theorem}

The following Theorem shows that taking Toeplitz-plus-finite rank approximations of a Toeplitz-plus-compact Jacobi operator is sufficient for computing the spectrum with error control with respect to the Hausdorff metric.

\begin{theorem}\label{thm:computeTpluscomp}
  Let $J = \Delta + K$ be a Toeplitz-plus-compact Jacobi operator. If for all $\epsilon > 0$ an integer $m$ can be computed such that
  \begin{equation}
  \sup_{k\geq m} |\alpha_k| + \sup_{k\geq m} \left|\beta_k -\frac12 \right| < \epsilon,
  \end{equation}
  then the spectrum can be computed with error control in the Hausdorff metric.
  \begin{proof}
    Let $\epsilon > 0$. By the oracle assumed in the statement of the theorem, compute $m$ such that 
    \begin{equation*}
    \sup_{k\geq m} |\alpha_k| + \sup_{k\geq m} \left|\beta_k -\frac12 \right| < \frac{\epsilon}{6}.
    \end{equation*}
    
    Now compute the point spectrum of the Toeplitz-plus-finite-rank approximation $J^{[m]}$ such that $d_H(\Sigma,\sigma(J^{[m]})) < \eps/2$, where $\Sigma$ denotes the computed set. Then, using Theorem \ref{thm:specpert}, we have
    \begin{align*}
    d_H(\Sigma,\sigma(J)) &\leq d_H(\Sigma,\sigma(J^{[m]})) + d_H(\sigma(J^{[m]}),\sigma(J)) \\
    &\leq \frac{\eps}{2} + \|J^{[m]}-J\|_2 \\
    &\leq \frac{\eps}{2} + 3 \frac{\eps}{6}\\
    &= \eps.
    \end{align*}
    Here we used the fact that for a self-adjoint tridiagonal operator $A$, 
    $$
    \|A\|_2 \leq 3(\sup_{k\geq 0} |a_{k,k}| + \sup_{k\geq 0}|a_{k,k+1}|).
    $$ 
This completes the proof.
  \end{proof}
\end{theorem}

An immediate question following Lemma \ref{lem:computingroots} and Theorem \ref{thm:computToeplusfinite} is why we have opted to use a Newton iteration in the complex plane instead of a purely real algorithm. We do this purely because Lemma \ref{lem:computingroots} is an interesting point to make in and of itself with regards to the Solvability Complexity Index of polynomial rootfinding with error control. The key point is that while there exist algorithms to compute all of the roots of a polynomial (without multiplicity) in a single limit (i.e. with SCI equal to 1), one does not necessarily know when to stop the algorithm to achieve a desired error. Lemma \ref{lem:computingroots} provides a basic condition on the polynomial to allow such control, which applies to this specific spectral problem.

%% file: conclusions.tex
\section{Conclusions}\label{sec:conclusions}

In this paper we have proven new results about the relationship between the connection coefficients matrix between two different families of orthonormal polynomials, and the spectral theory of their associated Jacobi operators. We specialised the discussion to finite-rank perturbations of the free Jacobi operator and demonstrated explicit formulas for the principal resolvent and the spectral measure in terms of entries of the connection coefficients matrix. We showed that the results extend to trace class perturbations. Finally, we discussed computability aspects of the spectra of Toeplitz-plus-compact Jacobi operators. We showed that the spectrum of a Toeplitz-plus-compact Jacobi operator can be computed with error control, as long as the tail of the coefficients can be suitably estimated.

There are some immediate questions.  Regarding regularity properties of the Radon-Nikodym derivative $\frac{\d\nu}{\d\mu}$ between the spectral measures $\nu$ and $\mu$ of Jacobi operators $D$ and $J$ respectively given in Propositions \ref{prop:Cmeasuremultplier} and \ref{prop:Cbounded} and Corollary \ref{cor:Cbddinvert}: can  weaker regularity of $\frac{\d\nu}{\d\mu}$ be related to weak properties of $C = C_{J\to D}$? For example, the present authors conjecture that the Kullbeck--Leibler divergence,
\begin{equation*}
K(\mu|\nu) = \left\{\begin{array}{rl}
               \int \frac{\d\nu}{\d\mu}(s) \log \frac{\d\nu}{\d\mu}(s) \,\d\nu(s) & \text{ if $\nu$ is absolutely continuous w.r.t. $\mu$} \\
               \infty & \text{ otherwise,}
              \end{array} \right.
\end{equation*}
is finite if and only if the function of operators, $C^T C \log(C^T C)$ is well-defined as an operator mapping $\ell_{\mathcal{F}}\to \ell_{\mathcal{F}}^\star$. The reasoning comes from Lemma \ref{lem:CTC}. Making such statements more precise for the case where $D = \Gamma$ or $D = \Delta$ (see equation \eqref{eqn:DeltaGamma}) could give greater insight into Sz\H{e}go and quasi-Sz\H{e}go asymptotics (respectively) for orthogonal polynomials \cite{gamboa2016sum,damanik2006jost1,killip2003sum}.

Regarding computability: is there a theorem that covers the ground between  Theorem \ref{thm:computToeplusfinite} (for Toeplitz-plus-finite-rank Jacobi operators) and Theorem \ref{thm:computeTpluscomp} (for Toeplitz-plus-compact Jacobi operators)? What can be said about the convergence of the continuous part of the spectral measure of a Toeplitz-plus-finite-rank truncations of a Toeplitz-plus-trace-class Jacob operator? Proposition \ref{prop:weakconvergence} implies that this convergence is at least weak sense when tested against $f \in C_b(\R)$.

The computability theorems in Section \ref{sec:computability} all assume real arithmetic. What can be said about floating point arithmetic? Under what situations can the computation fail to give an unambiguously accurate solution?  Answering this question is related to the mathematical problem of  stability of the spectral measure under small perturbations of the Jacobi operator.  

This paper also opens some broader avenues for future research.  The connection coefficient matrix can be defined for any two Jacobi operators $J$ and $D$. It is natural to explore what structure $C_{J\to D}$ has when $D$ is a different reference operator to $\Delta$, and $J$ is a finite rank, trace class, or compact perturbation of $D$. For example, do perturbations of the Jacobi operator with periodic entries \cite{damanik2010perturbations,geronimo1986orthogonal} have structured connection coefficient matrices?
 Beyond periodic Jacobi operators, it would be interesting from the viewpoint of ergodic theory if we could facilitate the study and computation of almost-periodic Jacobi operators, such as the discrete almost-Mathieu operator \cite{deift2008some}.   Perturbations of the Jacobi operators for Laguerre polynomials and the Hermite polynomials could also be of interest, but challenges associated with the unboundedness of these operators could hamper progress \cite{olverDLMF}.  Discrete Schr\"odinger operators with non-decaying potentials will also be of interest in this direction.

Spectra of banded self-adjoint operators may be accessible with these types of techniques too. Either using connection coefficient matrices between matrix orthogonal polynomials \cite{damanik2008analytic}, or developing tridiagonalisation techniques are possible approaches, but the authors also consider this nothing more than conjecture at present. The multiplicity of the spectrum for operators with bandwidth greater than 1 appears to be a major challenge here. This becomes even more challenging for non-banded operators, such as Schr\"odinger operators on ${\mathbb Z}^d$ lattices.

Lower Hessenberg operators define polynomials orthogonal with respect to Sobolev inner products \cite[pp.~40--43]{gautschi2004orthogonal}. Therefore, we have two families of (Sobolev) orthogonal polynomials with which we may define connection coefficient matrices, as discussed in \cite[p.~77]{golub2009matrices}. Whether the connection coefficient matrices (which are still upper triangular) have structure which can be exploited for studying and computing the spectra of lower Hessenberg operators is yet to be studied.

Besides spectra of \emph{discrete} operators defined on $\ell^2$, we conjecture that the results of this paper will also be applicable to continuous Schr\"odinger operators on $L^2(\R)$, which are of the form $L_V[\phi](x) = -\phi''(x) + V(x)\phi(x)$ for a potential function $V :\R \to \R$. The reference operator is the negative Laplacian $L_0$ (which is the ``free'' Schr\"odinger operator). In this scenario, whereas the entries of a discrete connection coefficient matrix satisfy a discrete second order recurrence relation on $\mathbb{N}_0^2$ (see Lemma \ref{lem:5ptsystem}), the continuous analogue of the connection coefficient operator $C_{L_V\to L_0}$ is an integral operator whose (distributional) kernel satisfies a second order PDE on $\R^2$.

%% file: appendix.tex
\begin{appendix}
  
\section{Numerical results and the \textit{SpectralMeasures} package}\label{sec:numericalexperiments}

\newlength{\figheight}
\setlength{\figheight}{.14\textheight}

In this appendix we demonstrate some of the features of the SpectralMeasures.jl Julia package \cite{SpectralMeasuresjl} that the authors have written to implement the ideas in the paper. This is part of the JuliaApproximation project, and builds on the package ApproxFun.jl \cite{olverapproxfun}. {ApproxFun} is an extensive piece of software influenced by the Chebfun package \cite{driscoll2014chebfun} in Matlab, which can represent functions and operators \cite{olver2014practical,olverapproxfun}. The code is subject to frequent changes and updates.

Given a Jacobi operator $J$ which is a finite-rank perturbation of the free Jacobi operator $\Delta$ with entries given by $\alpha_k = 0$, $\beta_{k-1}=\frac12$ for all $k \geq n$, SpectralMeasures.jl enables calculation of the following:
\begin{enumerate}[(i)]
  \item The connection coefficients matrix $C_{J\to\Delta}$: This is computed using the recurrences in equation \eqref{connectionrecurrence1}--\eqref{connectionrecurrence2}. By Theorem \ref{thm:connectioncoeffs}, we only need to compute $n(n+1)$ entries of $C$ to have complete knowledge of all entries. In SpectralMeasures.jl, there is a type of operator called {\tt PertToeplitz}, which allows such an operator to be stored and manipulated as if it were the full infinite-dimensional operator.
  \item The spectral measure $\mu(s)$: By Theorem \ref{thm:mainmeasure}, this measure has the form
  \begin{equation*}
  \d\mu(s) = \frac{1}{p_C(s)} \frac{2}{\pi}\sqrt{1-s^2}\d s + \sum_{k=1}^r w_k \delta_{\lambda_k}(s),
  \end{equation*}
  where $p_C$ is the polynomial given by the computable formula $p_C(s) = \sum_{k=0}^{2n-1} \langle C^Te_k ,C^T e_0 \rangle U_k(s)$ and $r \leq n$. By Theorem \ref{thm:joukowskimeasure}, the numbers $\lambda_k$ are found by finding the distinct real roots $z_k$ of $c$ (the Toeplitz symbol of the Toeplitz part of $C$, which here is a polynomial of degree $2n-1$) in the interval $(-1,1)$. Also by Theorem \ref{thm:joukowskimeasure}, the weights $w_k$ can be computed using the formula
  \begin{equation*}
  w_k = \frac12 z_k^{-1}(z_k-z_k^{-1})\frac{c_\mu(z_k)}{c'(z_k)}.
  \end{equation*}
  \item The principal resolvent $G(\lambda)$: For any $\lambda \in \C\setminus\sigma(J)$, by Theorem \ref{thm:mainresolvent}, this function can be defined by the formula
  \begin{equation*}
  G(\lambda) = \frac{G_\Delta(\lambda) - p_C^{\mu}(\lambda)}{p_C(\lambda)},
  \end{equation*}
  where $p_C$ is as above and $p^\mu_C(\lambda) = \sum_{k=0}^{2n-1} 
  \langle (C^\mu)^T e_k, C^T e_0 \rangle U_k(\lambda)$.
  \item The mapped principal resolvent $G(\lambda(z))$: which is the principal resolvent mapped to $z$ in the unit disc by the Joukowski map $\lambda:z \to \frac12(z+z^{-1})$. This is computed using the simple formula from Theorem \ref{thm:joukowskiresolvent},
  \begin{equation*}
  G(\lambda(z)) = -\frac{c_\mu(z)}{c(z)},
  \end{equation*}
  where $c$ and $c_\mu$ are the Toeplitz symbols of the Toeplitz parts of $C$ and $C^\mu$ respectively (these are polynomials of degree $2n-1$ and $2n-2$ respectively).
\end{enumerate}

Consider a concrete example of a Toeplitz-plus-finite-rank Jacobi operator,
\begin{equation}\label{eqn:explicitex}
J = \left(\begin{array}{cccccc}
\frac34 & 1 & & & & \\
1 & -\frac14 & \frac34 & & &\\
& \frac34 & \frac12 & \frac12 & & \\
&         & \frac12 & 0 & \frac12 & \\
&         &         & \ddots & \ddots & \ddots
\end{array}\right).
\end{equation}
The connection coefficients operator $C = C_{J\to\Delta}$ is
\begin{equation*}
C = \left(\begin{array}{cccccccc}
1 & -\frac34 & -\frac54 & \frac{49}{24} & -\frac{1}{12} & -\frac13      &               & \\
  & \frac12  & -\frac13 & -\frac43      & \frac{41}{24} & -\frac{1}{12} & -\frac13      & \\
  &          & \frac13  & -\frac23      & -\frac43      & \frac{41}{24} & -\frac{1}{12} & \ddots \\
  &          &          &  \frac13      & -\frac23      & -\frac43      & \frac{41}{24} & \ddots \\
  &          &          &               & \frac13       & -\frac23      & -\frac43      & \ddots \\
  &          &          &               &               & \ddots        & \ddots        & \ddots
\end{array}\right).
\end{equation*}
As was noted above, we only need to explicitly compute $3\cdot4 = 12$ entries and the rest are defined to be equal to the entry above and left one space.  Because the perturbation is rational, the connection coefficients operator can be calculated exactly using the {\tt Rational} type.

In Figure \ref{fig:3by3pert} we present plots of the spectral measure, the principal resolvent and the mapped principal resolvent. This format of figure is repeated for other Jacobi operators in the appendix. 

The plot on left in Figure \ref{fig:3by3pert} is the spectral measure. There is a continuous part supported in the interval $[-1,1]$ and Dirac deltas are represented by vertical lines whose heights are precisely their weights. As noted in Theorem \ref{thm:computToeplusfinite}, it is possible to compute the eigenvalues of $J$ with guaranteed error control. Computations with guaranteed error control are made quite straightforward and flexible using the \textit{ValidatedNumerics.jl} package \cite{ValidatedNumerics}, in which computations are conducted using interval arithmetic, and the desired solution is rigorously guaranteed to lie within the interval the algorithm gives the user \cite{tucker2011validated}. Using this open source Julia package, we can compute the two eigenvalues for this operator to be $-1.1734766767874558$ and $1.5795946563898884$ with a guaranteed maximum error of $8.9\times 10^{-16}$. This can be replicated using the command \texttt{validated\_spectrum([.75;-.25;.5],[1;.75])} in \textit{SpectralMeasures}.

The coloured plots in the middle and the right of Figure \ref{fig:3by3pert} are \emph{Wegert plots} (sometimes called phase portraits) \cite{wegert2012visual,ComplexPhasePortrait}. For a function $f : \C \to \C$, a Wegert plot assigns a colour to every point $z \in \C$ by the argument of $f(z)$. Specifically, if $f(z) = re^{i\theta}$, then $\theta = 0$ corresponds to the colour red, then cycles upwards through yellow, green, blue, purple as $\theta$ increases until at $\theta = 2\pi$ it returns to red. This makes zeros and poles very easy to see, because around them the argument cycles through all the colours the same number of times as the degree of the root or pole. In these particular Wegert plots, we also plot lines of constant modulus as shadowed steps.

The middle plot in Figure \ref{fig:3by3pert} is the principal resolvent $G(\lambda)$, which always has a branch cut along the interval $[-1,1]$ and roots and poles along the real line. The poles correspond to Dirac delta measures in the spectral measure.

The third plot is the principal resolvent of $J$ mapped to the unit disc by the Joukowski map. Poles and roots of this resolvent in the unit disc correspond to those of the middle plot outside $[-1,1]$.

\begin{figure}[!h]
  \begin{center}
    \includegraphics[height=\figheight,trim=0 -1 0 10]{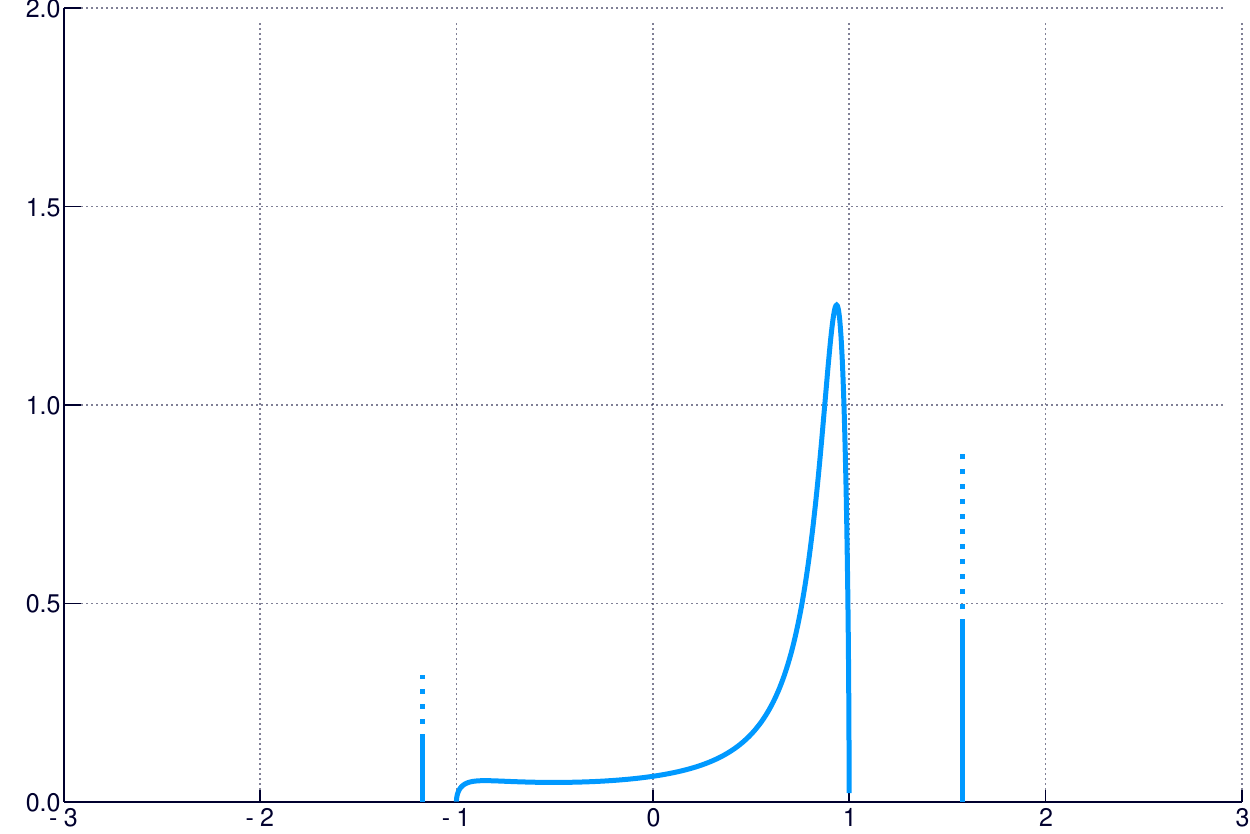}\hskip8pt\includegraphics[height=\figheight]{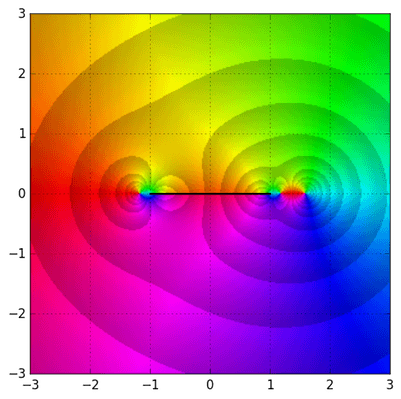}\hskip8pt\includegraphics[height=\figheight]{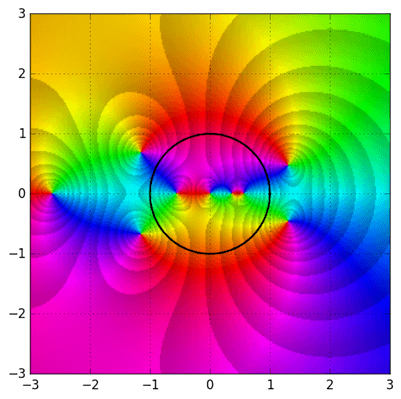}
    \caption{The left plot is the spectral measure $\mu(s)$ of the Jacobi operator in equation \eqref{eqn:explicitex}. The middle plot is a Wegert plot (explained in the text above) depicting the principal resolvent of the same Jacobi operator, and the right plot is the principal resolvent under the Joukowski mapping. The two Dirac deltas in the spectral measure correspond to two poles along the real line for the middle plot and two poles inside the unit disc for the right plot.}\label{fig:3by3pert}
  \end{center}
\end{figure}

In Figure \ref{fig:basic1revisited} we have plotted the spectral measure and principal resolvent of the Basic Perturbation 1 (see Examples \ref{ex:basic1}, \ref{ex:basic1rev}, \ref{ex:basic1rerev}) in which the top-left entry of the operator has been set to $\alpha/2$ for values $\alpha = 0, 0.15, 0.35, 0.5, 0.75, 1$. For the first four cases, the perturbation from the free Jacobi operator is small, and so the spectrum is purely continuous, which corresponds to no poles in the principal resolvent, and in the mapped resolvent there are only poles \emph{outside} the unit disc. For the cases $\alpha = 0.75$ and $1$, the Jacobi operator has a single isolated point of discrete spectrum. This is manifested as a Dirac delta in the spectral measure and a single pole in the principal resolvent.

In Figure \ref{fig:basic2revisited} we have plotted the spectral measure and principal resolvent of  Basic Perturbation 2 (see Examples \ref{ex:basic2}, \ref{ex:basic2rev}, \ref{ex:basic2rerev}) in which the $(0,1)$ and $(1,0)$ entries have been set to $\beta/2$ for values $\beta = 0.5, 0.707, 0.85, 1.0, 1.2, 1.5$. The effect is similar to that observed in Figure \ref{fig:basic1revisited}. For small perturbations the spectrum remains purely continuous, but for larger perturbations here two discrete eigenvalues emerge corresponding to Dirac deltas in the spectral measure and poles in the principal resolvent.

In Figure \ref{fig:legendre} we have plotted a sequence of approximations to the Jacobi operator for the Legendre polynomials, which has entries $\alpha_k = 0$ for $k = 0,1,2,\ldots$ and 
\begin{equation*}
\beta_{k-1} = \frac{1}{\sqrt{4k^2-1}}, \text{ for } k = 1,2,3,\ldots.
\end{equation*}
This is a Toeplitz-plus-trace-class Jacobi operator because $\beta_k = \frac12 + \mathcal{O}(k^{-2})$, and by taking Toeplitz-plus-finite-rank approximations $J^{[n]}$ as in equation \eqref{eqn:Jm}, we can compute approximations to the spectral measure and principal resolvent. Figure \ref{fig:legendre} depicts the spectral measure and the principal resolvent for the Toeplitz-plus-finite-rank Jacobi operators $J^{[n]}$ for the values $n = 1,2,3,10,30,100$. For the spectral measures, we see that there is no discrete part for any $n$, and as $n$ increases, the spectral measure converges to the scaled Lebesgue measure $\frac12 \d s$ restricted to $[-1,1]$. The convergence is at least weak by Proposition \ref{prop:weakconvergence}, but it would be interesting (as mentioned in the conclusions) to determine if there is a stronger form of convergence at play due to the perturbation of $\Delta$ lying in the space of trace class operators. There is a Gibbs-like effect occurring at the boundaries, which suggests that this convergence occur pointwise everywhere up to the boundary of $[-1,1]$. For the principal resolvents, the middle plots do not vary greatly, as the difference between the functions in the complex plane is not major. However, in right plots, there are hidden pole-root pairs in the resolvent lying outside the unit disc which coalesce around the unit disc and form a barrier. The meaning of this barrier is unknown to the authors.

Figures \ref{fig:ultraspherical} and \ref{fig:jacobi} demonstrate similar features to Figure \ref{fig:legendre}, except that the polynomials sequences they correspond to are the ultraspherical polynomials with parameter $\gamma =0.6$ (so that the spectral measure is proportional to $(1-s^2)^{1.1}$) and the Jacobi polynomials with parameters $(\alpha,\beta) = (0.4,1.9)$ (so that the spectral measure is proportional to $(1-s)^{0.4}(1+s)^{1.9}$). Similar barriers of pole-root pairs outside the unit disc occur for these examples as well.

Figure \ref{fig:random} presents a Toeplitz-plus-trace-class Jacobi operator with pseudo-randomly generated entries. With a random vector $\mathbf{r}$ containing entries uniformly distributed in the interval $[0,1)$, the following entries were used
\begin{equation*}
\alpha_k = 3\frac{2r_k - 1}{(k+1)^2}, \qquad \beta_k = \frac12.
\end{equation*}
Then the Toeplitz-plus-finite-rank approximations $J^{[n]}$ (see equation \eqref{eqn:Jm}) of this operator were taken for values $n=1,2,3,10,50,100$. Since the off-diagonal elements are constant, this is a scaled and shifted version of a discrete Schr\"odinger operator with a random, decaying potential.

\begin{figure}[!h]
 \begin{center}
  \includegraphics[height=\figheight,trim=0 -1 0 10]{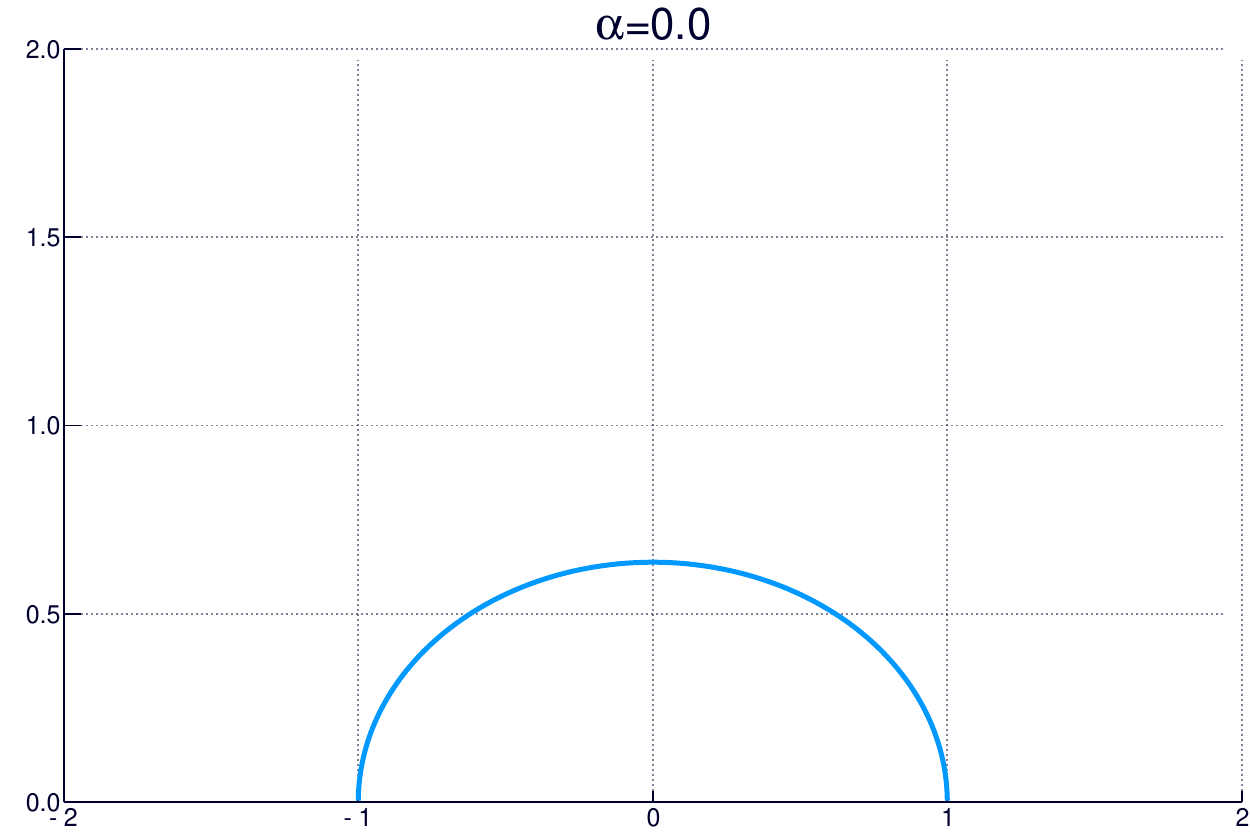}\hskip8pt\includegraphics[height=\figheight]{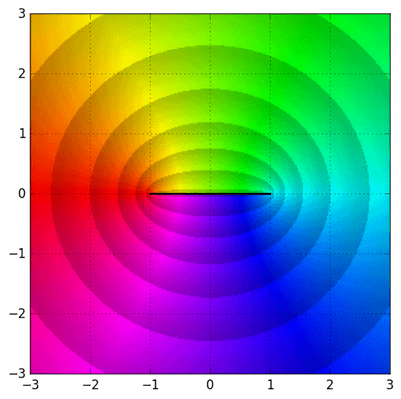}\hskip8pt\includegraphics[height=\figheight]{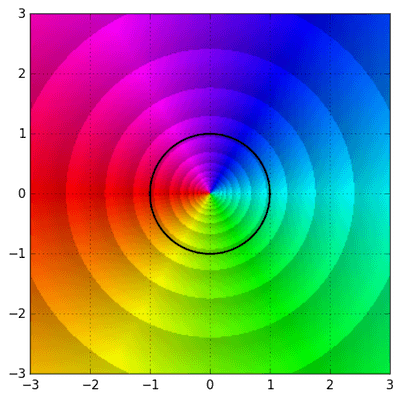}\\\vskip8pt
  \includegraphics[height=\figheight,trim=0 -1 0 10]{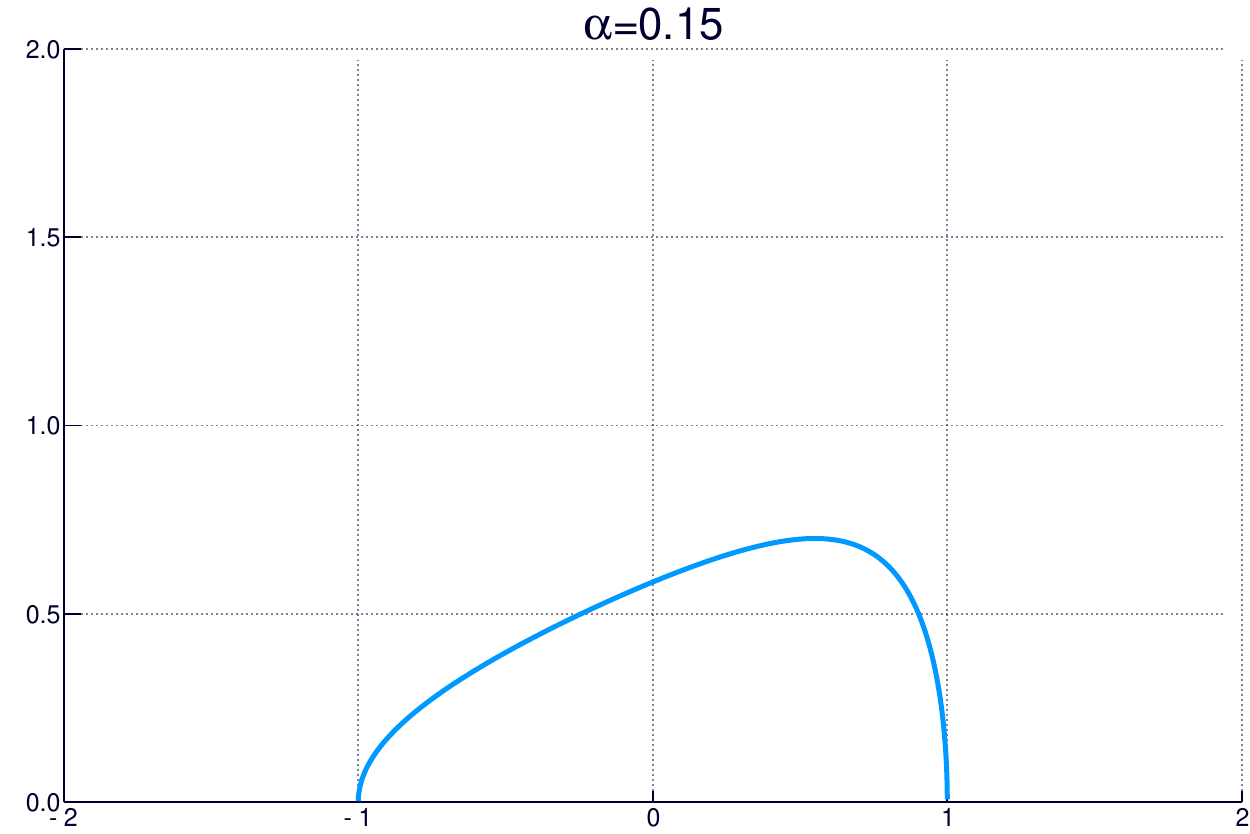}\hskip8pt\includegraphics[height=\figheight]{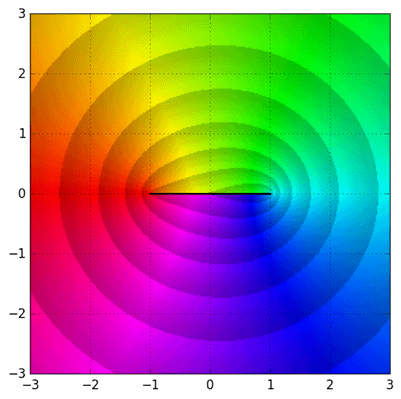}\hskip8pt\includegraphics[height=\figheight]{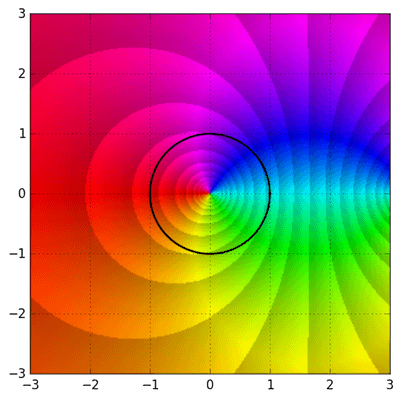}\\\vskip8pt
  \includegraphics[height=\figheight,trim=0 -1 0 10]{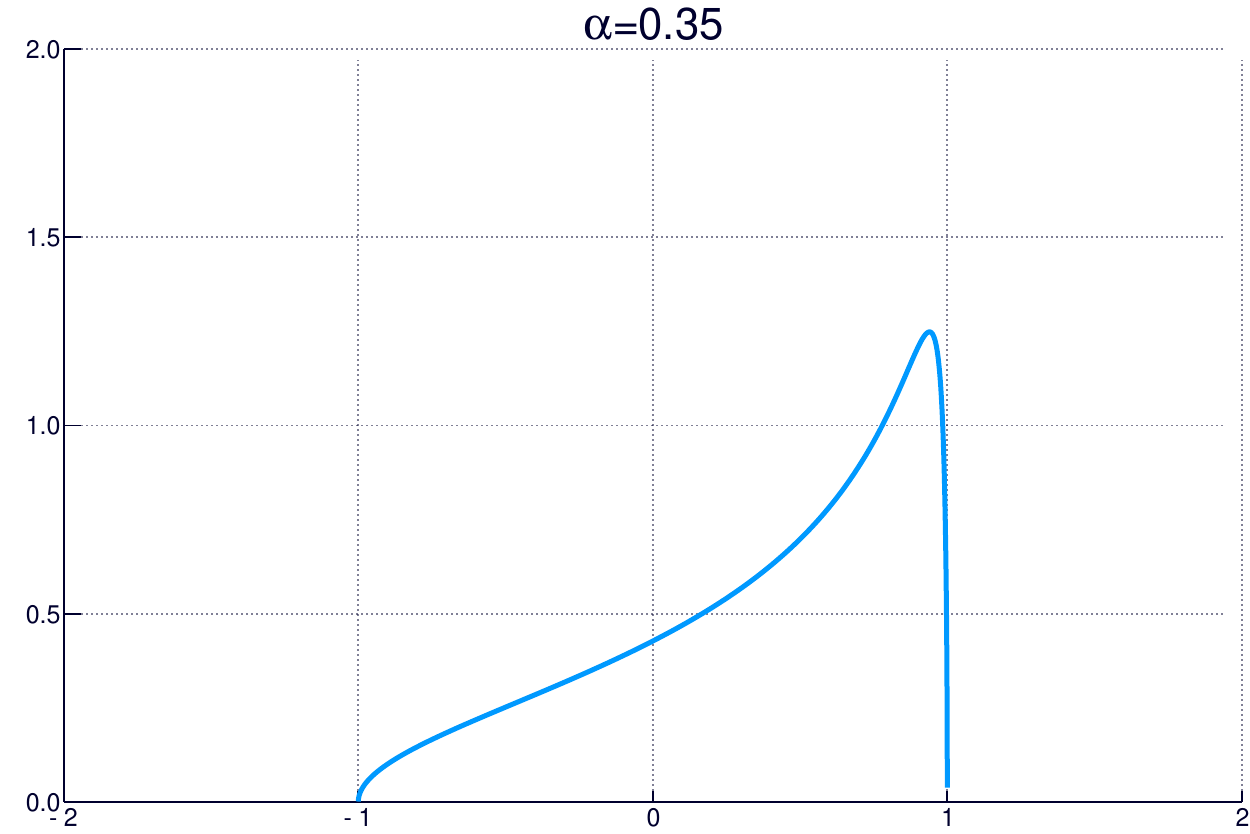}\hskip8pt\includegraphics[height=\figheight]{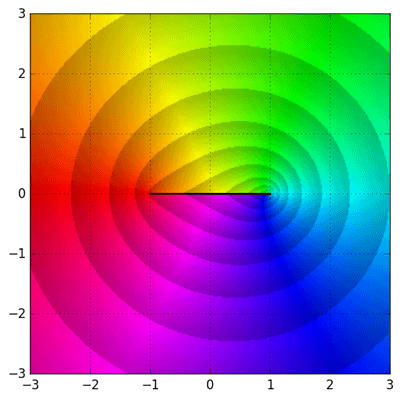}\hskip8pt\includegraphics[height=\figheight]{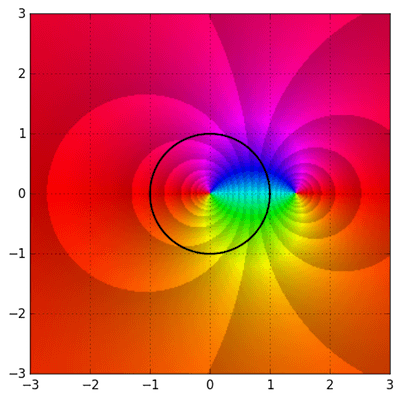}\\\vskip8pt
  \includegraphics[height=\figheight,trim=0 -1 0 10]{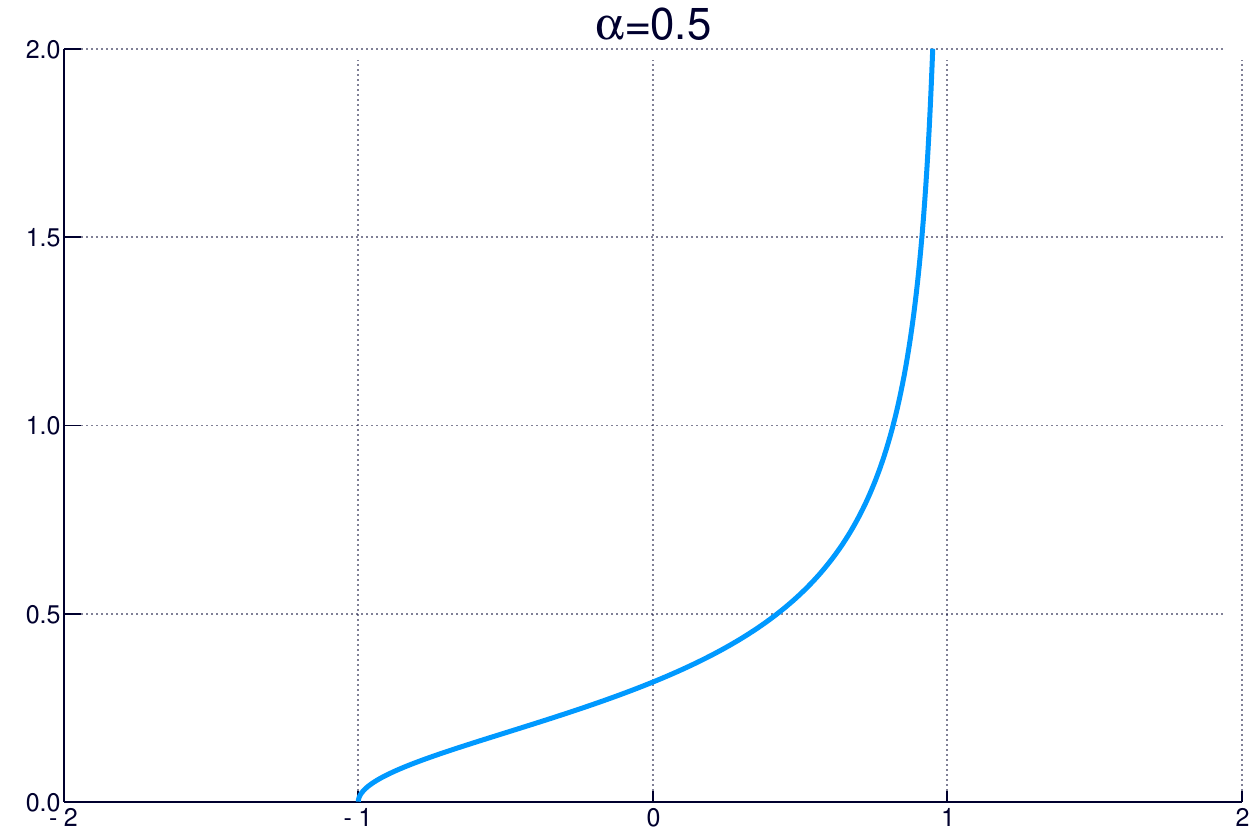}\hskip8pt\includegraphics[height=\figheight]{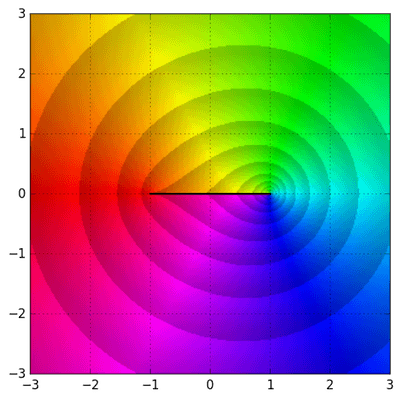}\hskip8pt\includegraphics[height=\figheight]{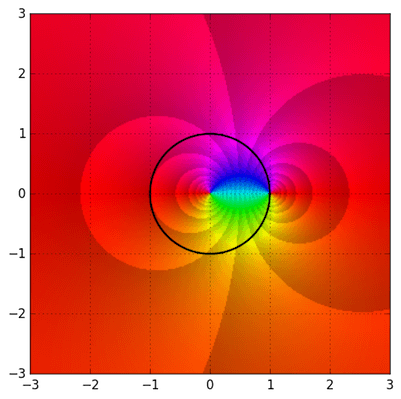}\\\vskip8pt
  \includegraphics[height=\figheight,trim=0 -1 0 10]{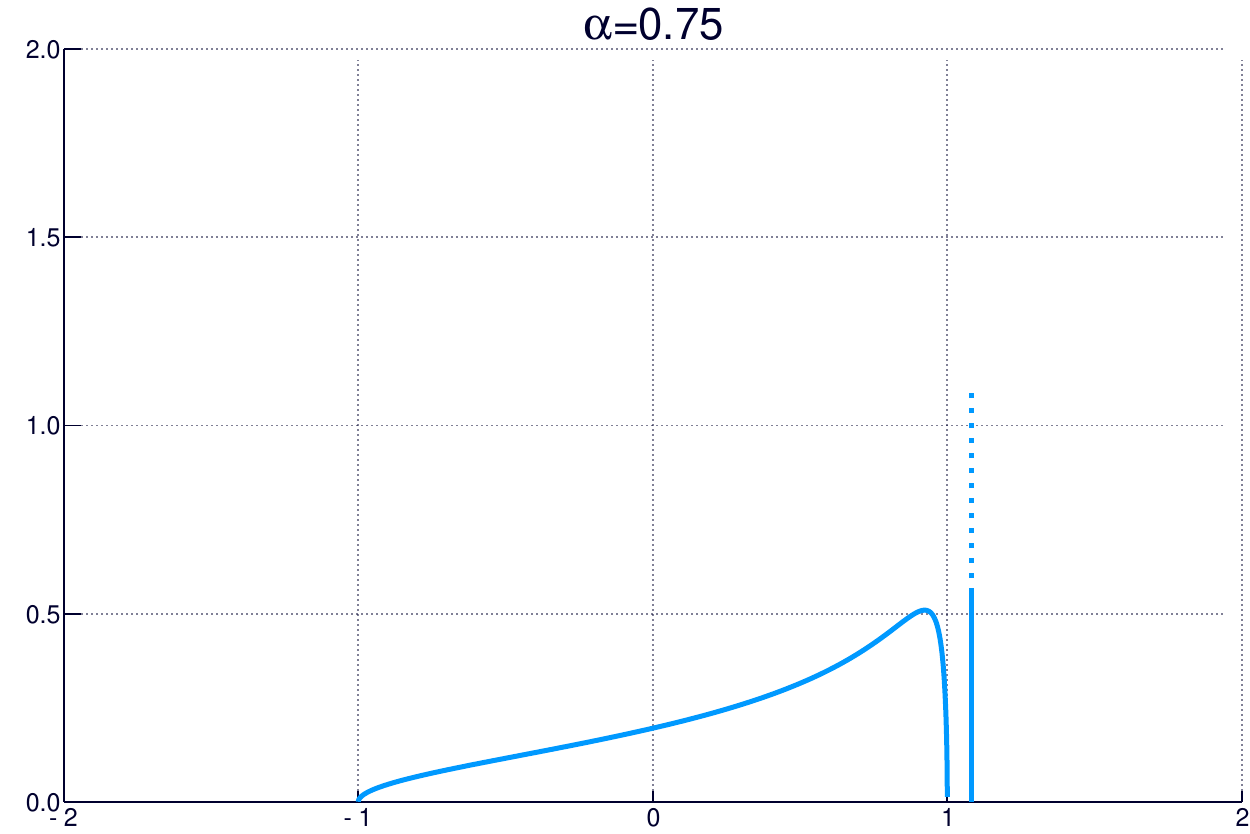}\hskip8pt\includegraphics[height=\figheight]{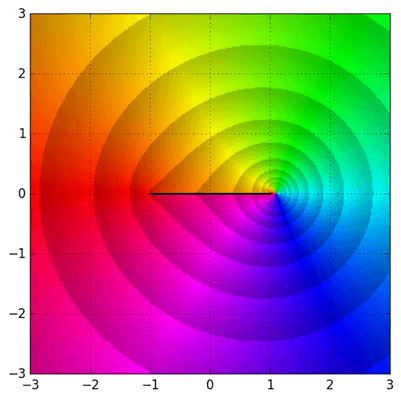}\hskip8pt\includegraphics[height=\figheight]{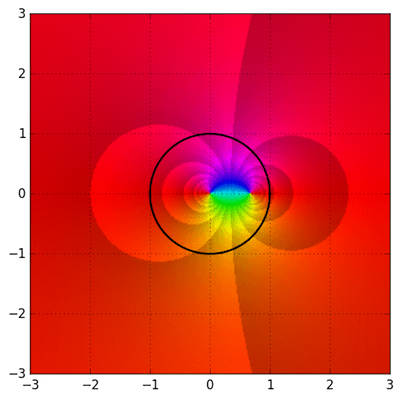}\\\vskip8pt
  \includegraphics[height=\figheight,trim=0 -1 0 10]{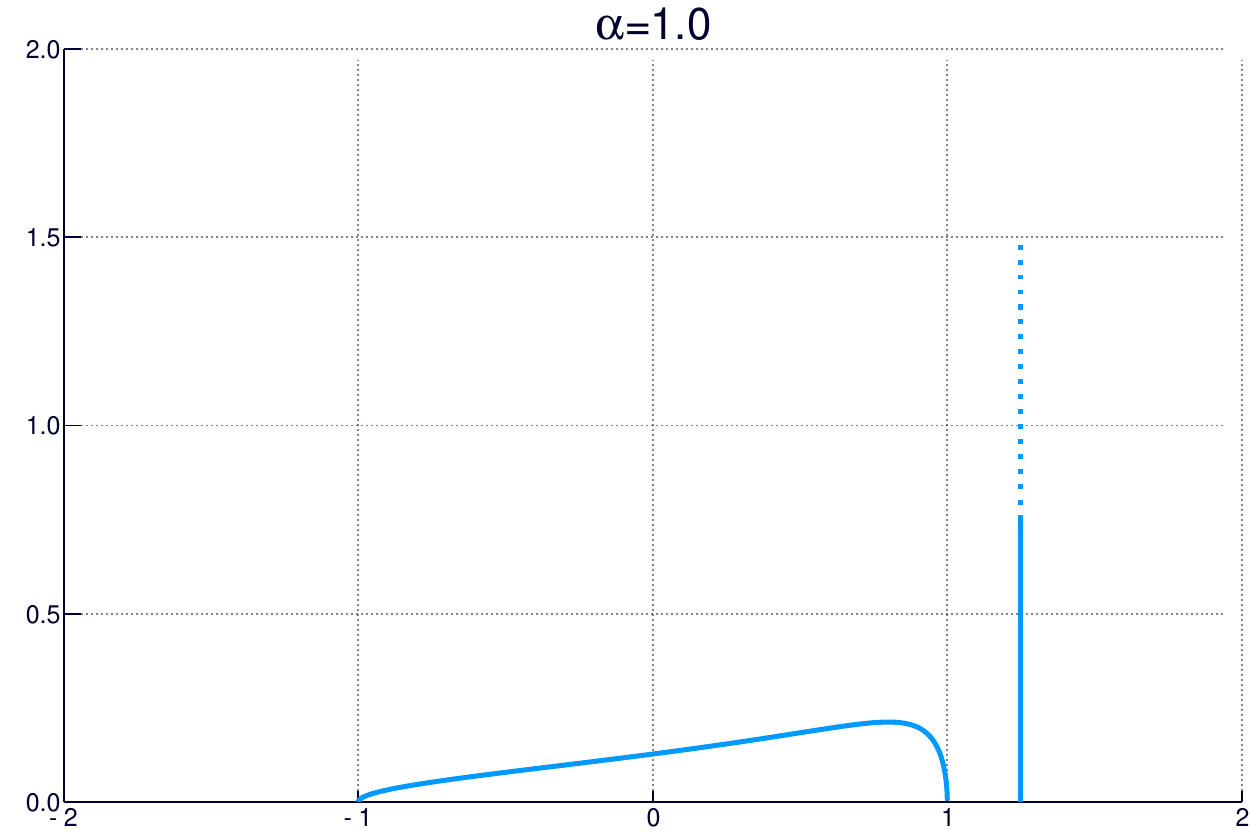}\hskip8pt\includegraphics[height=\figheight]{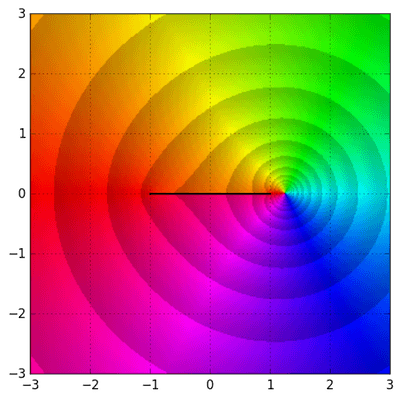}\hskip8pt\includegraphics[height=\figheight]{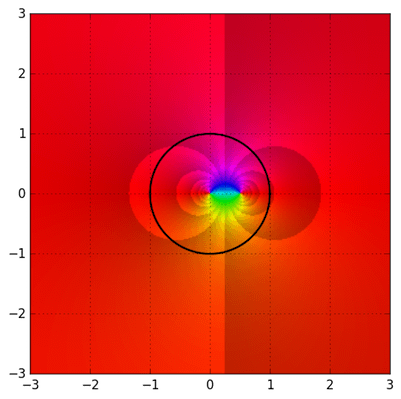}
 \caption{The left hand, centre and right hand figures show the spectral measures $\mu(s)$, principal resolvents $G(\lambda)$ and disc resolvents $G(\lambda(z))$ (analyticially continued outside the disc) respectively for $J_\alpha$, the Basic perturbation 1 example, with $\alpha = 0,0.15,0.35,0.5,0.75,1$. We see that a Dirac mass in the measure corresponds to a pole of the disc resolvent \emph{inside} the unit disc, which corresponds to a pole in the principal resolvent outside the interval $[-1,1]$.}\label{fig:basic1revisited}
 \end{center}
\end{figure}

\begin{figure}[!h]
 \begin{center}
  \includegraphics[height=\figheight,trim=0 -1 0 10]{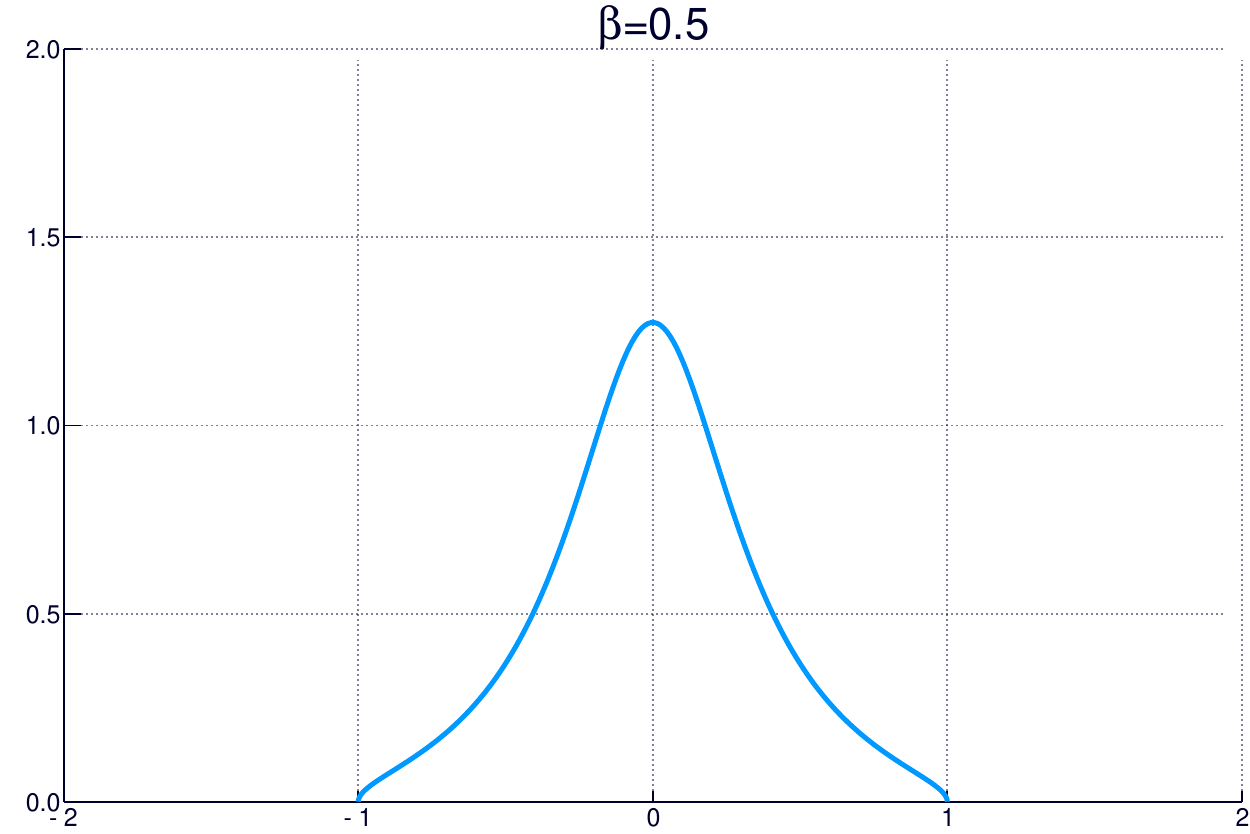}\hskip8pt\includegraphics[height=\figheight]{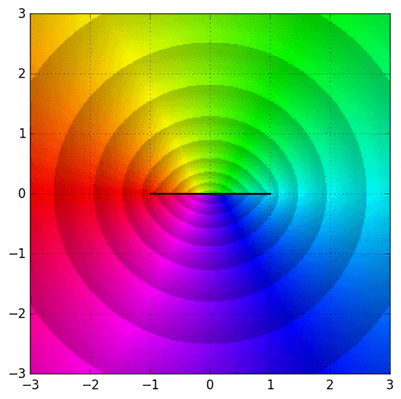}\hskip8pt\includegraphics[height=\figheight]{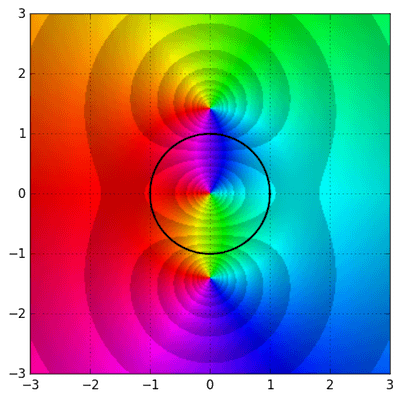}\\\vskip8pt
  \includegraphics[height=\figheight,trim=0 -1 0 10]{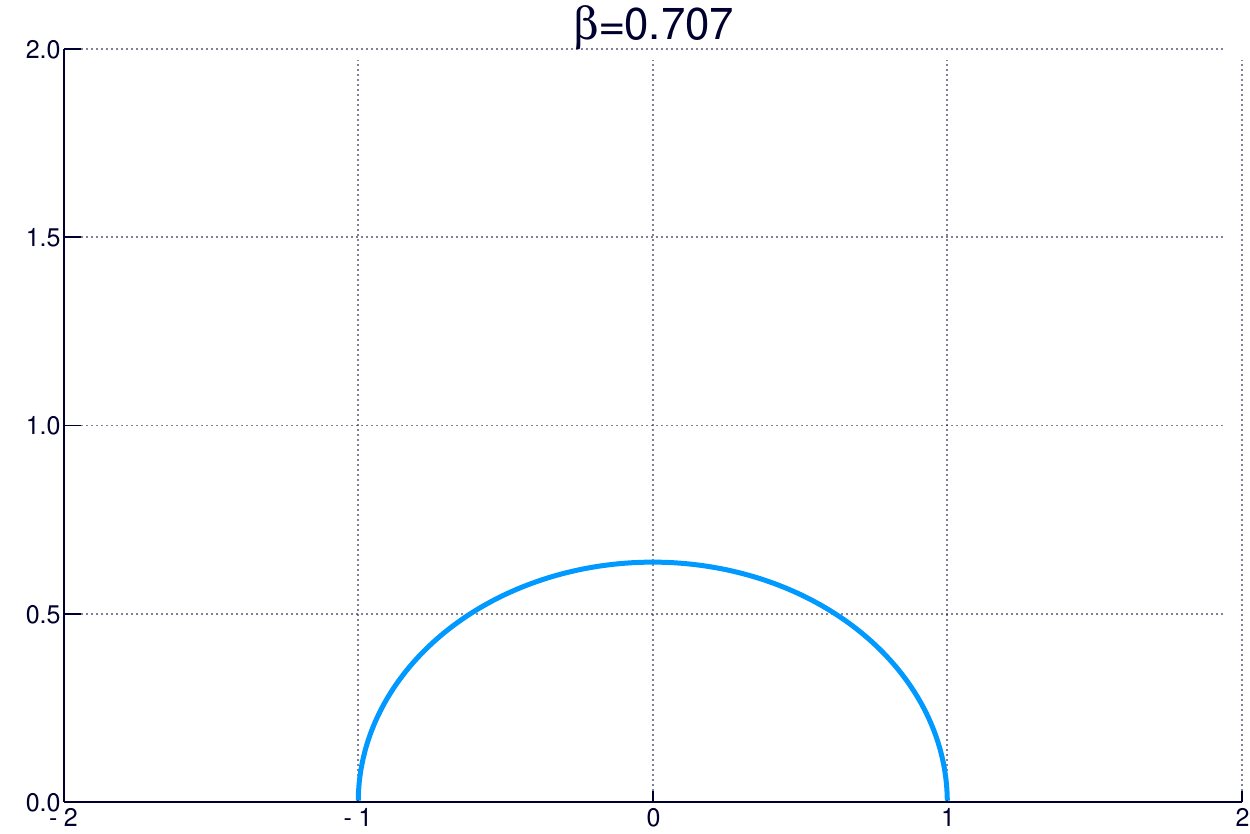}\hskip8pt\includegraphics[height=\figheight]{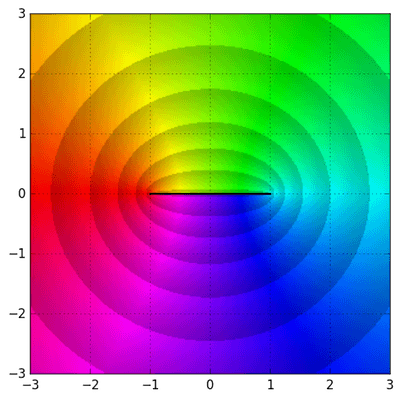}\hskip8pt\includegraphics[height=\figheight]{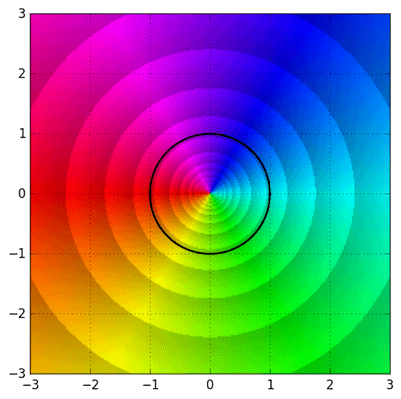}\\\vskip8pt
  \includegraphics[height=\figheight,trim=0 -1 0 10]{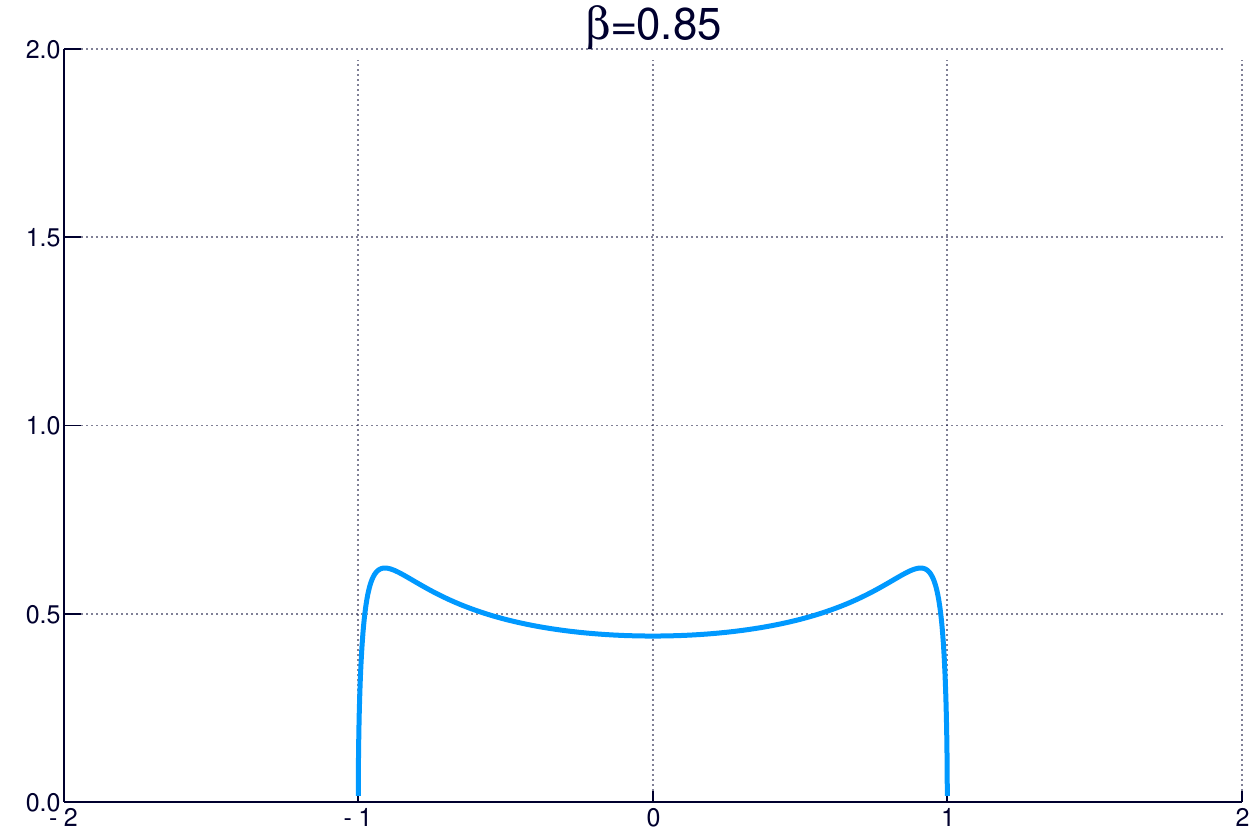}\hskip8pt\includegraphics[height=\figheight]{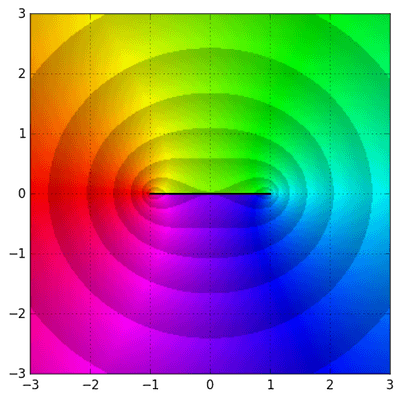}\hskip8pt\includegraphics[height=\figheight]{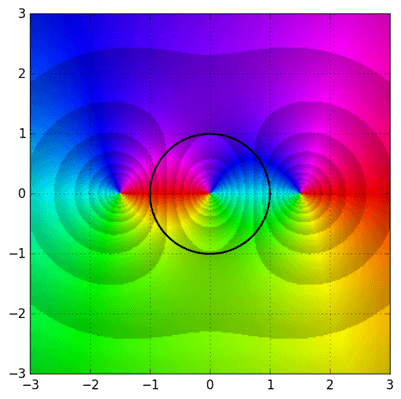}\\\vskip8pt
  \includegraphics[height=\figheight,trim=0 -1 0 10]{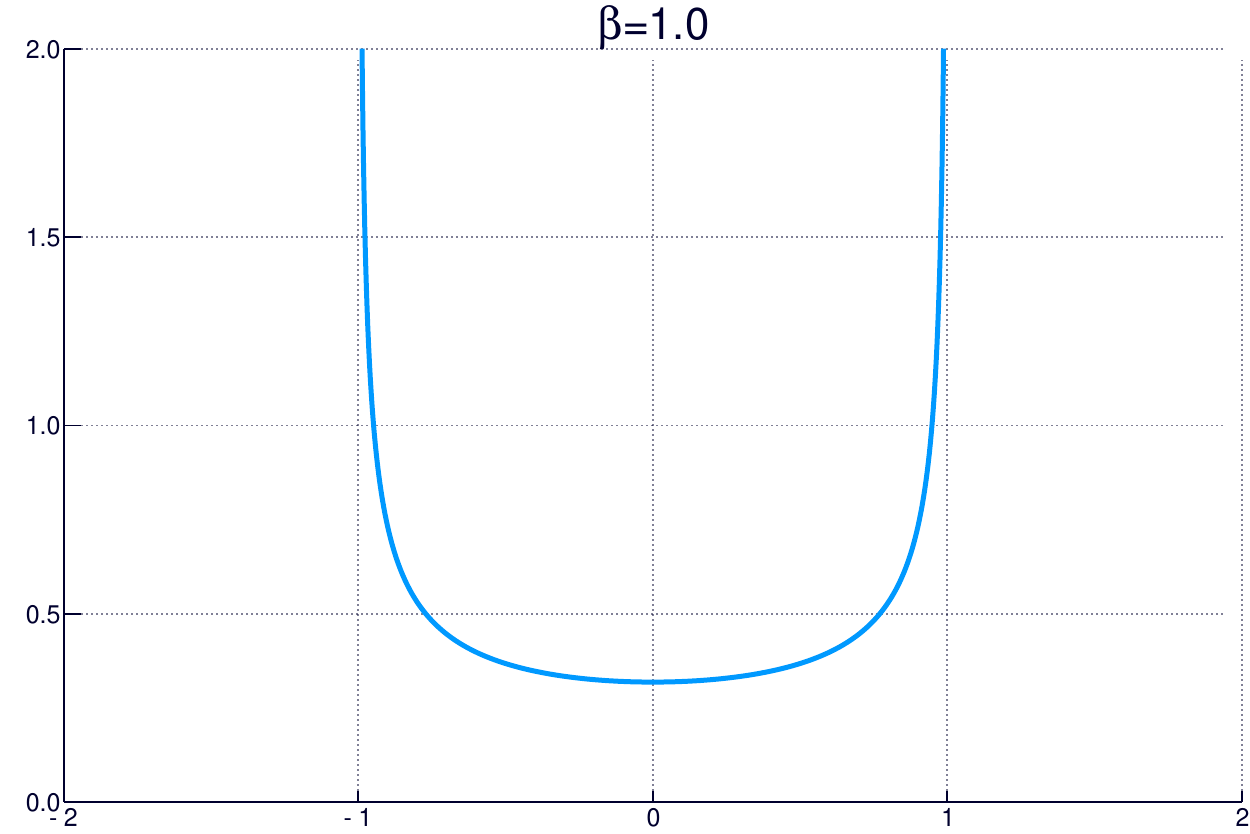}\hskip8pt\includegraphics[height=\figheight]{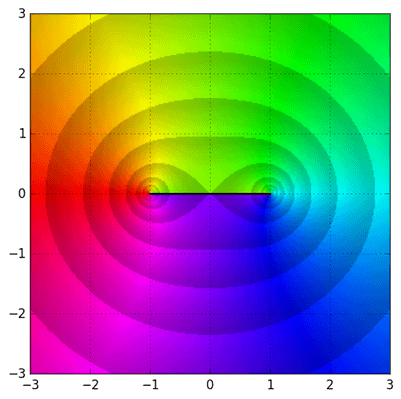}\hskip8pt\includegraphics[height=\figheight]{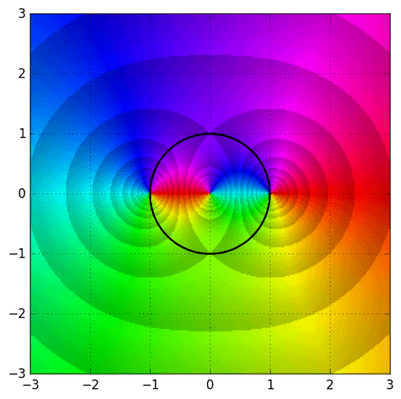}\\\vskip8pt
  \includegraphics[height=\figheight,trim=0 -1 0 10]{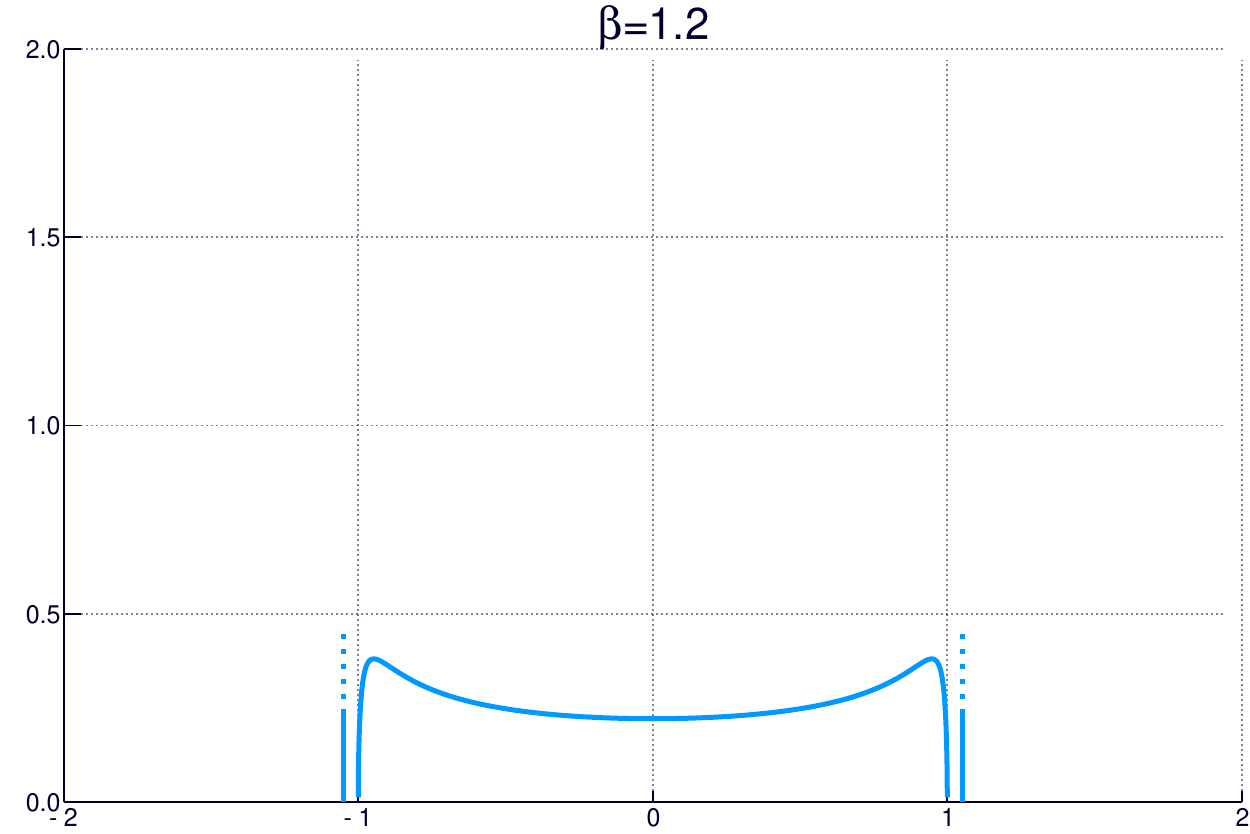}\hskip8pt\includegraphics[height=\figheight]{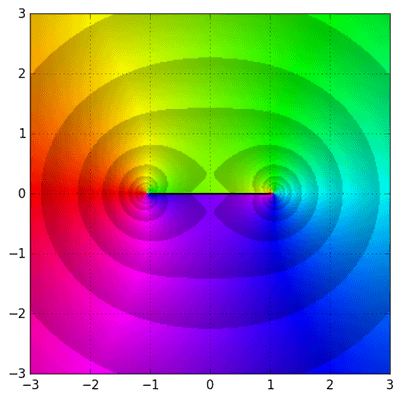}\hskip8pt\includegraphics[height=\figheight]{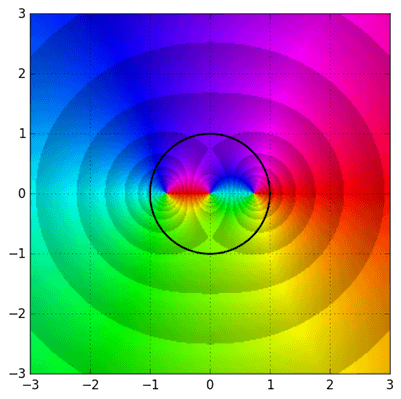}\\\vskip8pt
  \includegraphics[height=\figheight,trim=0 -1 0 10]{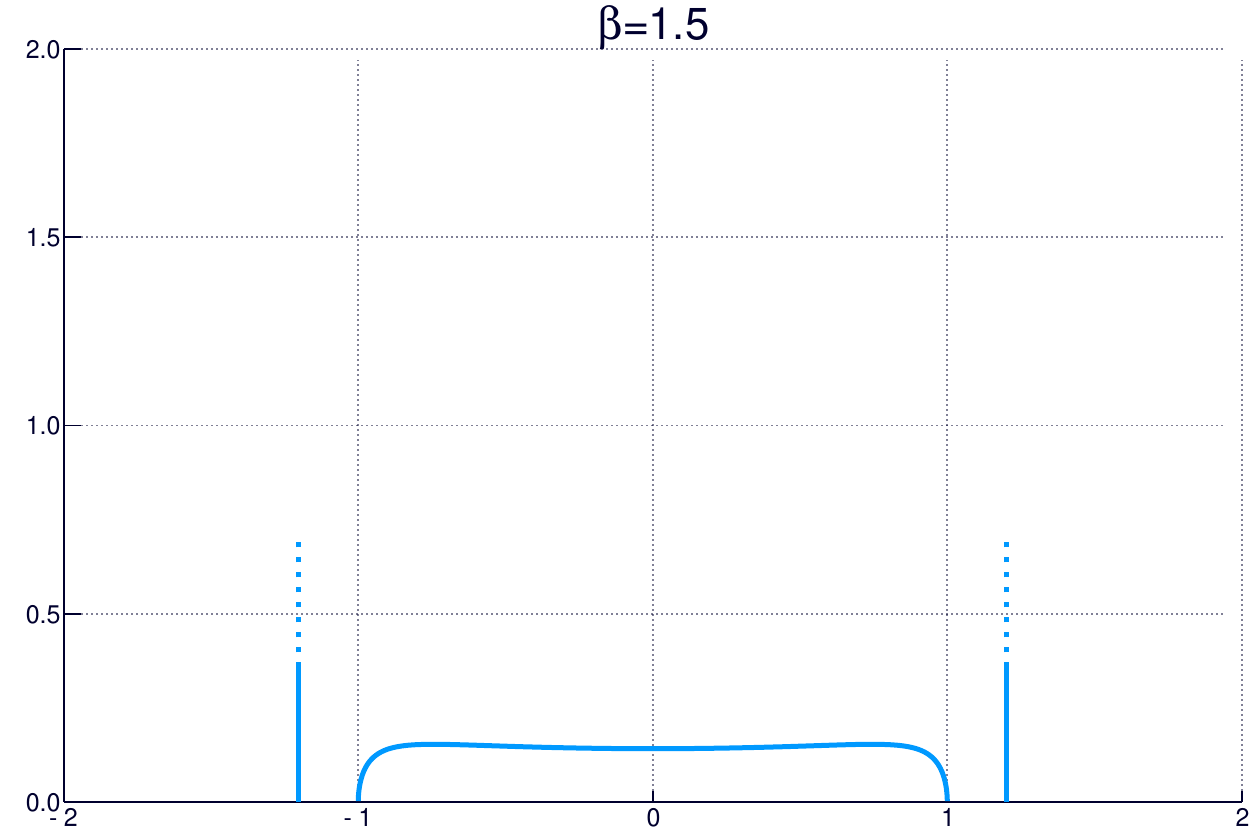}\hskip8pt\includegraphics[height=\figheight]{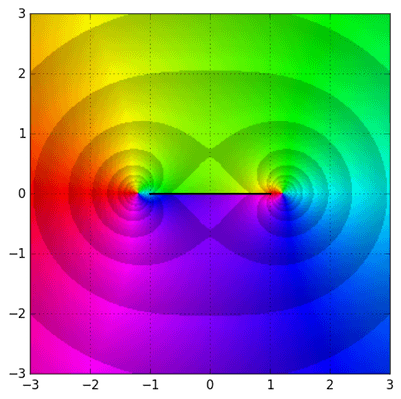}\hskip8pt\includegraphics[height=\figheight]{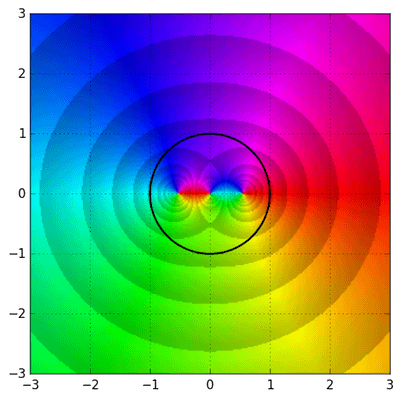}
\caption{The left hand, centre and right hand figures show the spectral measures $\mu(s)$, principal resolvents $G(\lambda)$ and disc resolvents $G(\lambda(z))$ (analyticially continued outside the disc) respectively for $J_\beta$, the Basic perturbation 2 example, with $\beta = 0.5,0.707,0.85,1,1.2,1.5$. Again, we see that a Dirac mass in the measure corresponds to a pole of the disc resolvent \emph{inside} the unit disc, which corresponds to a pole in the principal resolvent outside the interval $[-1,1]$.}\label{fig:basic2revisited}
 \end{center}
\end{figure}

\begin{figure}[!h]
 \begin{center}
  \includegraphics[height=\figheight,trim=0 -1 0 10]{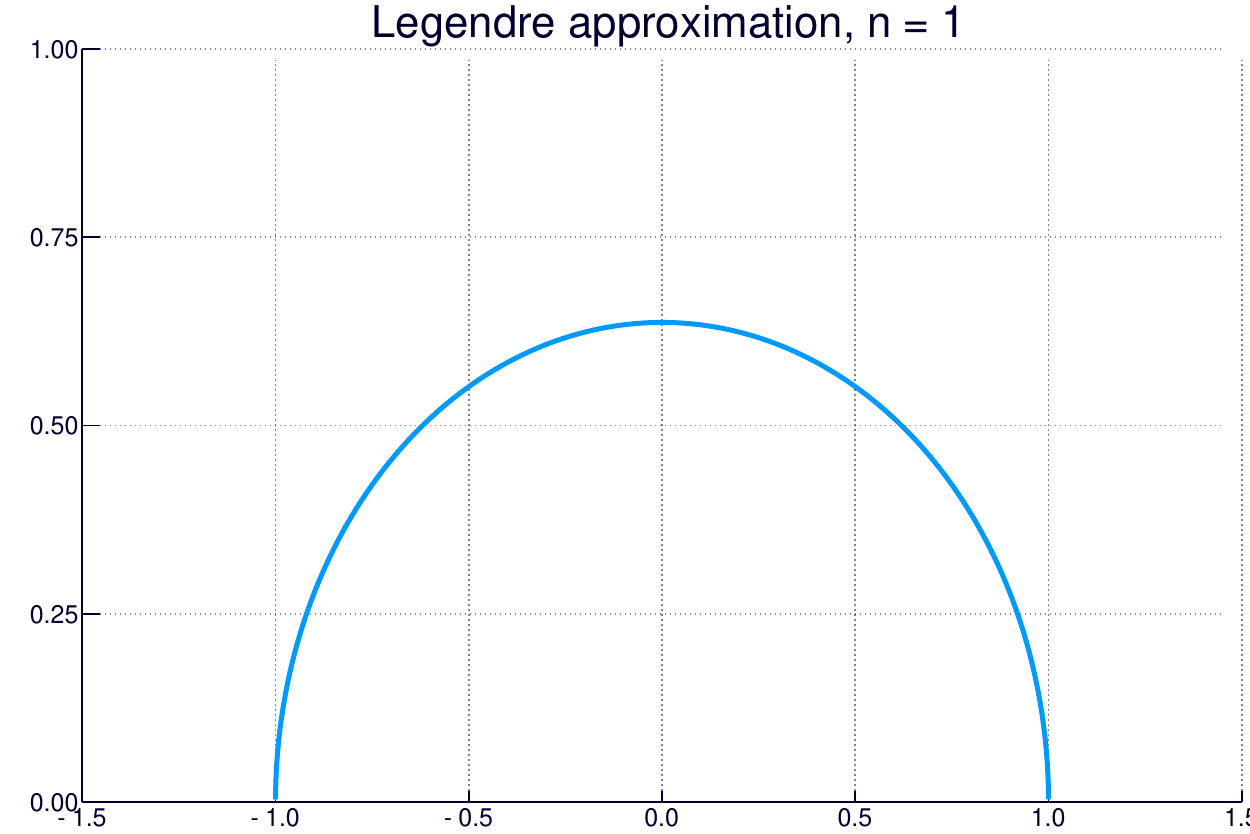}\hskip8pt\includegraphics[height=\figheight]{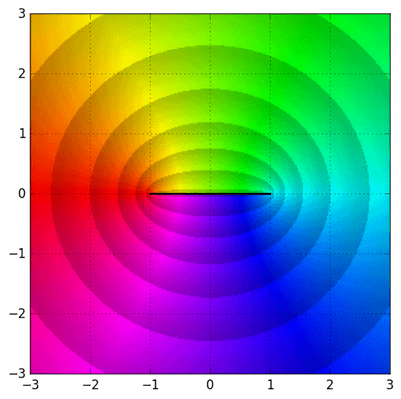}\hskip8pt\includegraphics[height=\figheight]{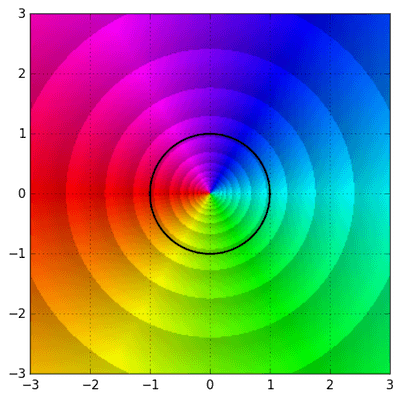}\\\vskip8pt
  \includegraphics[height=\figheight,trim=0 -1 0 10]{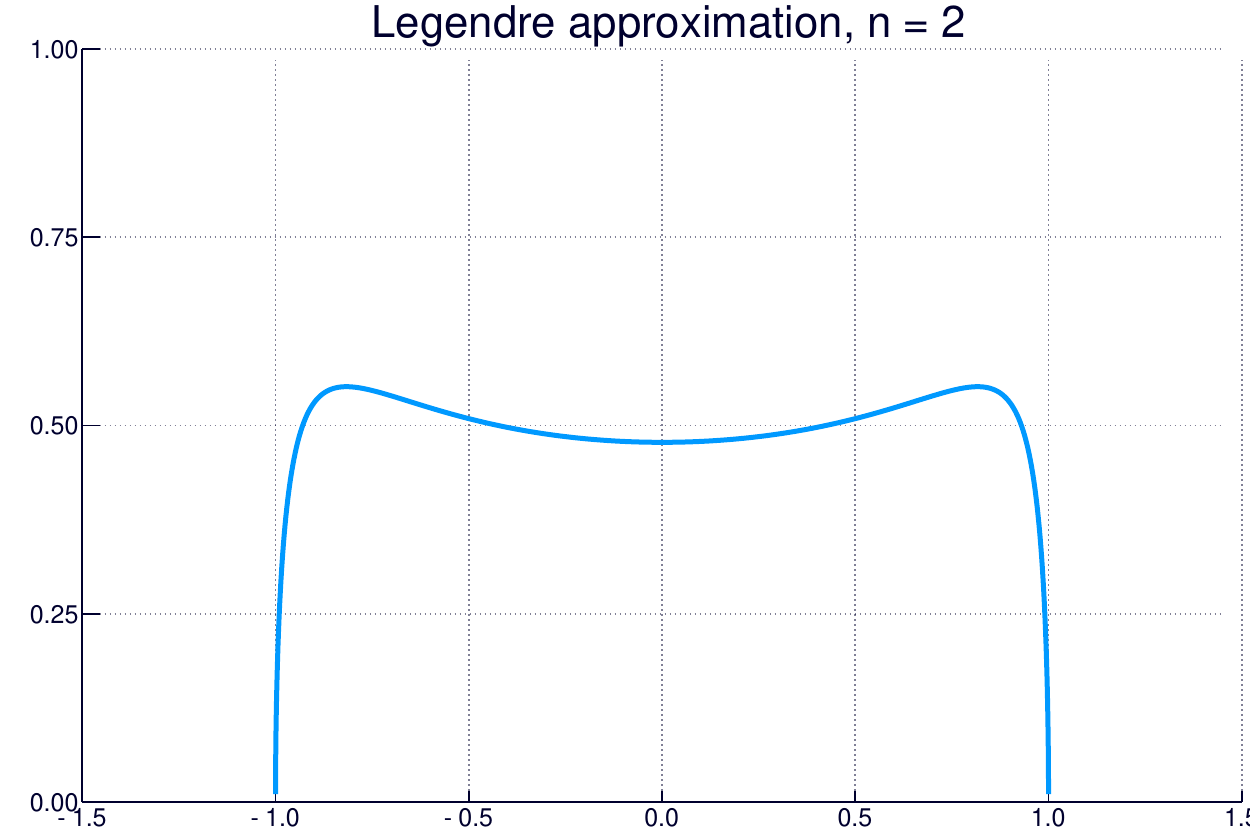}\hskip8pt\includegraphics[height=\figheight]{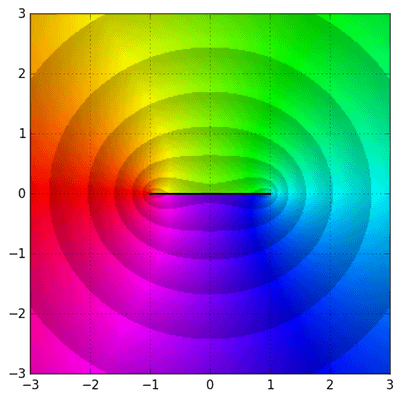}\hskip8pt\includegraphics[height=\figheight]{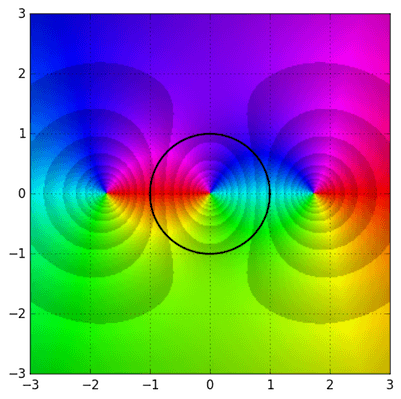}\\\vskip8pt
  \includegraphics[height=\figheight,trim=0 -1 0 10]{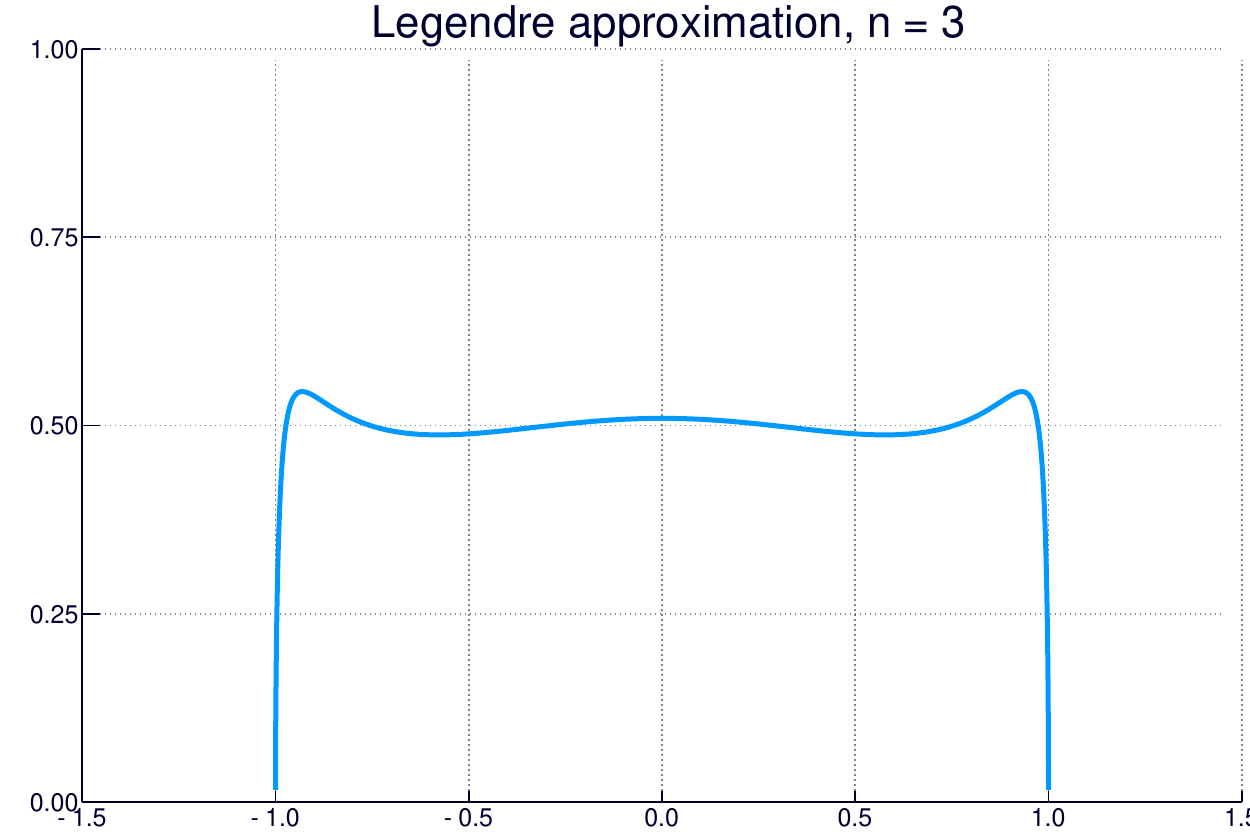}\hskip8pt\includegraphics[height=\figheight]{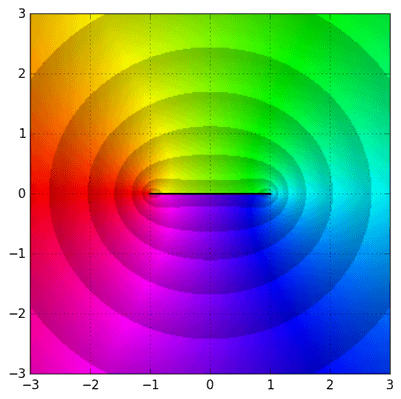}\hskip8pt\includegraphics[height=\figheight]{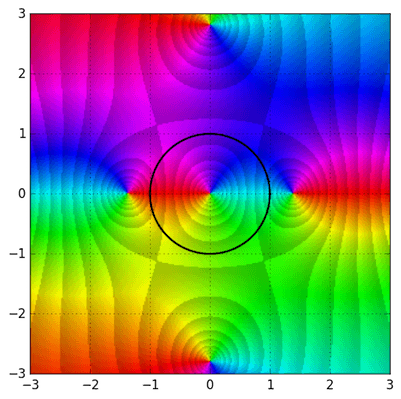}\\\vskip8pt
  \includegraphics[height=\figheight,trim=0 -1 0 10]{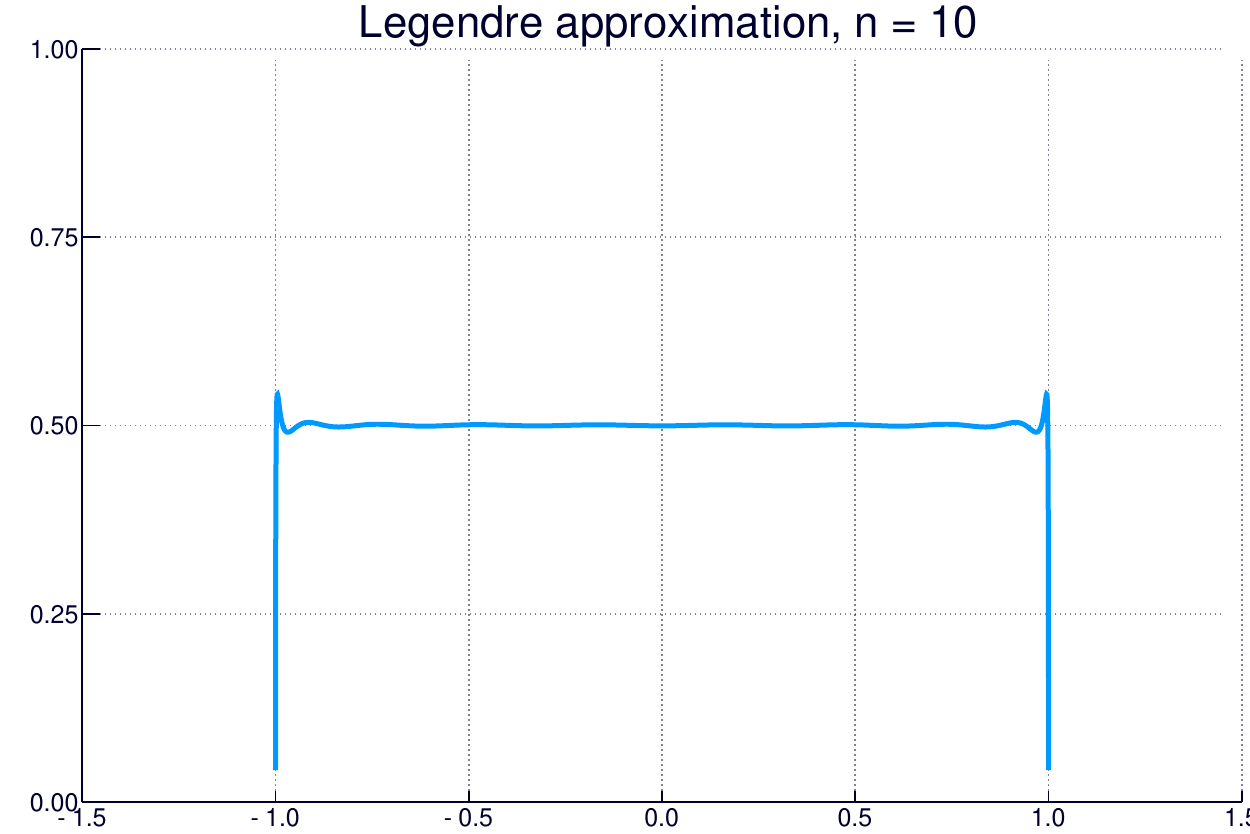}\hskip8pt\includegraphics[height=\figheight]{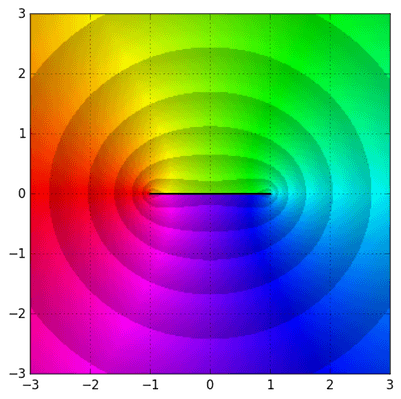}\hskip8pt\includegraphics[height=\figheight]{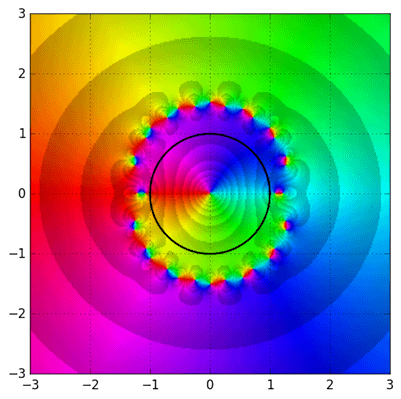}\\\vskip8pt
  \includegraphics[height=\figheight,trim=0 -1 0 10]{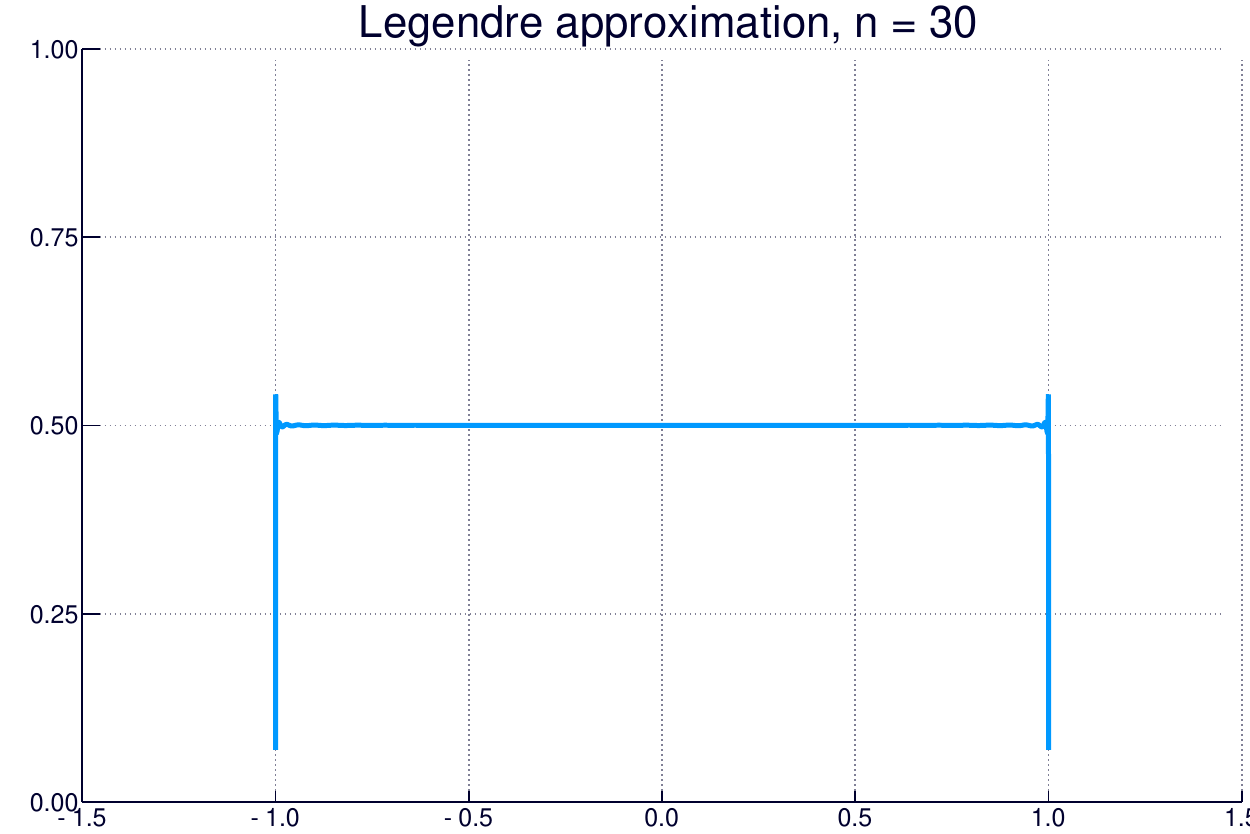}\hskip8pt\includegraphics[height=\figheight]{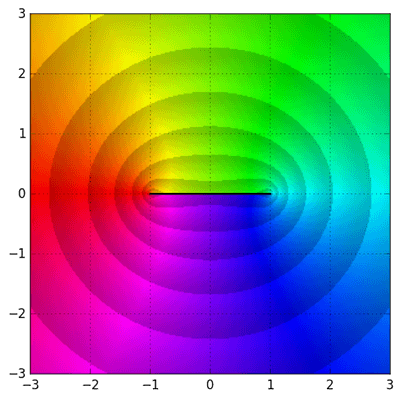}\hskip8pt\includegraphics[height=\figheight]{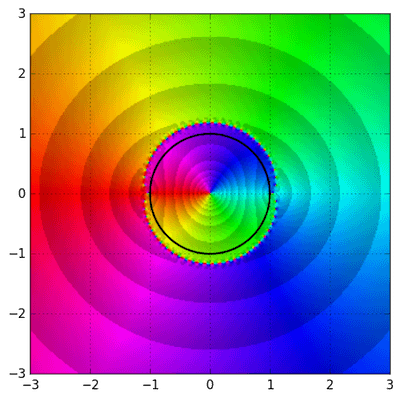}\\\vskip8pt
  \includegraphics[height=\figheight,trim=0 -1 0 10]{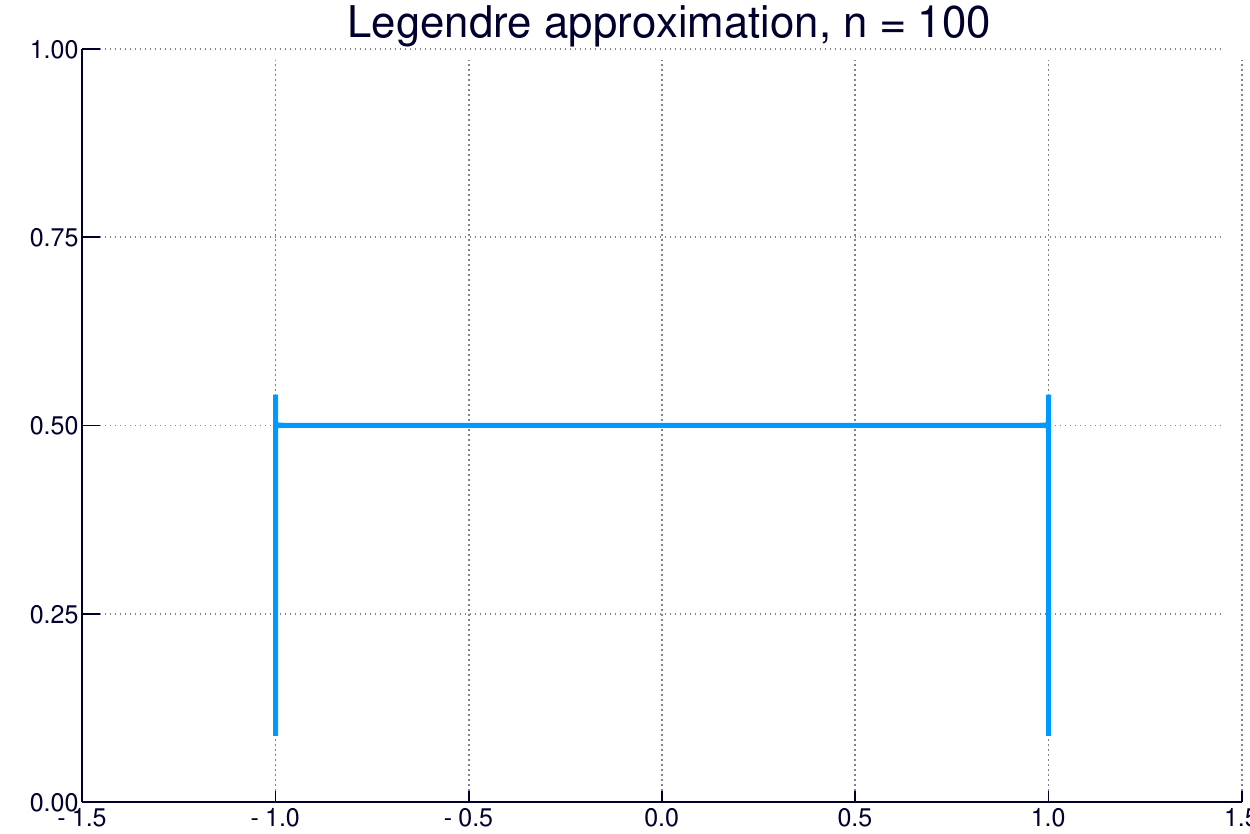}\hskip8pt\includegraphics[height=\figheight]{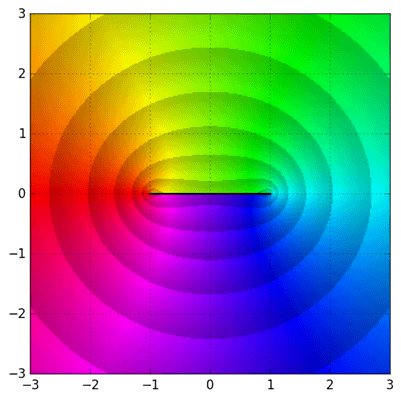}\hskip8pt\includegraphics[height=\figheight]{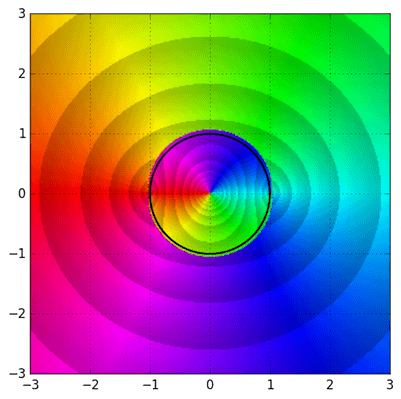}
\caption{These plots are of approximations to the spectral measure and principal resolvents of the Legendre polynomials, which has a Toeplitz-plus-trace-class Jacobi operator. The Jacobi operator can be found in Subsection \ref{subsec:Jacobipolys}. As the parameter $n$ of the approximation increases, a barrier around the unit circle forms. Also notice that a Gibbs phenomenon forms at the end points, showing that there are limitations to how good these approximations can be to the final measure.}\label{fig:legendre}
 \end{center}
\end{figure}

\begin{figure}[!h]
 \begin{center}
  \includegraphics[height=\figheight,trim=0 -1 0 10]{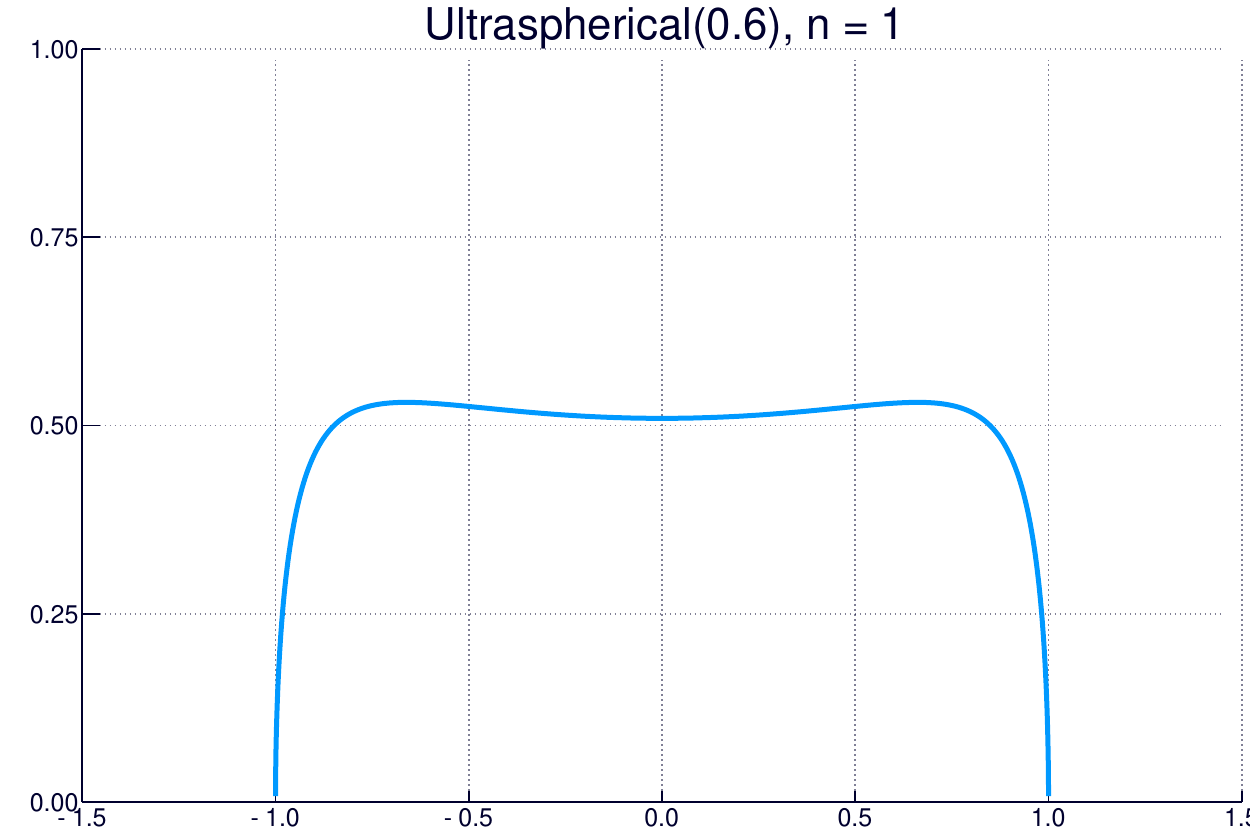}\hskip8pt\includegraphics[height=\figheight]{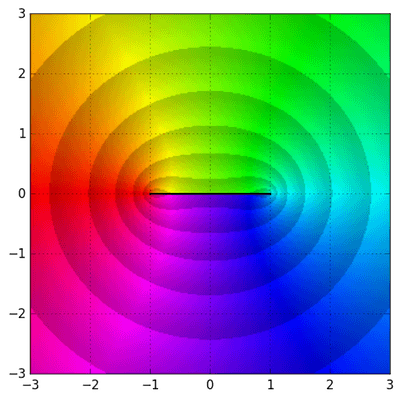}\hskip8pt\includegraphics[height=\figheight]{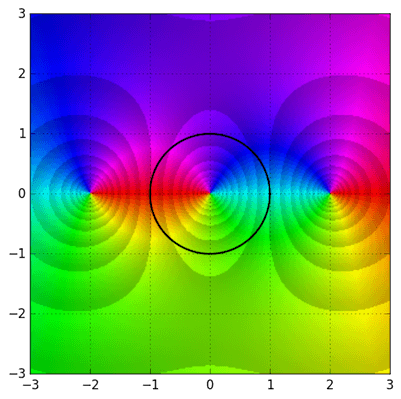}\\\vskip8pt
  \includegraphics[height=\figheight,trim=0 -1 0 10]{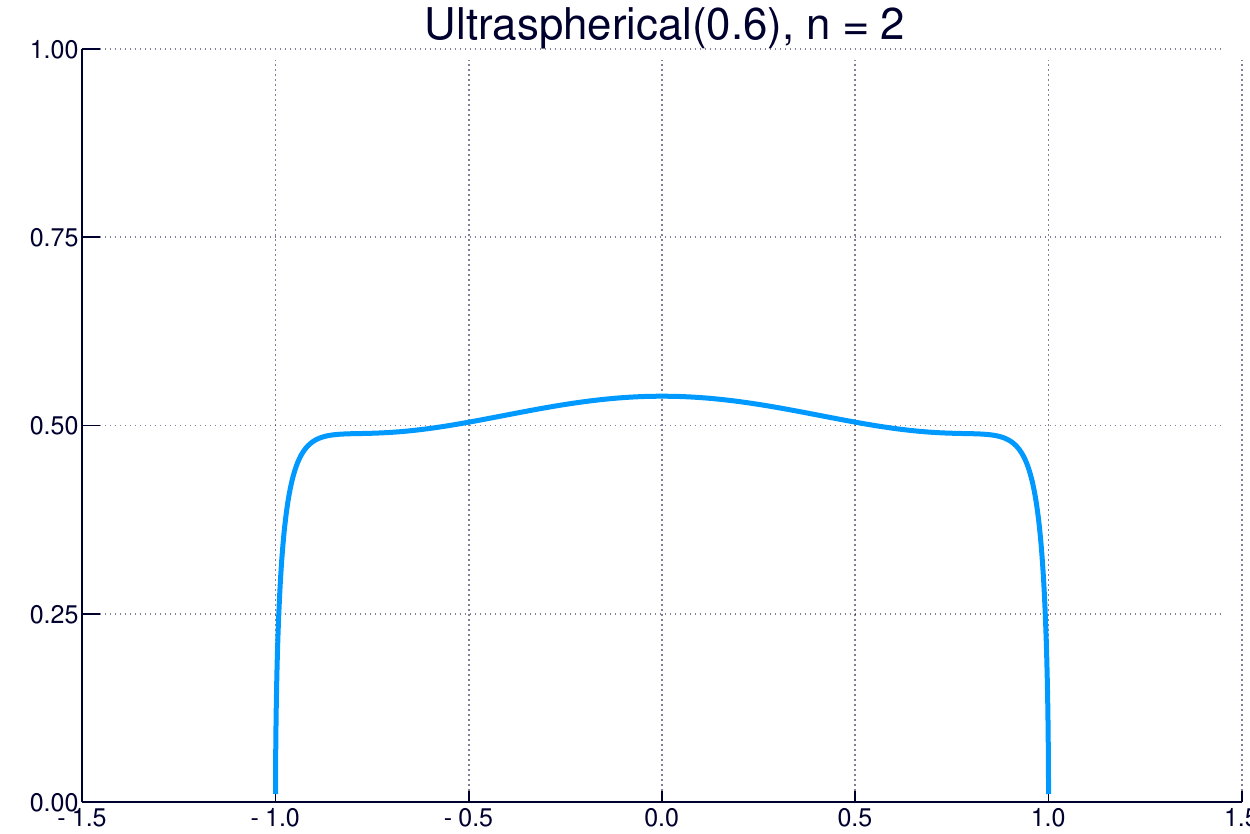}\hskip8pt\includegraphics[height=\figheight]{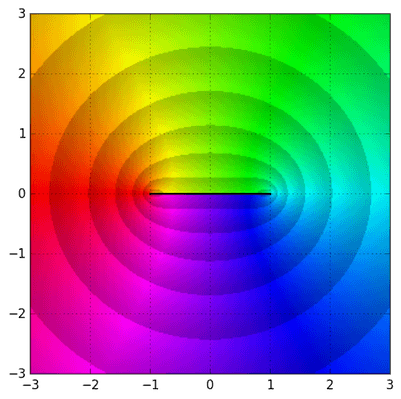}\hskip8pt\includegraphics[height=\figheight]{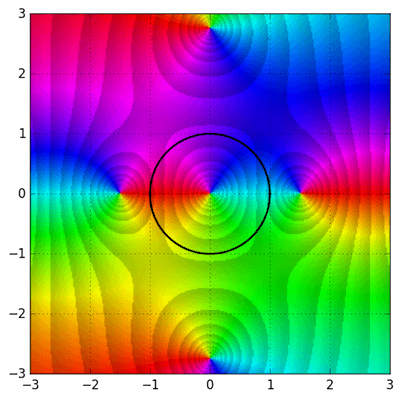}\\\vskip8pt
  \includegraphics[height=\figheight,trim=0 -1 0 10]{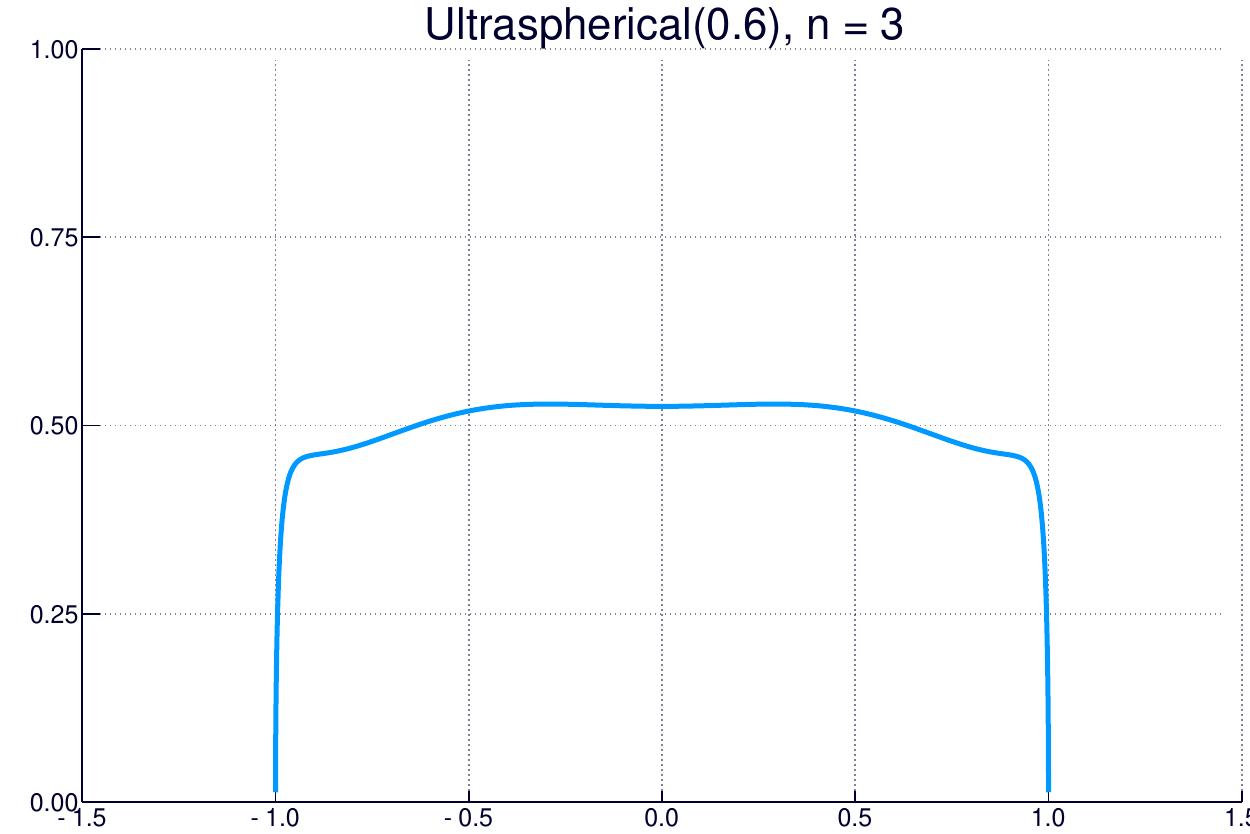}\hskip8pt\includegraphics[height=\figheight]{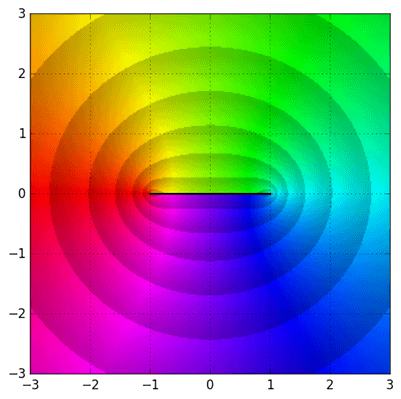}\hskip8pt\includegraphics[height=\figheight]{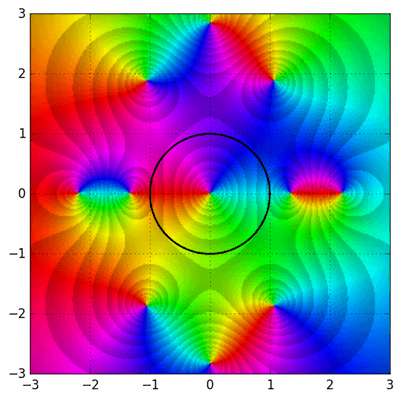}\\\vskip8pt
  \includegraphics[height=\figheight,trim=0 -1 0 10]{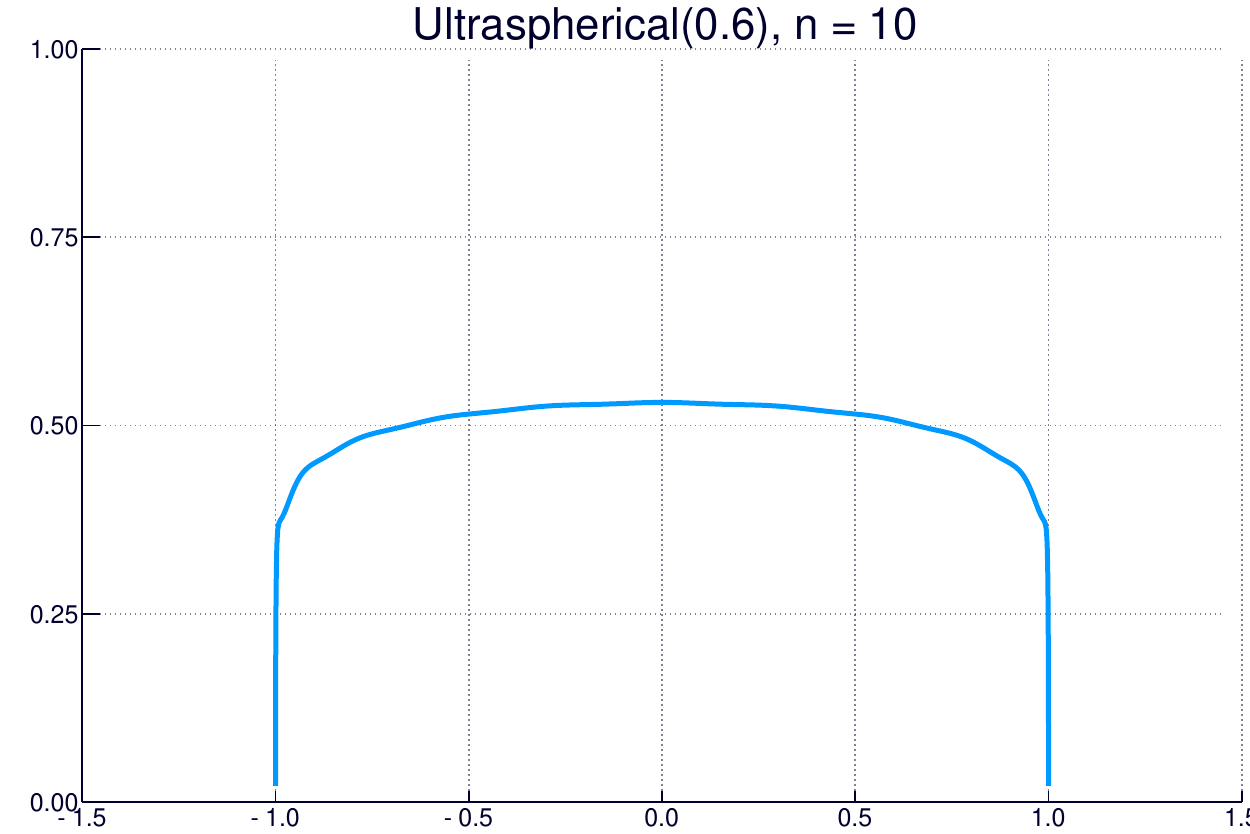}\hskip8pt\includegraphics[height=\figheight]{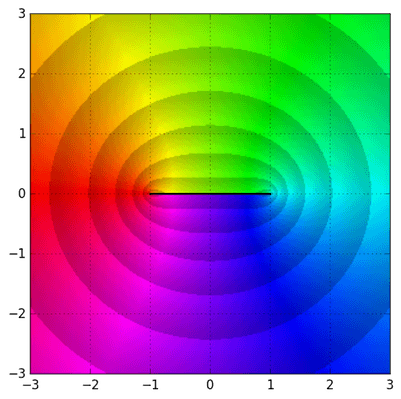}\hskip8pt\includegraphics[height=\figheight]{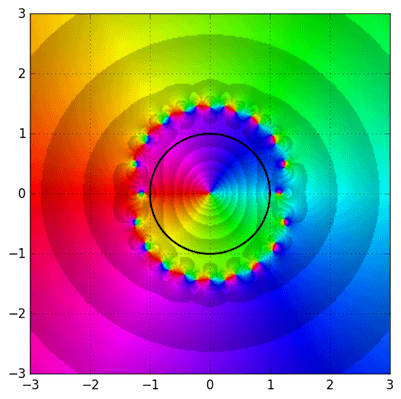}\\\vskip8pt
  \includegraphics[height=\figheight,trim=0 -1 0 10]{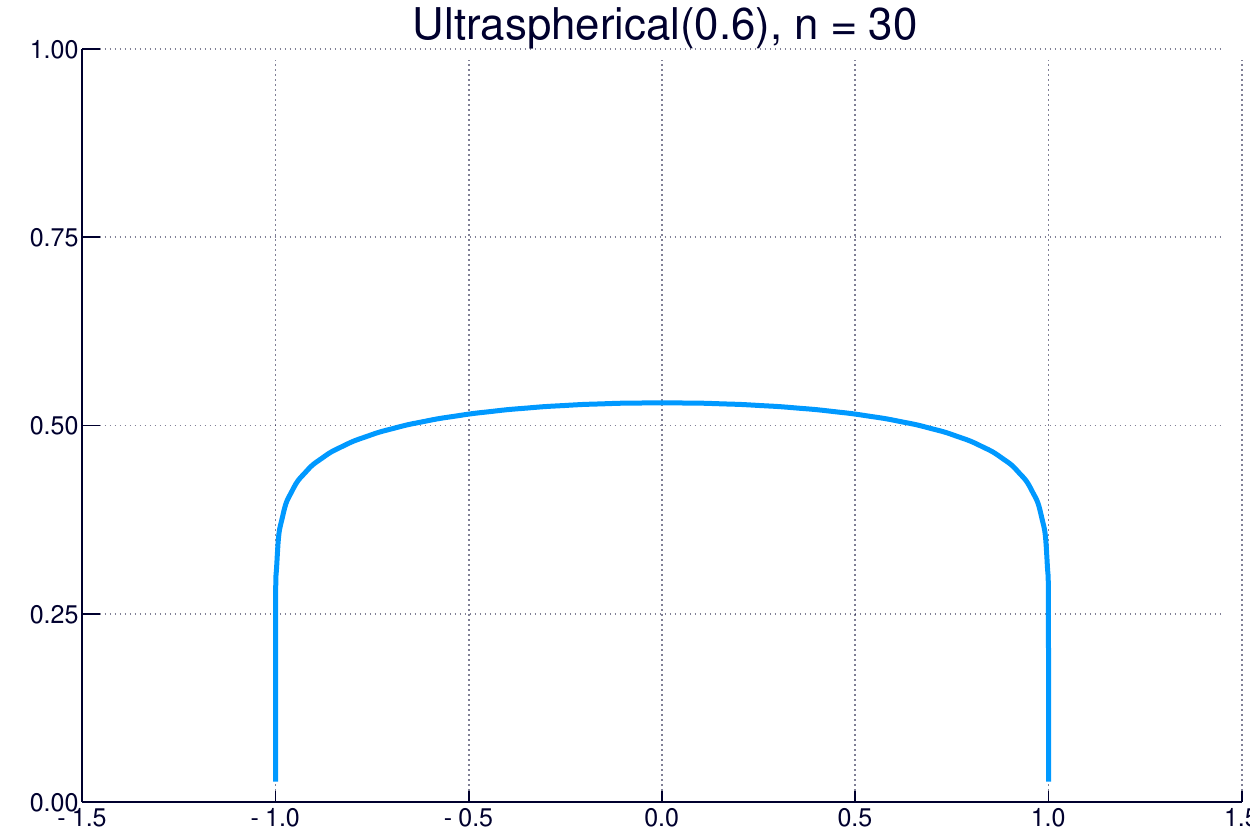}\hskip8pt\includegraphics[height=\figheight]{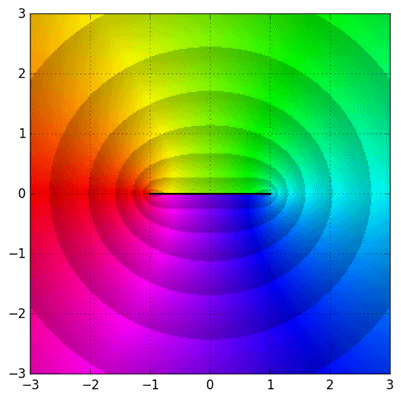}\hskip8pt\includegraphics[height=\figheight]{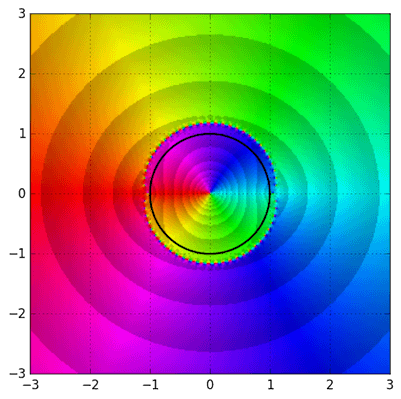}\\\vskip8pt
  \includegraphics[height=\figheight,trim=0 -1 0 10]{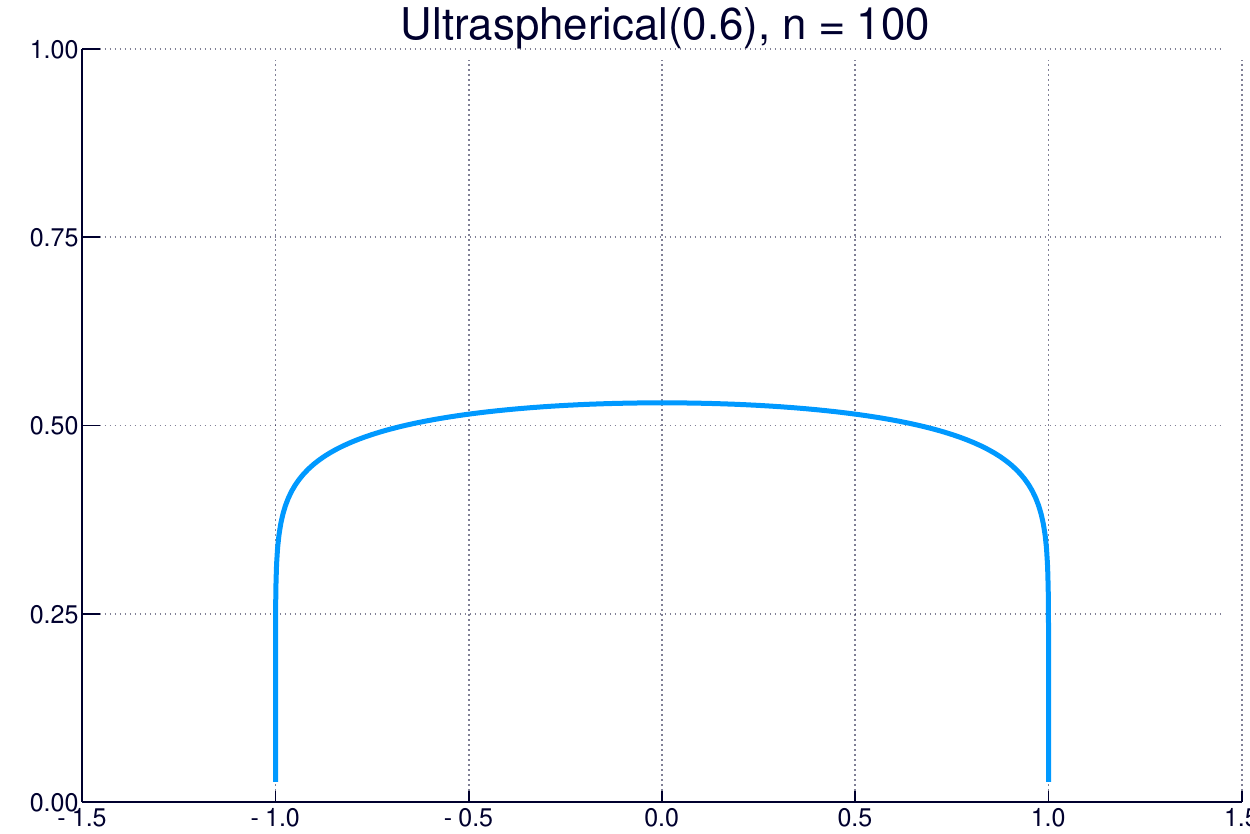}\hskip8pt\includegraphics[height=\figheight]{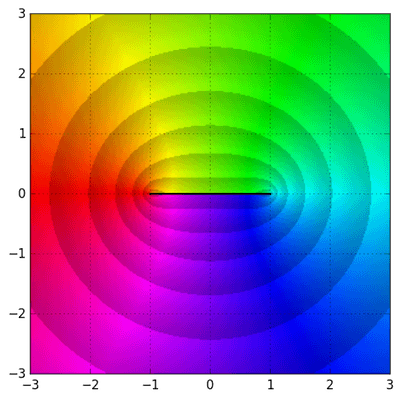}\hskip8pt\includegraphics[height=\figheight]{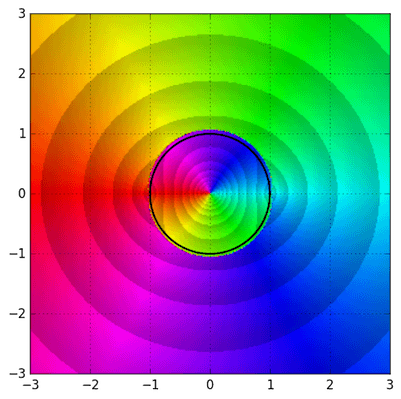}
\caption{These plots are of approximations to the spectral measure and principal resolvents of the Ultraspherical polynomials with parameter $\gamma = 0.6$, which has a Toeplitz-plus-trace-class Jacobi operator. The Jacobi operator can be found in Subsection \ref{subsec:Jacobipolys}. As the parameter $n$ of the approximation increases, a barrier around the unit circle forms.}\label{fig:ultraspherical}
 \end{center}
\end{figure}

\begin{figure}[!h]
 \begin{center}
  \includegraphics[height=\figheight,trim=0 -1 0 10]{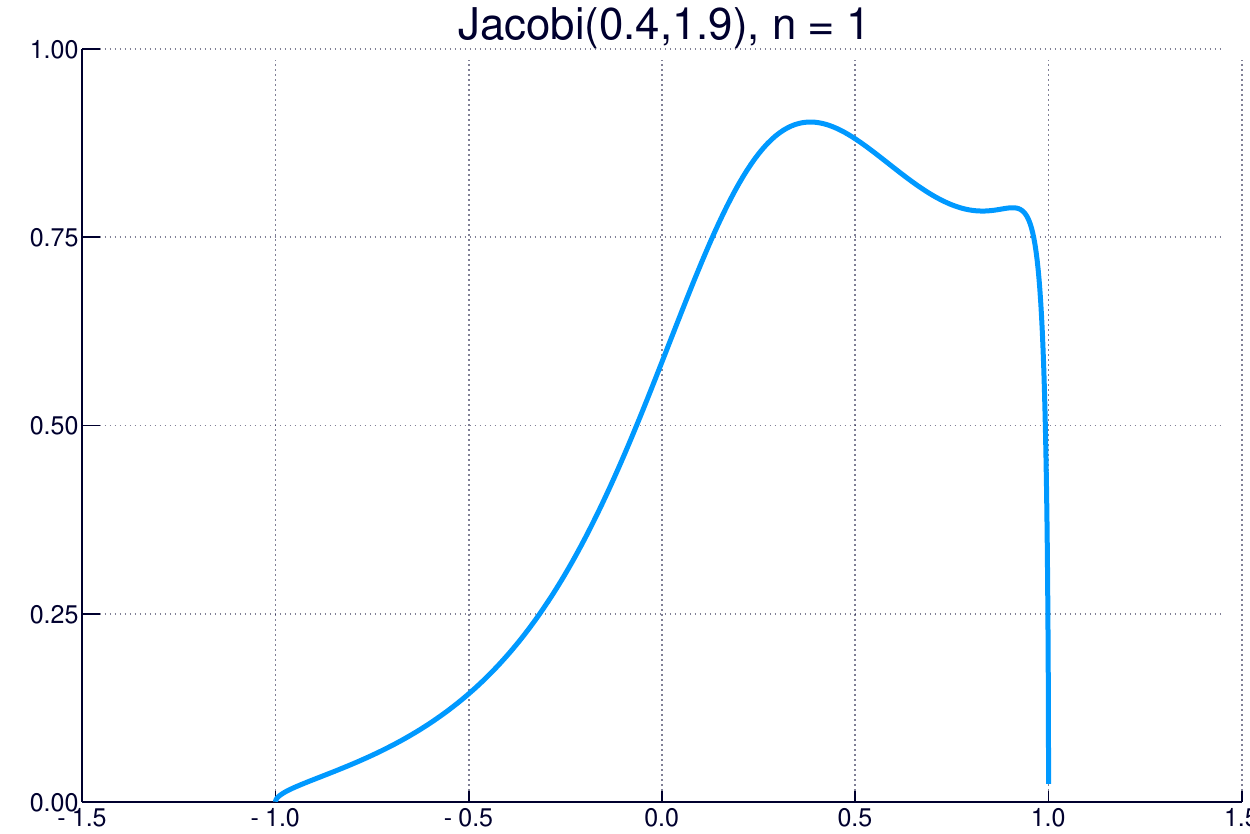}\hskip8pt\includegraphics[height=\figheight]{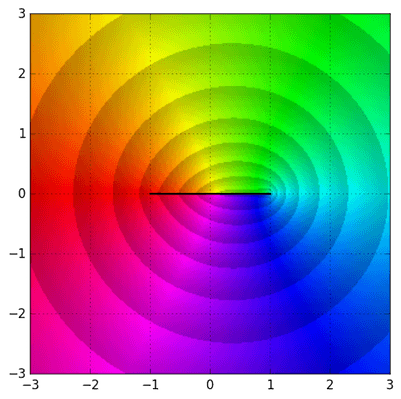}\hskip8pt\includegraphics[height=\figheight]{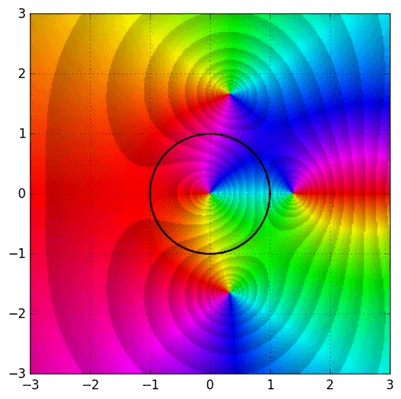}\\\vskip8pt
  \includegraphics[height=\figheight,trim=0 -1 0 10]{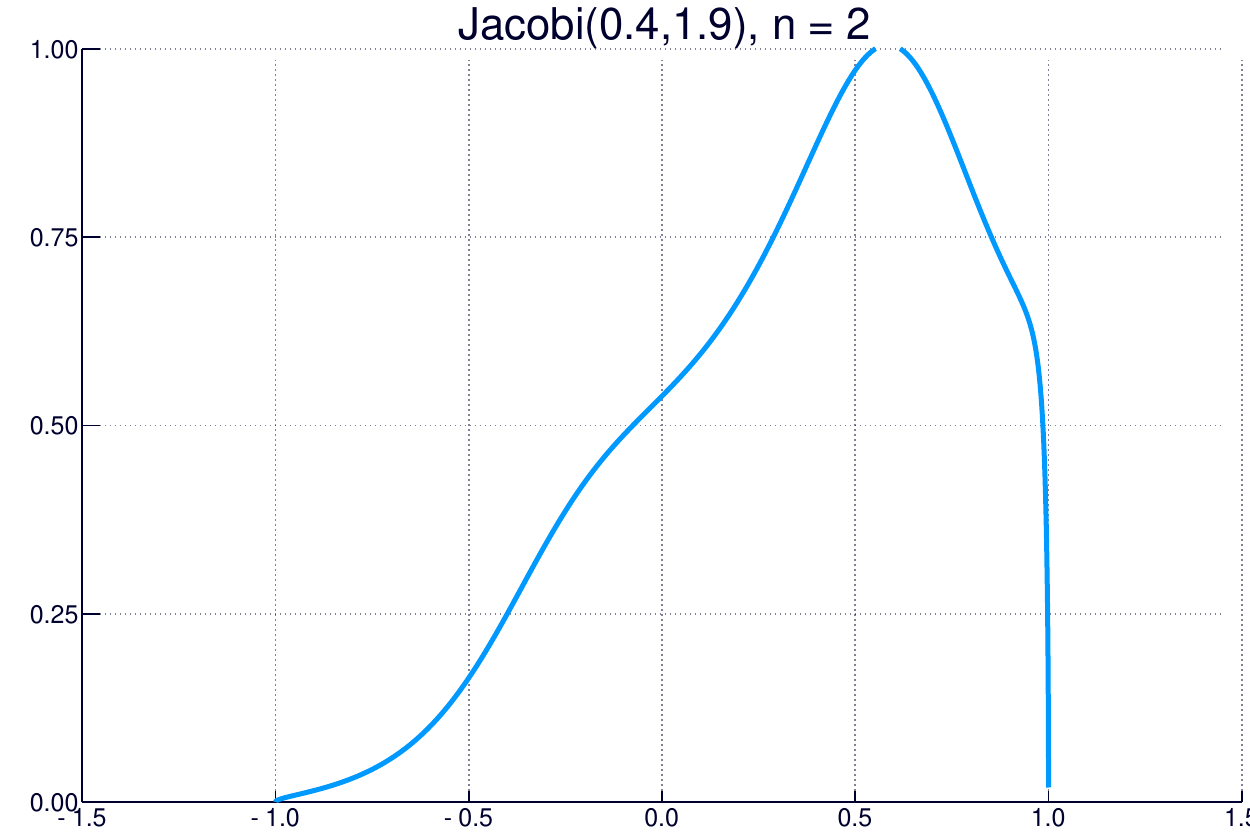}\hskip8pt\includegraphics[height=\figheight]{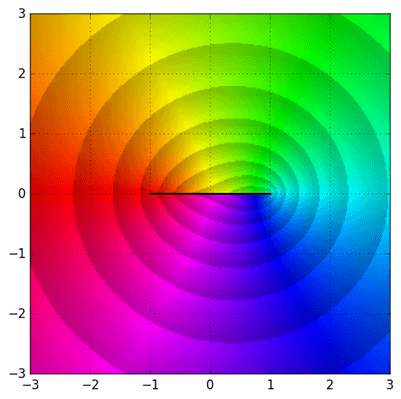}\hskip8pt\includegraphics[height=\figheight]{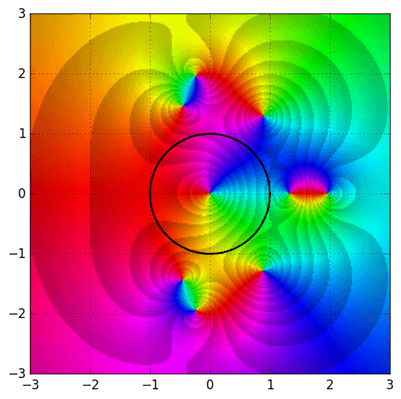}\\\vskip8pt
  \includegraphics[height=\figheight,trim=0 -1 0 10]{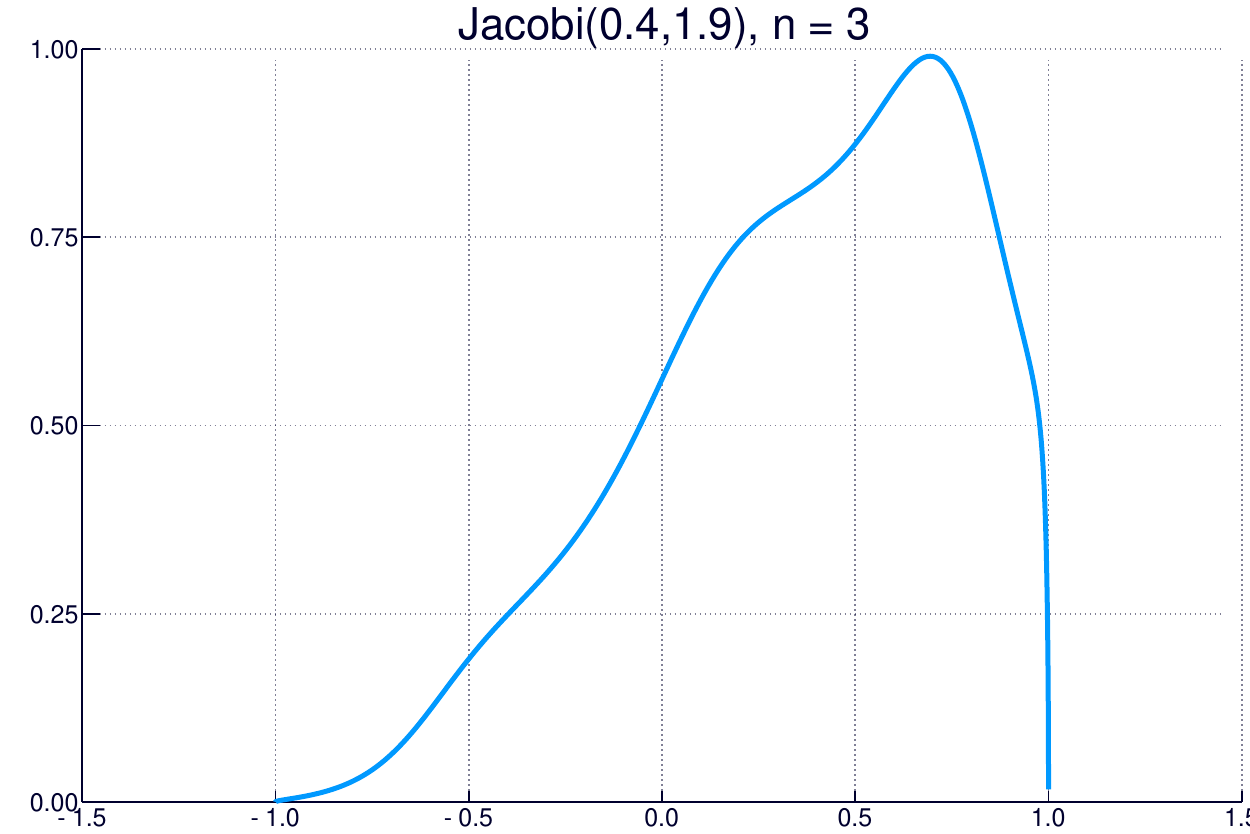}\hskip8pt\includegraphics[height=\figheight]{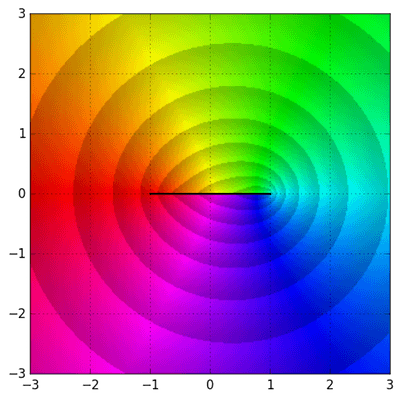}\hskip8pt\includegraphics[height=\figheight]{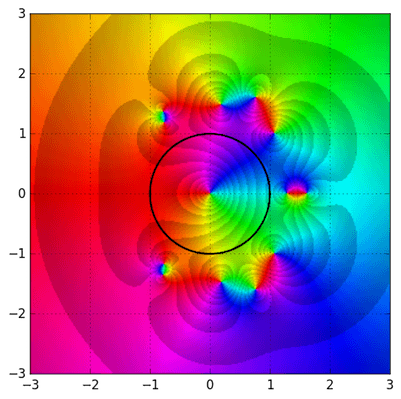}\\\vskip8pt
  \includegraphics[height=\figheight,trim=0 -1 0 10]{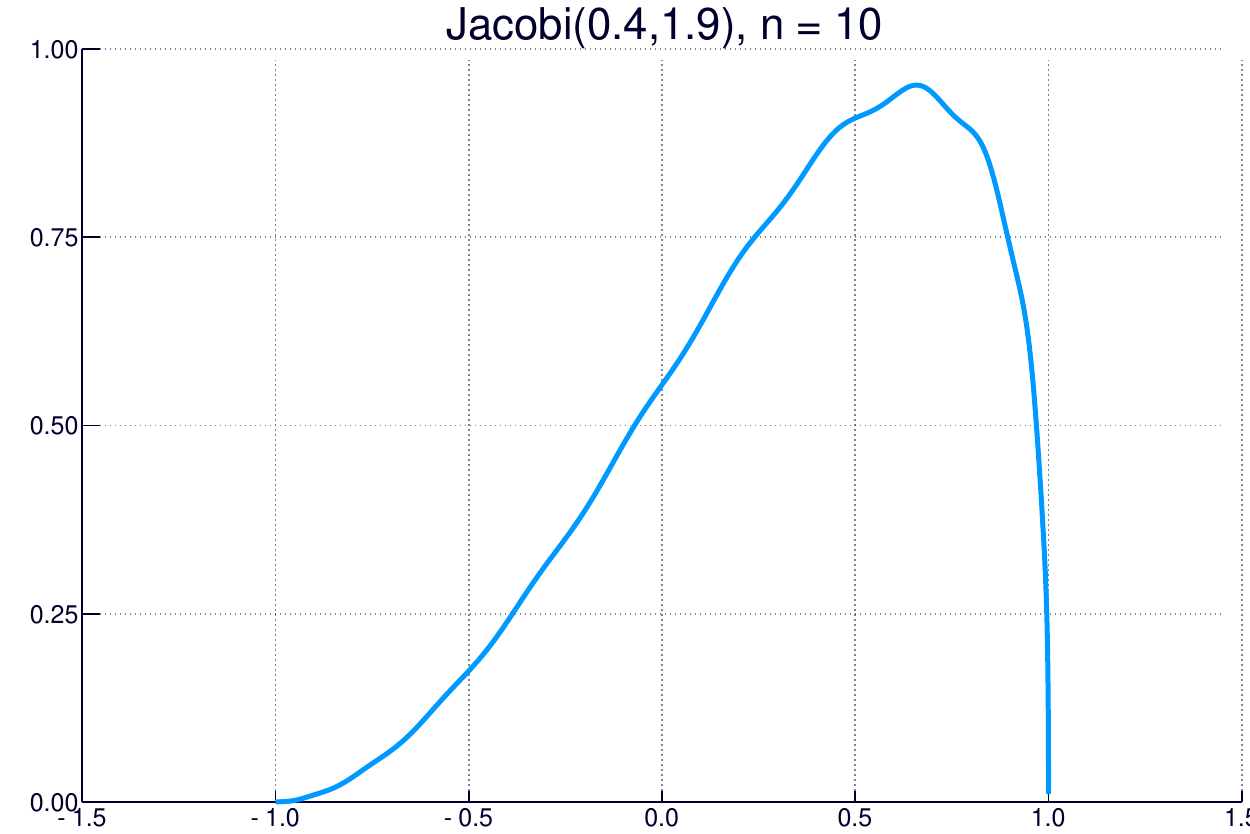}\hskip8pt\includegraphics[height=\figheight]{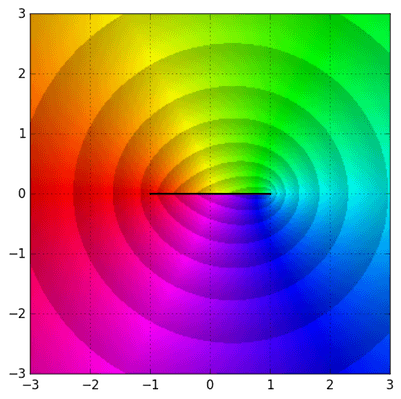}\hskip8pt\includegraphics[height=\figheight]{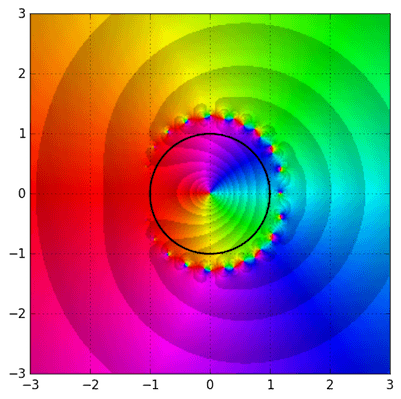}\\\vskip8pt
  \includegraphics[height=\figheight,trim=0 -1 0 10]{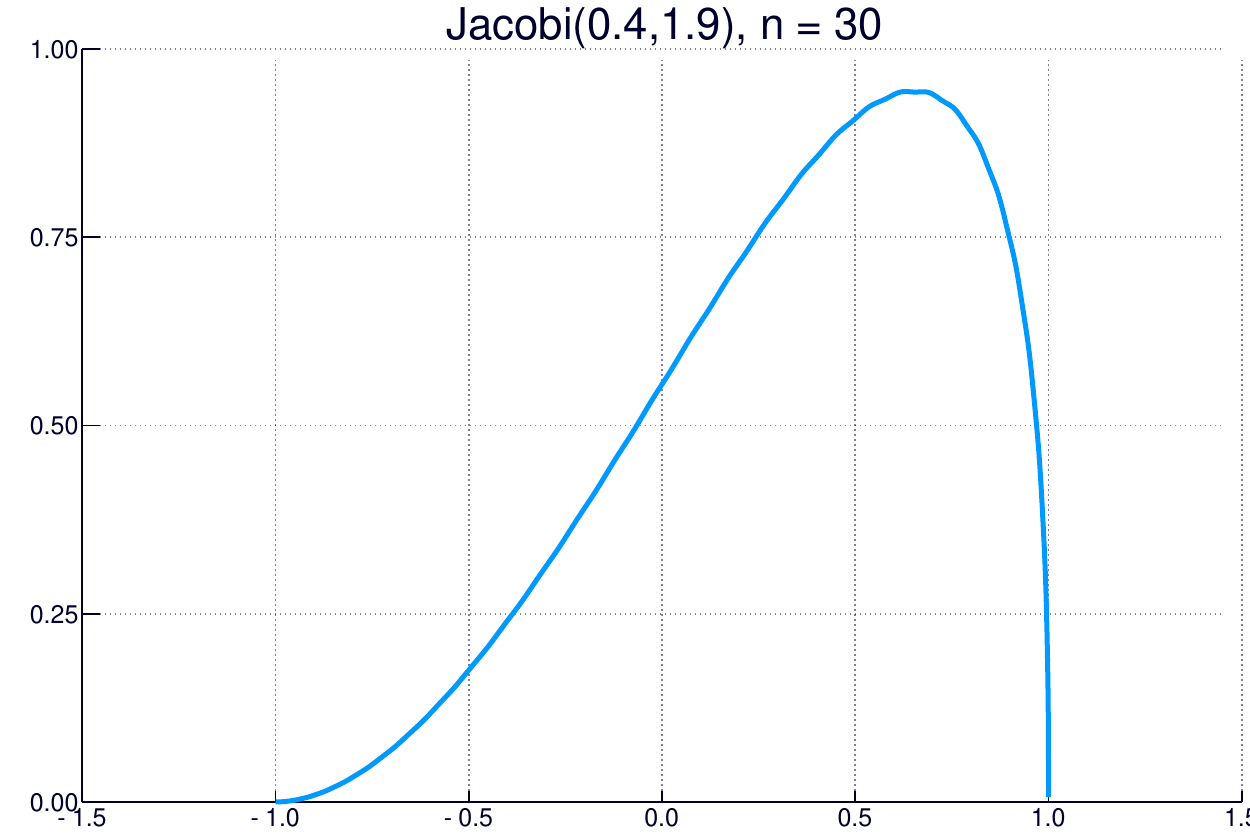}\hskip8pt\includegraphics[height=\figheight]{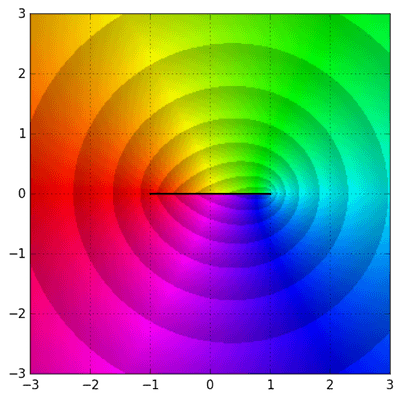}\hskip8pt\includegraphics[height=\figheight]{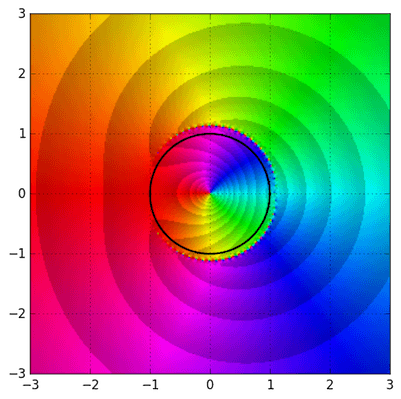}\\\vskip8pt
  \includegraphics[height=\figheight,trim=0 -1 0 10]{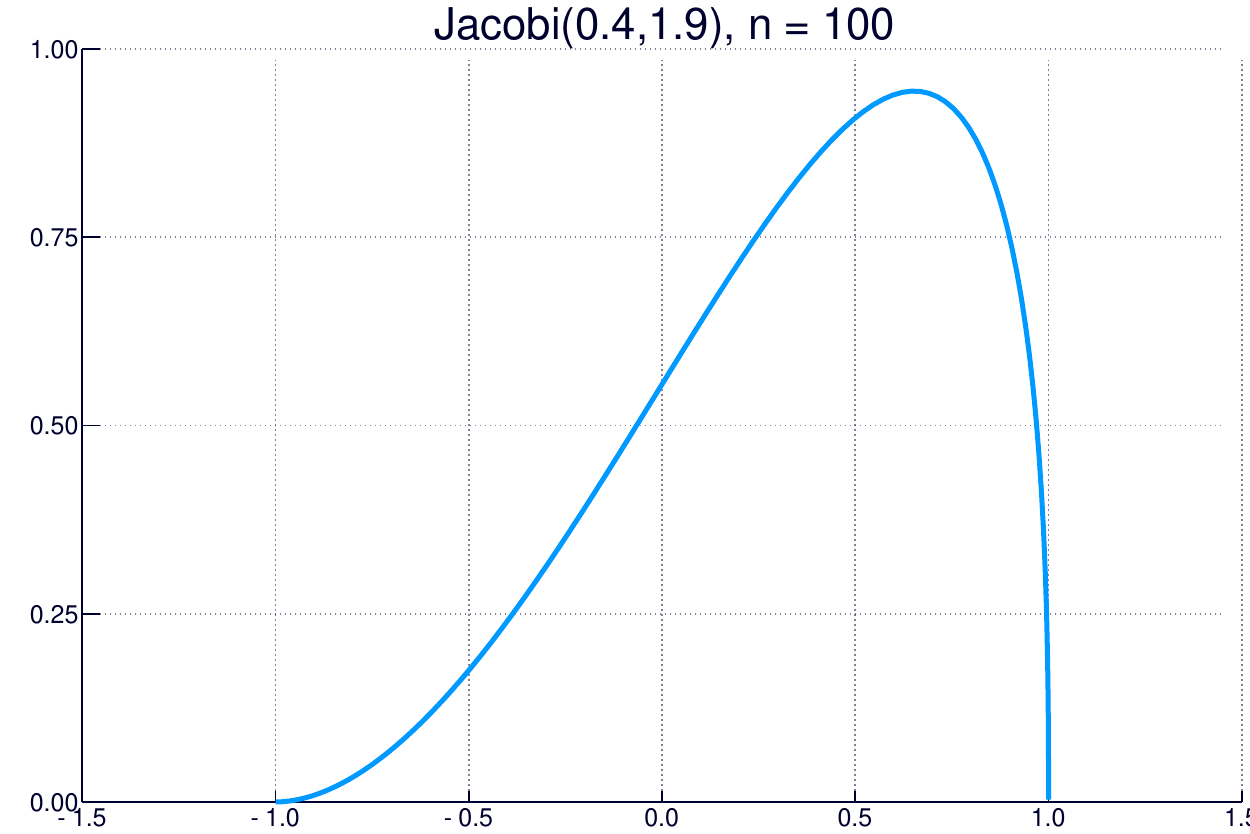}\hskip8pt\includegraphics[height=\figheight]{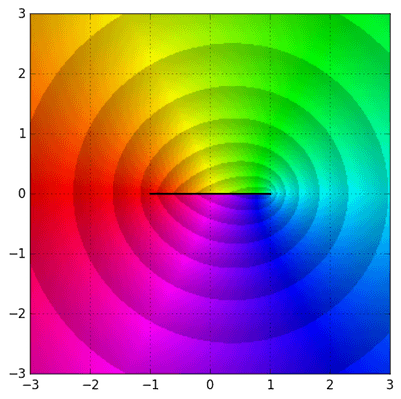}\hskip8pt\includegraphics[height=\figheight]{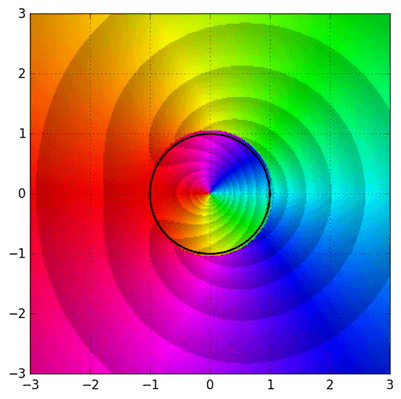}
\caption{These plots are of approximations to the spectral measure and principal resolvents of the Jacobi polynomials with parameter $\alpha,\beta = 0.4,1.9$, which has a Toeplitz-plus-trace-class Jacobi operator. The Jacobi operator can be found in Subsection \ref{subsec:Jacobipolys}. As the parameter $n$ of the approximation increases, a barrier around the unit circle forms.}\label{fig:jacobi}
 \end{center}
\end{figure}

\begin{figure}[!h]
 \begin{center}
  \includegraphics[height=\figheight,trim=0 -1 0 10]{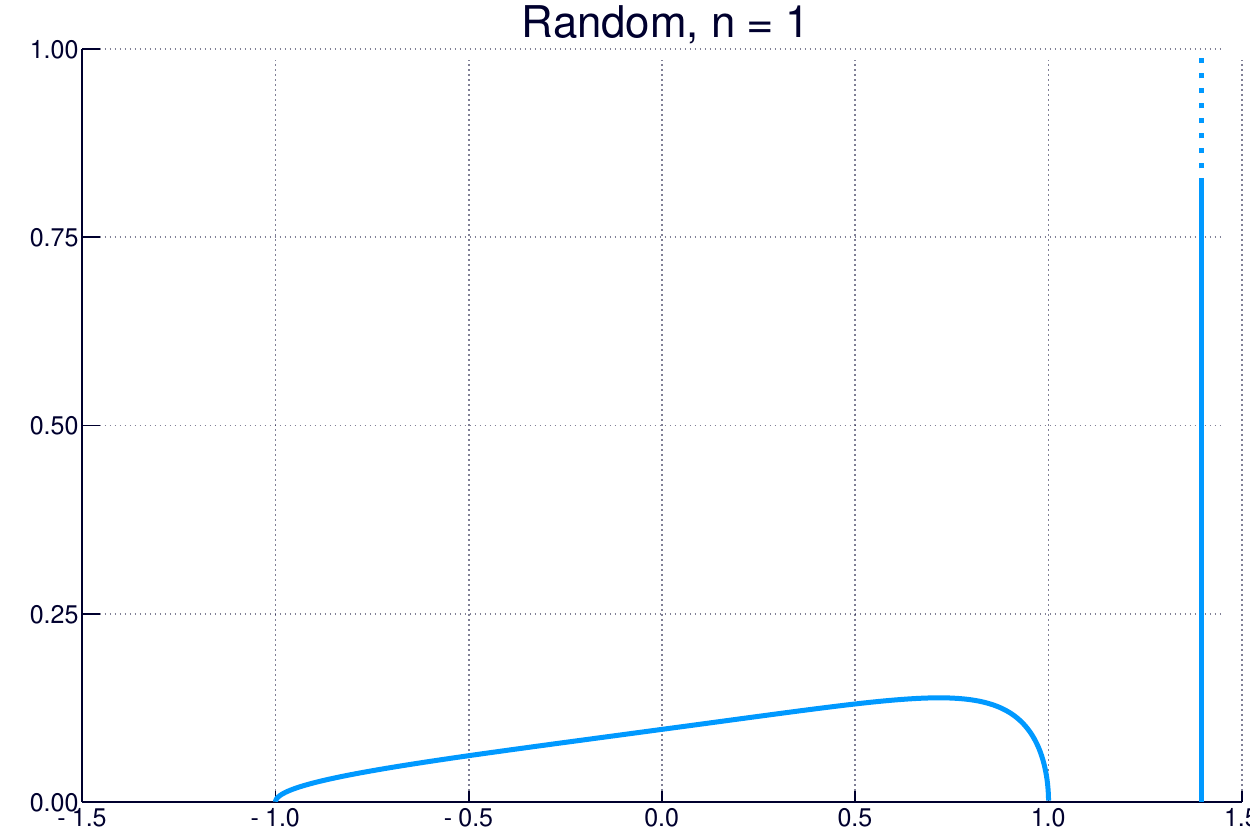}\hskip8pt\includegraphics[height=\figheight]{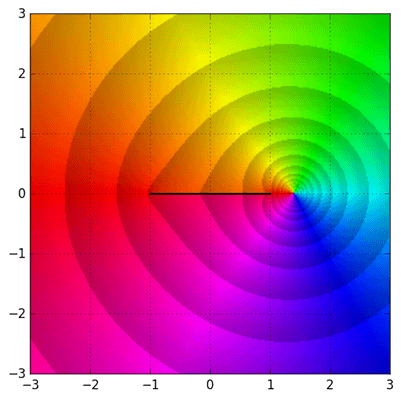}\hskip8pt\includegraphics[height=\figheight]{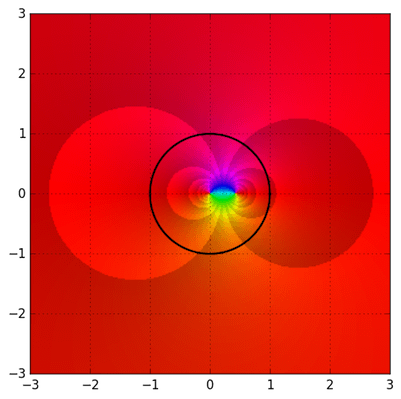}\\\vskip8pt
  \includegraphics[height=\figheight,trim=0 -1 0 10]{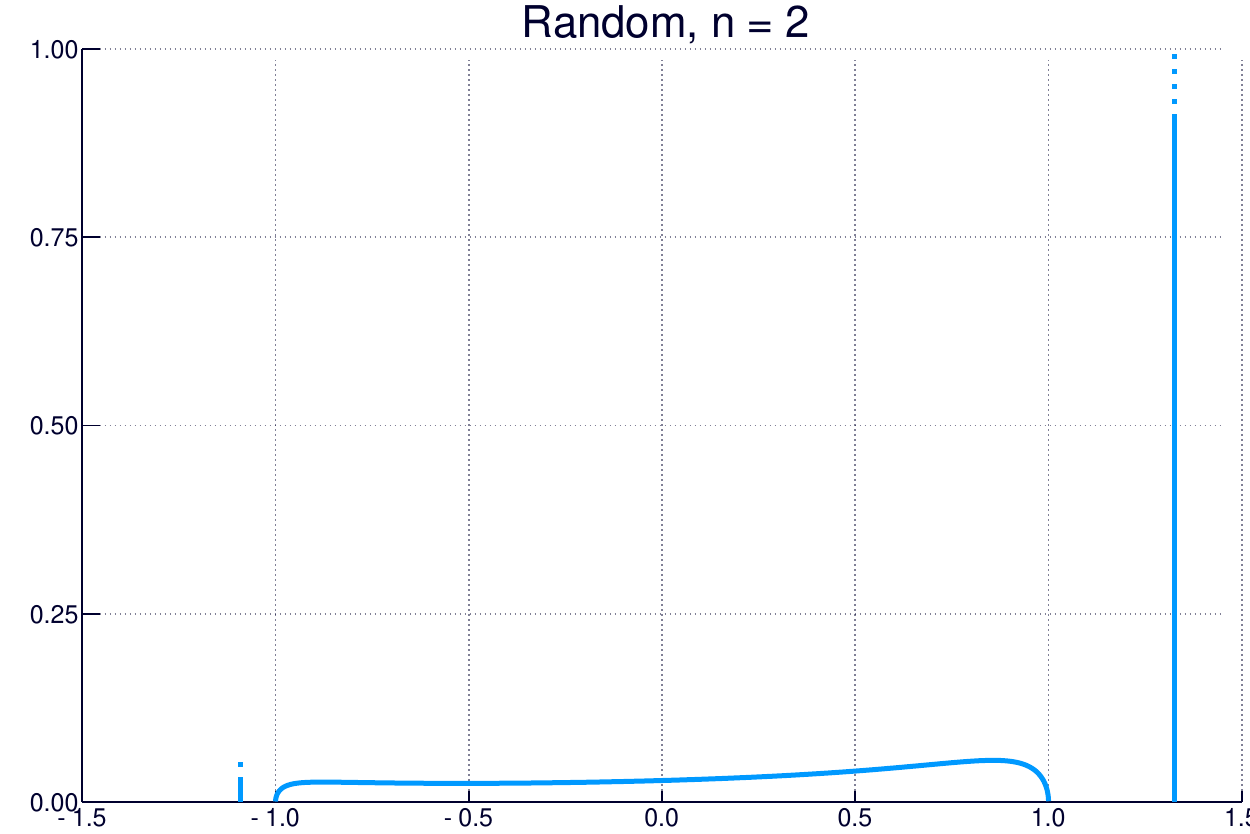}\hskip8pt\includegraphics[height=\figheight]{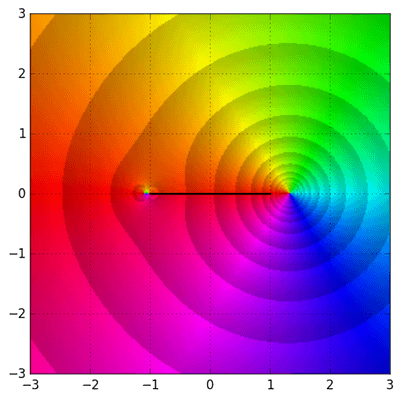}\hskip8pt\includegraphics[height=\figheight]{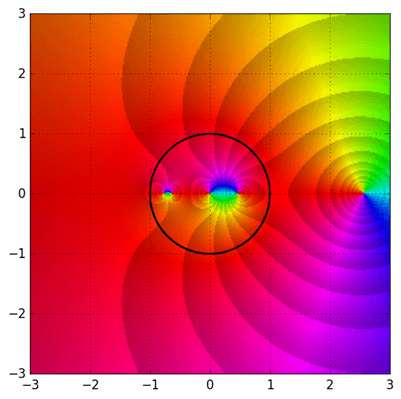}\\\vskip8pt
  \includegraphics[height=\figheight,trim=0 -1 0 10]{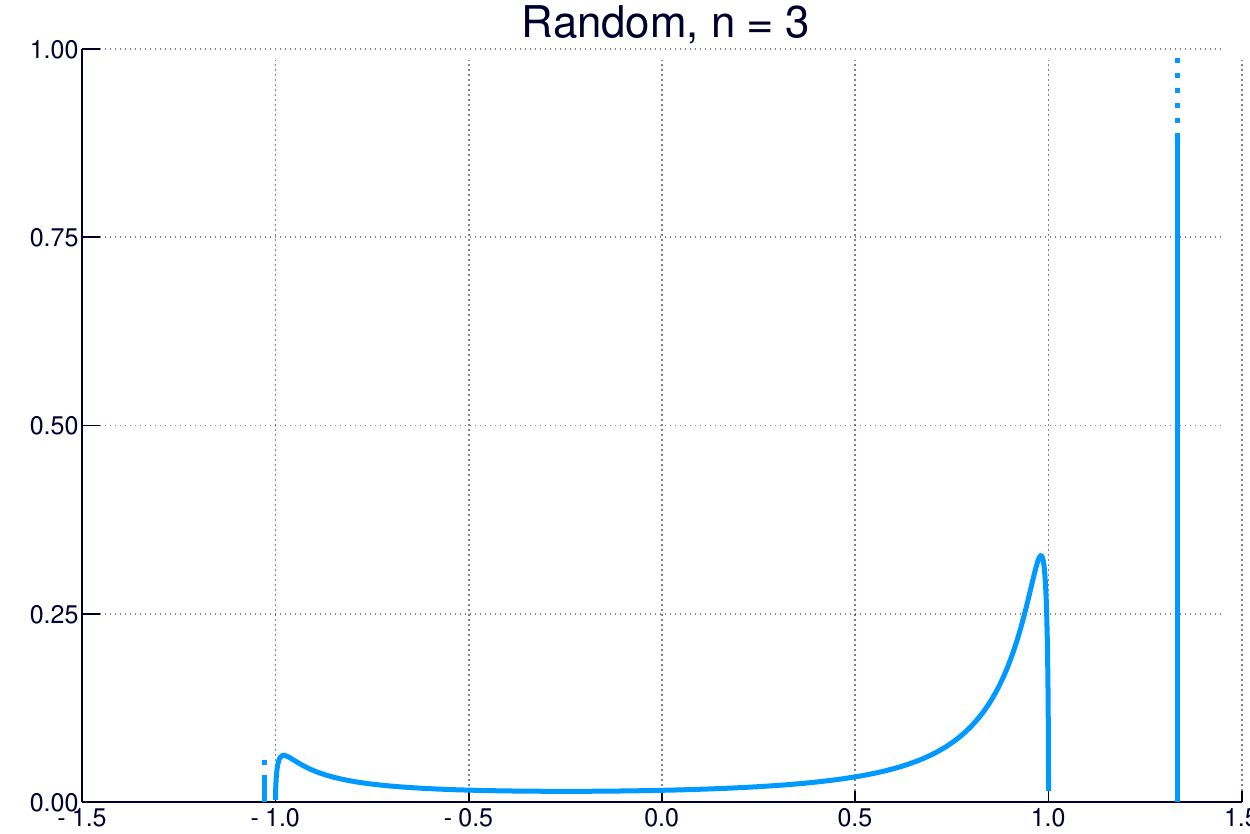}\hskip8pt\includegraphics[height=\figheight]{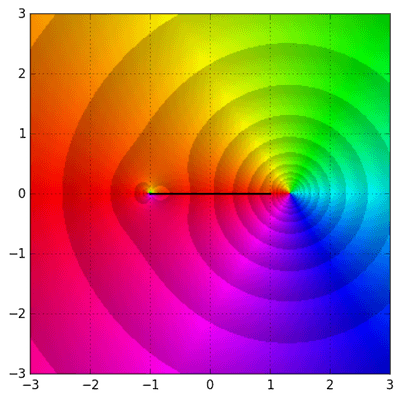}\hskip8pt\includegraphics[height=\figheight]{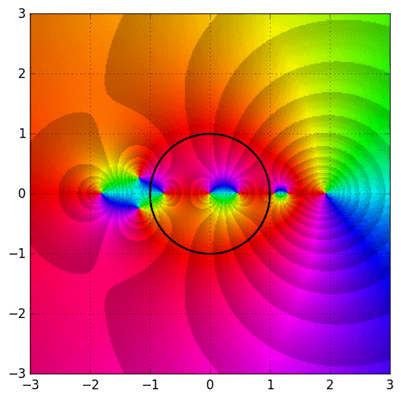}\\\vskip8pt
  \includegraphics[height=\figheight,trim=0 -1 0 10]{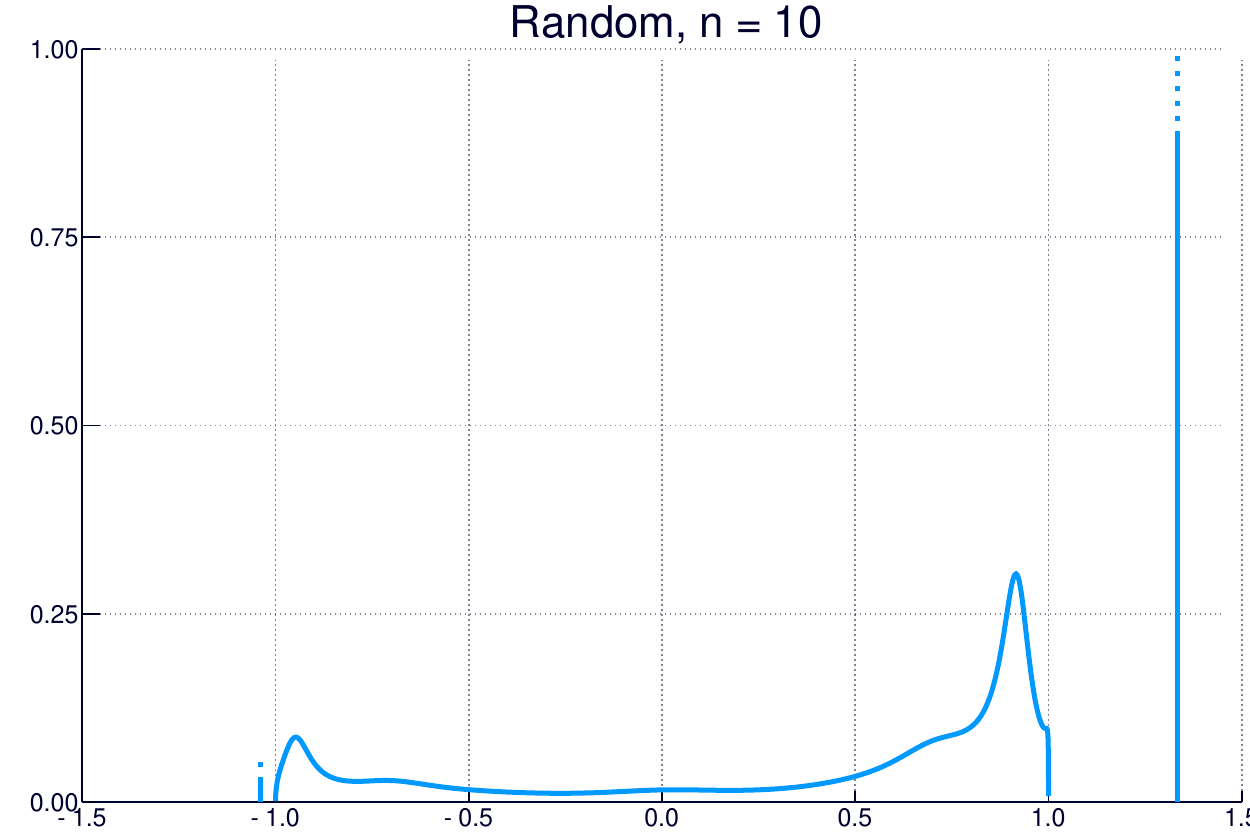}\hskip8pt\includegraphics[height=\figheight]{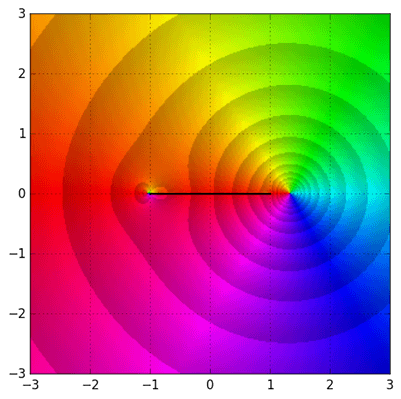}\hskip8pt\includegraphics[height=\figheight]{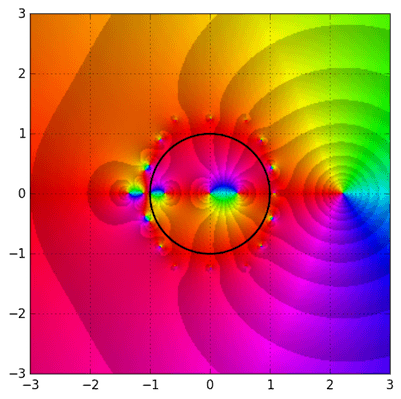}\\\vskip8pt
  \includegraphics[height=\figheight,trim=0 -1 0 10]{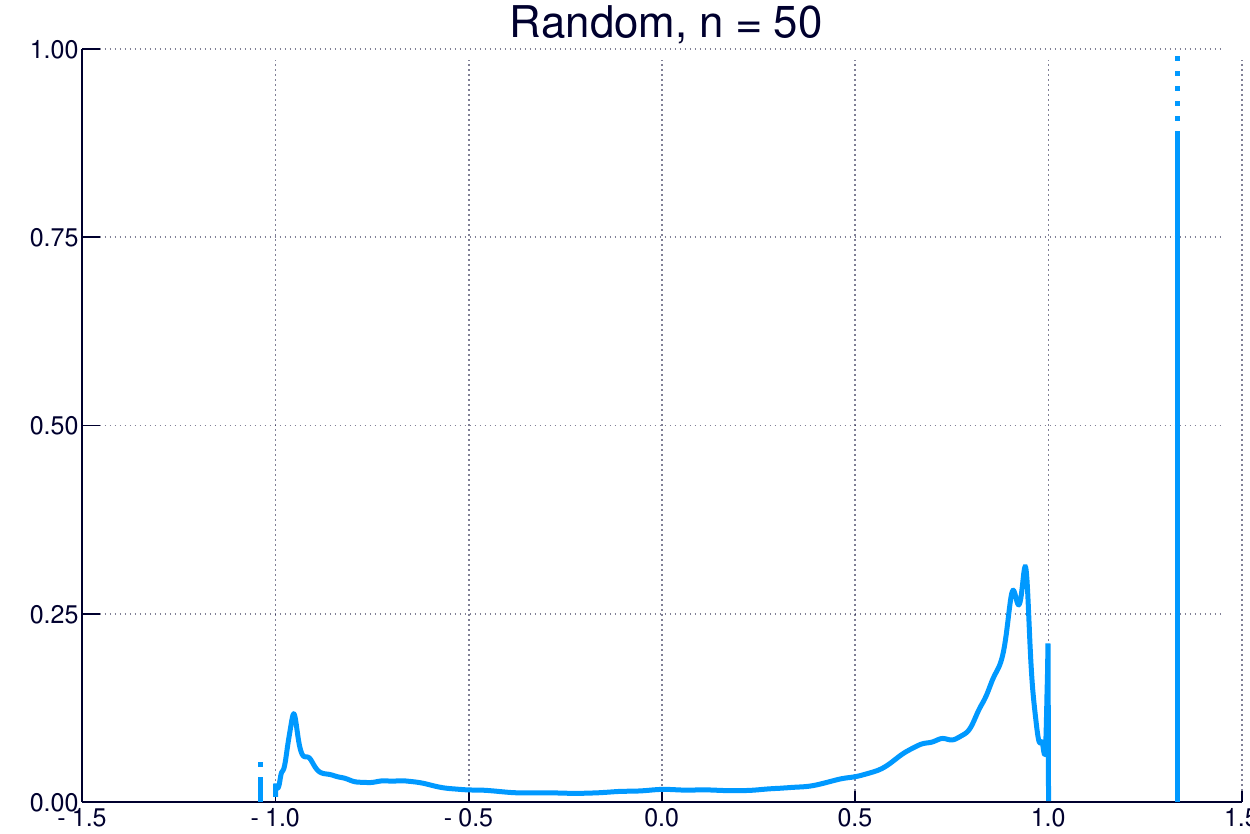}\hskip8pt\includegraphics[height=\figheight]{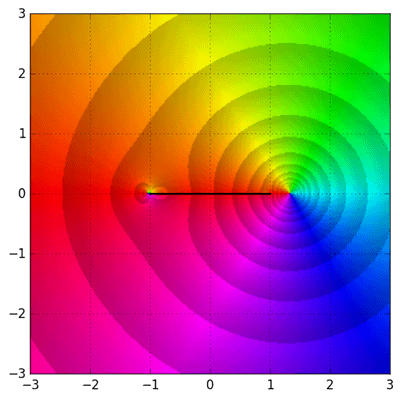}\hskip8pt\includegraphics[height=\figheight]{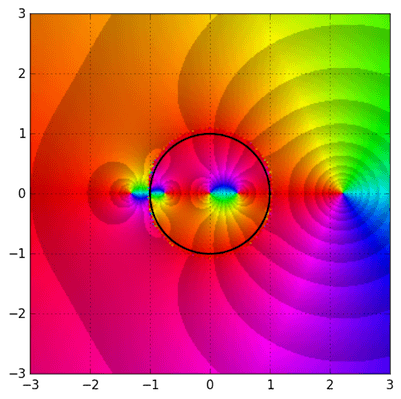}\\\vskip8pt
  \includegraphics[height=\figheight,trim=0 -1 0 10]{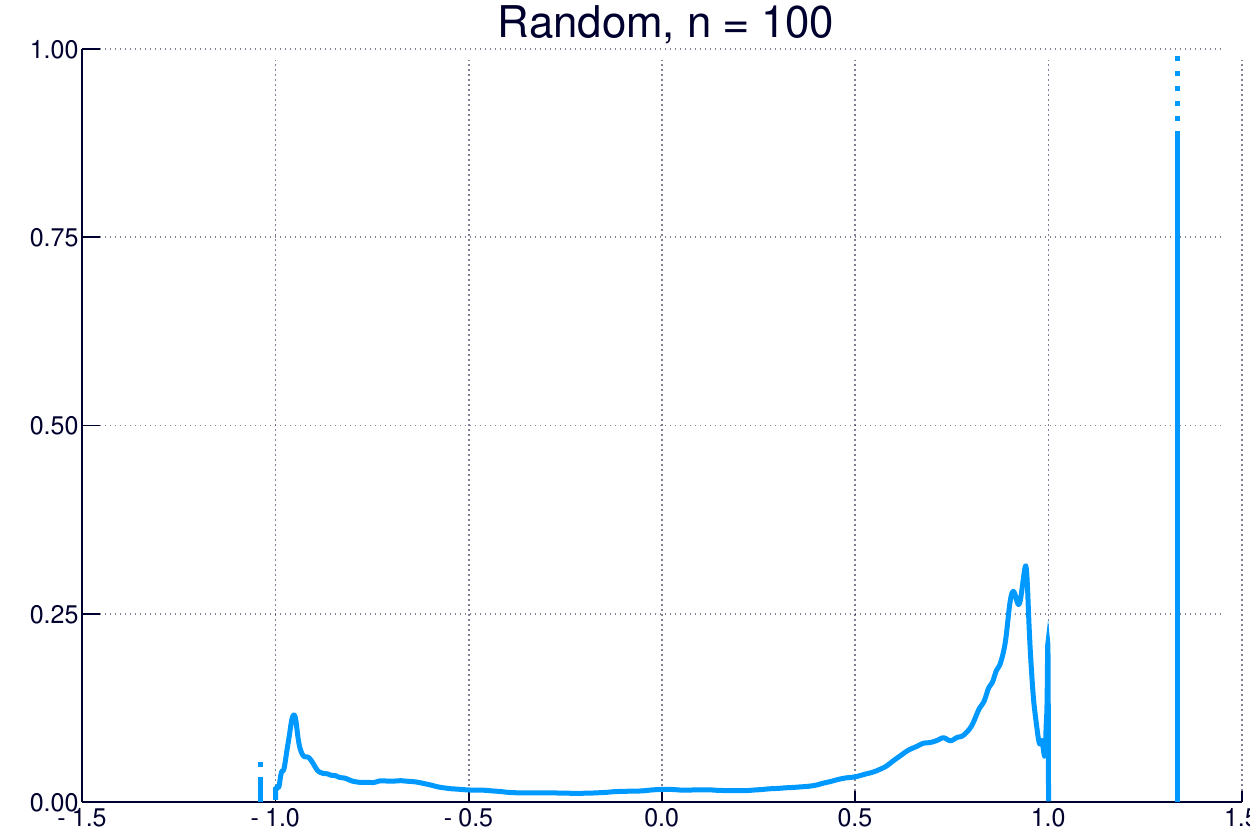}\hskip8pt\includegraphics[height=\figheight]{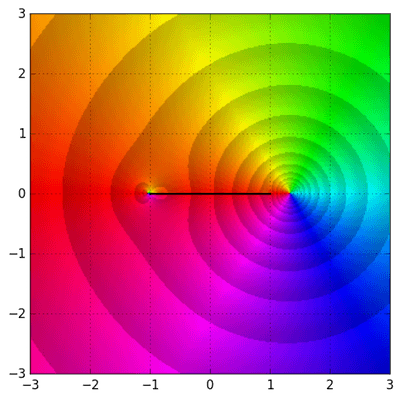}\hskip8pt\includegraphics[height=\figheight]{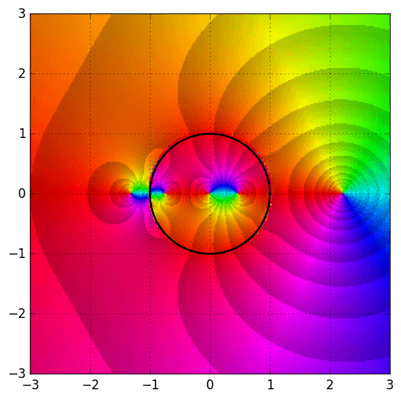}
\caption{These plots are of approximations to the spectral measure and principal resolvents of a trace-class pseudo-random diagonal perturbation of the free Jacobi operator.}\label{fig:random}
 \end{center}
\end{figure}

\end{appendix}